\documentclass[11pt]{amsart}

\usepackage{amsmath,amssymb}
\usepackage{accents}
\usepackage{graphicx}
\usepackage{verbatim}

\newlength{\dhatheight}

\newtheorem{theorem}{Theorem}[section]

\newtheorem{proposition}[theorem]{Proposition}
\newtheorem{lemma}[theorem]{Lemma}
\newtheorem{remark}[theorem]{Remark}
\newtheorem{corollary}[theorem]{Corollary}
\newtheorem{definition}[theorem]{Definition}
\newtheorem{notation}[theorem]{Notation}

\newcommand{\N}{{\mathbb N}}
\newcommand{\R}{{\mathbb R}}
\newcommand{\Z}{{\mathbb Z}}

\newcommand{\ep}{\epsilon}
\newcommand{\dd}{\Delta}
\newcommand{\ra}{\rightarrow}
\newcommand{\ras}{{\stackrel{~~*}{\ra}}}
\newcommand{\ga}{\Gamma}
\newcommand{\xx}{X^1}
\newcommand{\pp}{{\mathcal{P}}}
\newcommand{\bo}{{\partial}}
\newcommand{\td}{\widetilde d}

\newcommand{\tdd}{\widetilde \dd}
\newcommand{\nn}{\N[\frac{1}{4}]}
\newcommand{\ps}{\Psi}
\newcommand{\ph}{\Phi}
\newcommand{\tht}{\Theta}
\newcommand{\vt}{\vartheta}
\newcommand{\slex}{<_{sl}}

\newcommand{\up}{\Upsilon}
\newcommand{\tp}{\widetilde \phi}
\newcommand{\tth}{\widetilde \theta}
\newcommand{\dmi}{diameter inequality}
\newcommand{\dms}{diameter inequalities}
\newcommand{\dmc}{diameter}
\newcommand{\cfl}{combed filling}
\newcommand{\tfi}{tame filling inequality}
\newcommand{\tfs}{tame filling inequalities}
\newcommand{\tci}{tame combing inequality}
\newcommand{\dia}{diagrammatic}
\newcommand{\nff}{{normal filling}}
\newcommand{\cnf}{{combed normal filling}}
\newcommand{\cni}{{tame normal inequality}}
\newcommand{\ehy}{{edge homotopy}}
\newcommand{\ehs}{{edge homotopies}}
\newcommand{\edg}{{normal form diagram}}
\newcommand{\cc}{{\mathcal{N}}}
\newcommand{\cld}{{\mathcal{D}}}
\newcommand{\cle}{{\mathcal{E}}}

\newcommand\ct{{T}}
\newcommand\hc{{f}}
\newcommand\he{{\hat e}}

\newcommand\hhe{{\widetilde e}}
\newcommand\hhr{{\widetilde r}}

\newcommand\hhp{{\widetilde p}}

\newcommand{\stkg}{{stacking}}
\newcommand{\stkbl}{{stackable}}
\newcommand{\fstkg}{{stacking}}
\newcommand{\fstkbl}{{stackable}}
\newcommand{\wstkbl}{{weakly stackable}}
\newcommand{\astkg}{{algorithmic stacking}}
\newcommand{\astkbl}{{algorithmically stackable}}

\newcommand{\afstkbl}{{algorithmically stackable}}
\newcommand{\awstkbl}{{algorithmically weakly stackable}}
\newcommand{\rstkbl}{{regularly stackable}}
\newcommand{\rstkg}{{regular stacking}}
\newcommand{\rwstkbl}{{regularly weakly stackable}}

\newcommand{\prs}{{\stackrel{~~p*}{\ra}}}
\newcommand{\mui}{\mu^i}
\newcommand{\mue}{\mu^e}
\newcommand{\tj}{\widetilde j}
\newcommand{\dire}{{\vec E(X)}}  
\newcommand\ves{{\vec E_{r}}}  
\newcommand\ugd{{E_d}}  
\newcommand\dgd{{\vec E_d}}  

\newcommand{\rnf}{{recursive normal filling}}
\newcommand{\rf}{{recursive  filling}}
\newcommand{\rcnf}{{recursive combed normal filling}}
\newcommand{\rcf}{{recursive combed  filling}}
\newcommand{\edi}{{ediam}}
\newcommand{\idi}{{idiam}}
\newcommand{\maxr}{\zeta}

\newcommand{\kti}{k_\cc^i}
\newcommand{\kxi}{k_r^i}
\newcommand{\kte}{k_\cc^e}
\newcommand{\kxe}{k_r^e}
\newcommand{\hr}{c(\ves)}
\newcommand{\wa}{e_{w,a}}
\newcommand{\xa}{e_{x,a}}
\newcommand{\alg}{S_c}
\newcommand{\tc}{\tilde c}

\newcommand{\ega}{e_{g,a}}

\begin{document}
\title[Stackable groups and tame filling invariants]
{Stackable groups, tame filling invariants, and algorithmic properties of groups}

\author[M.~Brittenham]{Mark Brittenham}
\address{Department of Mathematics\\
        University of Nebraska\\
         Lincoln NE 68588-0130, USA}
\email{mbrittenham2@math.unl.edu}

\author[S.~Hermiller]{Susan Hermiller}
\address{Department of Mathematics\\
        University of Nebraska\\
         Lincoln NE 68588-0130, USA}
\email{smh@math.unl.edu}

\begin{abstract}
We introduce a combinatorial property for 
finitely generated groups
called \stkbl\ that
implies the existence of 
an inductive procedure for constructing
van Kampen diagrams with respect to a
canonical finite presentation. 
We also define \astkbl\ groups, for which
this procedure is an effective algorithm.
This property gives a common
model for algorithms arising from both rewriting
systems and almost convexity for groups.

We also introduce a new pair of asymptotic invariants that are
filling inequalities refining the
notions of intrinsic and extrinsic diameter inequalities for
finitely presented groups.  These \tfs\ 
are quasi-isometry invariants, up to
Lipschitz equivalence of functions (and, in the case of
the intrinsic \tfi, up to choice of a sufficiently large
set of defining relators).  
We show that the radial 
tameness functions of~\cite{hmeiermeastame} are 
equivalent to the extrinsic \tfi\ condition,
and so intrinsic \tfs\ can be viewed as the
intrinsic analog of radial tameness functions.

We discuss both intrinsic and extrinsic
\tfs\ for many examples of \stkbl\ groups,
including groups with a finite complete
rewriting system,
Thompson's group $F$, Baumslag-Solitar groups
and their iterates, and almost convex groups. 
We show that the fundamental group
of any closed 3-manifold with a uniform
geometry is algorithmically stackable
using a regular language of normal forms.
\end{abstract}

\maketitle

\tableofcontents


\section{Introduction and definitions}\label{sec:intro}



\subsection{Overview} \label{subsec:overview}


In geometric group theory, several properties of finitely
generated groups have been defined using a language of
normal forms together with geometric or combinatorial conditions
on the associated Cayley graph, most notably in the concepts of
combable groups and automatic groups in which the normal forms
satisfy a fellow traveler property. 
In this paper,  
in Subsection~\ref{subsec:stackdef}, we use 
a set of normal forms together with another
combinatorial property
on the Cayley graph of a finitely generated group
to define a property 
which we call \stkbl.  
We show that for any \stkbl\ group, these combinatorial
properties yield a 
finite presentation for the group (in Lemma~\ref{lem:prefixclosed}) and
an inductive procedure which, upon input of a word 
in the generators that represents the identity of the group, 
constructs a van Kampen
diagram for that word over this presentation. 
We also define a notion of \astkbl, which
guarantees that this procedure is an
effective algorithm, and a
notion of \rstkbl, in which the \astkbl\ 
structure utilizes a regular 
language of normal forms.
The structure of these van Kampen diagrams for \stkbl\ groups 
differs from the canonical diagrams arising in combable groups,
and the \stkbl\ property allows a wider spectrum of 
filling invariant functions (discussed below) and hence
applies to wider classes of groups.

\smallskip

\noindent{\bf Propositions~\ref{prop:solvwp} and~\ref{prop:fillalg}.} 
{\em 
If $G$ is \astkbl\ over the finite 
generating set $A$,
then $G$ has solvable word problem, and
there is an algorithm which, upon input of a word $w \in A^*$
that represents the identity in $G$, will construct a van 
Kampen diagram for $w$ over the \stkg\ presentation.
}

\smallskip

This \stkbl\ property provides a uniform
model for the canonical procedures for
building van Kampen diagrams that arise
in both the example of groups with a finite complete
rewriting system and the example of almost convex groups,
as we show in Sections~\ref{sec:rs} and~\ref{subsec:ac}.  
The \stkbl\ property for a group $G$ also enables computing
asymptotic filling invariants for $G$ 
by an inductive method; we give an illustration
of this in Section~\ref{sec:rnf}.

Many asymptotic
invariants associated to any
group $G$ with a finite presentation
$\pp=\langle A \mid R\rangle$ have been defined
using properties of van Kampen diagrams over
this presentation.  Collectively, these are
referred to as filling invariants; an
exposition of many of these is given by
Riley in~\cite[Chapter II]{riley}.  
One of the most well-studied
filling functions is the isodiametric, or intrinsic
diameter, function for $G$.  
The adjective ``intrinsic''
refers to the fact that the distances are 
measured in the van Kampen diagram 
It is natural to consider
the distance in the Cayley graph, instead, giving
an ``extrinsic'' property, and
in~\cite{bridsonriley}, 
Bridson and Riley defined and studied 
properties of extrinsic
diameter functions.
It is often easier to compute upper bounds for these
functions, rather than compute exact values;
a group satisfies a \dmi\ for
a function $f$ if $f$ is an upper bound for 
the respective diameter function.

In this paper we also introduce refinements of 
the notions of \dms, called \tfs.
Essentially, intrinsic and
extrinsic \dms\ measure the 
height of the highest peak in van Kampen diagrams,
where height refers to (intrinsic or extrinsic)
distance from the basepoint of the diagram, 
while intrinsic and extrinsic \tfs\ give
a finer measure of the ``hilliness'' of these
diagrams.  
In order to accomplish this, we consider not only
van Kampen diagrams, but also homotopies
that ``comb'' these diagrams; a collection of
van Kampen diagrams and homotopies for
all of the words representing the identity of
the group is a \cfl.
In the last part of this Introduction,
in Subsection~\ref{subsec:tfdef}, we give the details
of the definitions of {\cfl}s and \dmc\ and \tfs.  

In Section~\ref{sec:relationships}, 
we show in Proposition~\ref{prop:itimpliesid}
that an intrinsic or
extrinsic \tfi\ with respect to a function $f$ 
implies an intrinsic or extrinsic (respectively)
\dmi\ for the function $n \mapsto \lceil f(n) \rceil$.
In \cite{bridsonriley}, Bridson and Riley give an example
of a finitely presented group $G$ whose (minimal) intrinsic and
extrinsic diameter functions are not Lipschitz
equivalent.  (Two functions $f,g: \N \ra \N$ are {\em Lipschitz equivalent}
if there is a constant $C$ such that for all $n \in N$ 
we have both $f(n) \le Cg(Cn+C)+C$ and $g(n) \le Cf(Cn+C)+C$.)
While we have not resolved the relationship
between \tfs\  in general, we give bounds
on their interconnections in Theorem~\ref{thm:itversuset}. 

\smallskip

\noindent{\bf Theorem~\ref{thm:itversuset}.} {\em
Let $G$ be a finitely presented group with Cayley complex $X$ 
and \cfl\ $\cld$.
Suppose that $j:\N \ra \N$ is a nondecreasing function such that
for every vertex $v$ of a van Kampen diagram $\dd$ in $\cld$,
$d_{\dd}(*,v) \le j(d_X(\ep,\pi_{\dd}(v)))$, and
let $\tj:\nn \ra \nn$ be
defined by $\tj(n):=j(\lceil n \rceil)+1$. 
\begin{enumerate}
\item If $G$ satisfies an extrinsic \tfi\ for the function 
$f$ with respect to $\cld$, then $G$ 
satisfies an intrinsic \tfi\ for the function $\tj \circ f$.
\item If $G$ satisfies an intrinsic \tfi\ for the function 
$f$ with respect to $\cld$, then $G$ 
satisfies an extrinsic \tfi\ for the function $f \circ \tj$.
\end{enumerate}
}

Our definition of {\cfl} was also motivated by
the concept of tame combing defined by
Mihalik and Tschantz~\cite{mihaliktschantz}, 
and associated radial tame combing functions
advanced by Hermiller and Meier~\cite{hmeiermeastame}.
Tame combings are homotopies in the Cayley complex;
in Section~\ref{sec:relax}, we recast them
into the setting of van Kampen diagrams
in the following portion of Proposition~\ref{prop:etistc}.

\smallskip

\noindent{\bf Proposition~\ref{prop:etistc}'.} {\em
Let $G$ be a group with 
a finite symmetrized presentation $\pp$.
Up to Lipschitz equivalence of nondecreasing functions,
the pair $(G,\pp)$ satisfies an extrinsic 
\tfi\ for a function
$f$ if and only if $(G,\pp)$ satisfies a radial \tci\ 
with respect to $f$.
}

\smallskip

\noindent Effectively, Proposition~\ref{prop:etistc}' shows that 
the intrinsic \tfi\ is the
intrinsic analog of the radial \tci.

Every group admitting a radial \tci\ for a finite-valued function,
and hence, by Proposition~\ref{prop:etistc}', every group admitting a
finite-valued extrinsic \tfi\
function, must also be tame combable
as defined in~\cite{mihaliktschantz}.  
Although every group admits 
extrinsic 
diameter inequalities for finite-valued functions,
it is not yet clear whether every finitely presented group admits
\tfs\ for such functions.  
Tschantz~\cite{tschantz} has conjectured
that there is a finitely presented group that does not
admit a tame combing, and as a result, that there exists
a finitely presented group which admits an extrinsic diameter inequality
for a finite-valued function $f$, but which does not
satisfy an extrinsic \tfi\ for 
any finite-valued function, and in particular does not satisfy
an extrinsic \tfi\ for
any function Lipschitz equivalent to $f$.

Section~\ref{sec:relax} also contains Definitions~\ref{def:nfi}
and~\ref{def:rtfi} of two more asymptotic
invariants which we prove to be equivalent to \tfs\ in
Propositions~\ref{prop:htpydomain} and~\ref{prop:etistc}.
These alternative views are applied in later sections.
A fundamental difference between the intrinsic and
extrinsic cases arises in this section, in that 
iterative constructions that glue van Kampen diagrams
preserve extrinsic distances, but not necessarily
intrinsic distances.

In Section~\ref{sec:rnf}, we show that 
\stkbl\ groups admit a stronger inductive 
procedure, which produces a combed van Kampen
diagram for any input word.  That is, there is a
canonical combed filling associated to a
\stkbl\ group.
The inductive nature of
this associated combed filling yields the following.

\smallskip

\noindent{\bf Theorem~\ref{thm:rnftame}'.}  {\em
If $G$ is a \fstkbl\ group,  
then $G$ admits intrinsic and extrinsic \tfs\ for 
finite-valued functions.
}

\smallskip

\noindent{\bf Theorem~\ref{thm:astktf}.}  {\em
If $G$ is an \afstkbl\ group, then $G$ satisfies
both intrinsic and extrinsic \tfs\ with respect to
a recursive function.
}

\smallskip

An immediate consequence of Theorem~\ref{thm:rnftame}
and Proposition~\ref{prop:etistc} is that every \fstkbl\ group
satisfies the quasi-isometry invariant property of
having a tame combing, developed by Mihalik
and Tschantz~\cite{mihaliktschantz}.  
If Tschantz's
conjecture~\cite{tschantz} that a non-tame-combable
group exists is true, this would then
also imply that there is a finitely presented group
that does not admit the \fstkbl\ property with
respect to any finite generating set.

In Section~\ref{sec:examples} we discuss several
examples of (classes of)
\fstkbl\ groups, and compute bounds on their
tame filling invariants.
To begin, in Section~\ref{sec:rs} we consider groups
that can be presented by rewriting systems.  
A {\em finite complete rewriting system}
for a group $G$ consists of a finite set $A$
and a finite set of rules $R \subseteq A^* \times A^*$
such that 
as a monoid, $G$ is presented by 
$G = Mon\langle A \mid u=v$ whenever $(u,v) \in R \rangle,$
and the rewritings
$xuy \ra xvy$ for all $x,y \in A^*$ and $(u,v)$ in $R$ satisfy:
(1) each $g \in G$ is 
represented by exactly one word over $A$ that cannot be rewritten,
and
(2) the (strict) partial ordering
$x>y$ if $x \rightarrow
x_1 \rightarrow ... \rightarrow x_n \rightarrow y$ is 
well-founded.
The {\em length} of a rewriting rule $u \ra v$ in $R$
is the sum of the lengths of the words $u$ and $v$.
The {\em string growth complexity} function 
$\gamma:\N \ra \N$ associated to this system is defined by
$\gamma(n)$ = the maximal length of a word that is
a rewriting of a word of length $\le n$.
We use the algorithm of Section~\ref{sec:rnf}
to obtain \tfs\ in terms of $\gamma$ in this case.

\smallskip

\noindent{\bf Theorem~\ref{thm:crsrecit}} and 
{\bf Corollary~\ref{cor:rsgrowth}.} {\em
Let $G$ be a group with a finite complete
rewriting system.  Let
$\gamma$ be the
string growth complexity function 
for the associated
minimal system and let $\maxr$ denote the length of 
the longest rewriting rule for this system.
Then $G$ is \rstkbl\ and satisfies 
both intrinsic and extrinsic \tfs\ 
for the recursive
function $n \mapsto \gamma(\lceil n \rceil +\maxr+2)+1$.
}

\smallskip

This result has 
potential 
to reduce the amount of work in 
searching for
finite complete rewriting systems for groups.
A choice of partial ordering used in (2) above implies an
upper bound on the string growth complexity function.
Then given a lower bound on the intrinsic or extrinsic
\tfs\ or \dms, this corollary
can be used to eliminate partial orderings before
attempting to use them (e.g., via the
Knuth-Bendix algorithm) to construct a rewriting system.
A further paper by the present authors will 
address this application more fully.

In Section~\ref{subsec:f}, we consider Thompson's group $F$;
i.e., the group of orientation-preserving piecewise linear
automorphisms of the unit interval for which all linear
slopes are powers of 2, and all breakpoints lie in the
the 2-adic numbers. 
Thompson's group $F$ has been the focus of considerable
research in recent years, and yet the questions of
whether $F$ is automatic or has a finite complete rewriting system  
are open (see the problem list at~\cite{thompsonpbms}).
In~\cite{chst}, Cleary, Hermiller, Stein, and Taback
show that Thompson's group 
$F$ is \fstkbl\ (and their proof can be shown to give an \astkg),
and we note in Section~\ref{subsec:f} that the 
set of normal forms associated to this \stkg\ is a 
deterministic context-free language.
In~\cite{chst} the authors also show (after combining their result
with Proposition~\ref{prop:etistc}) that $F$ admits a
linear extrinsic \tfi.  In Section~\ref{subsec:f}
we show that this group also 
admits a linear intrinsic \tfi, thus refining
the result of Guba~\cite[Corollary 1]{guba} that $F$ has a linear
intrinsic diameter function.

In the next two subsections of Section~\ref{sec:examples},
we discuss two specific examples of classes of groups admitting finite
complete rewriting systems in more detail.
We show in Section~\ref{subsec:bs} that the Baumslag-Solitar group 
$BS(1,p)$ with $p \ge 3$ admits an intrinsic \tfi\ Lipschitz
equivalent to the exponential function $n \mapsto p^n$, 
utilizing the linear extrinsic \tfi\ for these groups
shown in~\cite{chst}.
We note in Section~\ref{subsec:iteratedbs}
that the iterated Baumslag-Solitar groups $G_k$
are examples of \rstkbl\ groups admitting recursive intrinsic
and extrinsic \tfs.  However, applying the lower bound of
Gersten~\cite{gerstenexpid} on their intrinsic diameter functions,
for each natural number $k>2$
the group $G_k$ does not admit intrinsic or extrinsic
\tfs\ with respect to
a $k-2$-fold tower of exponentials.

Building upon the characterization
of Cannon's almost convexity property~\cite{cannon}
by a radial \tci\ 
in~\cite{hmeiermeastame}, 
in Section~\ref{subsec:ac} we show the following.

\smallskip

\noindent{\bf Theorem~\ref{thm:aceti}.}  {\em
Let $G$ be a group with finite generating set $A$, and
let $\iota: \nn \ra \nn$ denote the identity
function.  The following
are equivalent:
\begin{enumerate}
\item The pair $(G,A)$ is almost convex 
\item There is a finite
presentation $\pp=\langle A \mid R \rangle$ for $G$ that
satisfies an intrinsic \tfi\ with respect to $\iota$.
\item There is a finite
presentation $\pp=\langle A \mid R \rangle$ for $G$ that
satisfies an extrinsic \tfi\ with respect to $\iota$.
\end{enumerate}
Moreover, if any of these hold, then $G$ is \afstkbl\ over $A$.
}

\smallskip

\noindent The properties in Theorem~\ref{thm:aceti}
are satisfied by all
word hyperbolic groups and  
cocompact discrete groups of isometries of Euclidean
space, with respect to every generating set~\cite{cannon}.
They are also satisfied by any group $G$ that is shortlex
automatic with
respect to the generating set $A$ (again this includes
all word hyperbolic groups~\cite[Thms~3.4.5,2.5.1]{echlpt});
for these groups, the set of shortlex normal forms is
a regular language.
In the proof of Theorem~\ref{thm:aceti}, the 
stackable structure constructed for almost convex groups also
utilizes the shortlex normal forms. Hence every
shortlex automatic group, including every
word hyperbolic group, is \rstkbl.

One of the motivations for the definition of automatic 
groups was to understand the computational properties of
fundamental groups of 3-manifolds.  However,
the fundamental group of a 3-manifold 
is automatic if and only if its JSJ decomposition does not 
contain manifolds with a uniform Nil or Sol 
geometry~\cite[Theorem~12.4.7]{echlpt}.
In contrast,~\cite{hs} Hermiller and Shapiro have shown that
the fundamental group of every closed 3-manifold with a
uniform geometry other than hyperbolic must have a
finite complete rewriting system, and so combining this result
with Theorems~\ref{thm:crsrecit} and~\ref{thm:aceti}
yields the following.

\smallskip

\noindent{\bf Corollary~\ref{cor:3mfd}.} {\em
If $G$ is the fundamental group of a closed 3-manifold with
a uniform geometry, then $G$ is \rstkbl. 
} 

\smallskip 

In~\cite{kkm}, Kharlampovich, Khoussainov, and
Miasnikov introduced the concept of Cayley graph automatic groups,
which utilize a fellow-traveling regular set of ``normal forms''
in which the alphabet for the normal form words is not necessarily
a generating set, resulting in a class of groups
which includes all automatic groups but also includes many
nilpotent and solvable nonautomatic groups.  An
interesting open question to ask, then, is what relationships,
if any, exist between the classes of \stkbl\ groups and Cayley
graph automatic groups.

In Section~\ref{sec:combable},
we consider tame filling invariants for
a class of combable groups.

\smallskip

\noindent{\bf Corollary~\ref{cor:combable}.}  {\em 
If a finitely generated group $G$ admits a 
quasi-geodesic language of simple word normal forms
satisfying a $K$-fellow traveler property, then $G$ satisfies linear
intrinsic and extrinsic \tfs.
}

\smallskip

\noindent In particular,
all automatic groups
over a prefix-closed language of normal forms
satisfy the hypotheses of Corollary~\ref{cor:combable}.
This result refines that of Gersten~\cite{gersten},
that combable groups have a linear intrinsic
diameter function. 

Finally, in Section~\ref{sec:qiinv}, we prove that \tfs\ are
quasi-isometry invariants, in the following.

\smallskip

\noindent{\bf Theorem~\ref{thm:itisqi}.}  {\em
Suppose that $(G,\pp)$ and $(H,\pp')$ 
are quasi-isometric groups with
finite presentations.
If $(G,\pp)$ satisfies an extrinsic 
\tfi\  with respect to $f$, then $(H,\pp')$
satisfies an extrinsic
\tfi\  with respect to a function
that is Lipschitz equivalent to $f$.
If $(G,\pp)$ satisfies an intrinsic 
\tfi\  with respect to $f$, then after adding
all relators of length up to a sufficiently 
large constant to the 
presentation $\pp'$, 
the pair $(H,\pp')$
satisfies an intrinsic
\tfi\  with respect to a function
that is Lipschitz equivalent to $f$.
}

\eject


\subsection{Notation} \label{subsec:notation}


$~$

\vspace{.1in}

Throughout this paper, let $G$ be a group
with a finite {\em symmetric} generating set; that
is, such that the generating set $A$ is closed under inversion.
We will also assume that for each $a \in A$,
the element of $G$ represented by $a$ is not 
the identity $\ep$ of $G$.
For a word $w \in A^*$, we write $w^{-1}$ for the 
formal inverse of $w$ in $A^*$. 
For words $v,w \in A^*$, we write $v=w$ if $v$
and $w$ are the same word in $A^*$, and write $v=_G w$ if
$v$ and $w$ represent the same element of $G$.

The group $G$ also has a presentation 
$\pp = \langle A \mid R \rangle$
that is {\em symmetrized}; that is, 
such that the generating set $A$ is symmetric,
and the set $R$ of defining relations is closed under
inversion and cyclic conjugation.
Let $X$ be the Cayley 2-complex corresponding to this presentation,
whose 1-skeleton $X^1=\ga$ is the Cayley graph of $G$ with
respect to $A$.
Let $E(X)=E(\ga)$
be the set of 1-cells (i.e., undirected edges) 
in $X^1$.
By usual convention, for all $g \in G$ and $a \in A$,
we consider both the directed edge labeled 
labeled $a$ from the vertex $g$ to $ga$ and the
 directed edge labeled $a^{-1}$ from $ga$ to $g$
to have the same underlying undirected CW complex
edge between the vertices labeled $g$ and $ga$.
Let $\dire=\vec E(\ga)$ be the set of these directed 
edges of $X^1$.

For an arbitrary word $w$ in $A^*$
that represents the
trivial element $\ep$ of $G$, there is a {\em van Kampen
diagram} $\dd$ for $w$ with respect to $\pp$.  
That is, $\dd$ is a finite,
planar, contractible combinatorial 2-complex with 
edges directed and
labeled by elements of $A$, satisfying the
properties that the boundary of 
$\dd$ is an edge path labeled by the
word $w$ starting at a basepoint 
vertex $*$ and
reading counterclockwise, and every 2-cell in $\dd$
has boundary labeled by an element of $R$.

Note that although the definition in the previous
paragraph is standard, it is a slight abuse of notation,
in that the 2-cells of a van Kampen diagram
are polygons whose boundaries are labeled by words in $A^*$,
rather than elements of a (free) group.  We will
also consider the set $R$ of defining relators
as a finite subset of $A^* \setminus \{1\}$, where 1 is
the empty word.
We do not assume that every defining relator
is freely reduced, but the freely reduced representative
of every defining relator, except 1, must also be in $R$.

In general, there may be many different van 
Kampen diagrams for the word $w$.  Also,
we do not assume that van Kampen diagrams
in this paper are reduced; that is, we allow adjacent
2-cells in $\dd$ to be labeled by the same relator with
opposite orientations.  

For any van Kampen diagram
$\dd$ with basepoint $*$, let $\pi_\dd:\dd \ra X$
denote the canonical cellular map such that $\pi_\dd(*)=\ep$ and
$\pi_\dd$ maps edges to edges preserving both
label and direction.

See for example~\cite{bridson} or~\cite{lyndonschupp} 
for an exposition of the theory of van Kampen diagrams.


\subsection{Stackable groups: Definitions and 
  motivation} \label{subsec:stackdef}


$~$

\vspace{.1in}

Our goal is to define a class of groups for which there
is an inductive procedure to construct van Kampen diagrams
for all words representing the trivial element. 
Such a collection $\{\dd_w \mid w \in A^*, w=_G\ep\}$ of
van Kampen diagrams for a group $G$ over a presentation 
$\pp=\langle A \mid R \rangle$
(where for each $w$, the
diagram $\dd_w$ has boundary label $w$)
is called a {\em filling} for the pair $(G,\pp)$.

We begin with a group $G$ together with a finite
inverse-closed generating set $A$ of $G$.  
Let $\ga$ be the associated
Cayley graph.  For each $g \in G$ and $a \in A$,
let $\ega$ denote the directed edge in $\vec E(\ga)$
with initial vertex $g$, terminal vertex $ga$,
and label $a$.  Whenever $x \in A^*$ and $a \in A$,
we also write $\xa:=\ega$, where $g$ is the
element of $G$ represented by $x$.

For any set $\cc  = \{y_g \mid g \in G\}$
of normal forms for $G$ (where $y_g \in A^*$ represents
the element $g \in G$),
we define the set
\[
\dgd=\vec E_{d,\cc} := \{ \ega \mid \mbox{either }y_ga=y_{ga} \mbox{ or } y_g=y_{ga}a^{-1}\}
\subseteq \vec E(\ga)
\]
of {\em degenerate edges}, 
and let $\ugd = E_{d,\cc} \subseteq E(\ga)$ be the set
of undirected edges underlying  the directed edges
in $\dgd$.   
The complementary set
\[
\ves =\vec E_{r,\cc}:= \vec E(\ga) \setminus \dgd
\]
will be called the set of {\em recursive edges}.

\begin{definition}\label{def:fstkbl}
A group $G$ is {\em \stkbl} with respect to a finite symmetric
generating set $A$ if there exist a set $\cc$ of normal forms
for $G$ over $A$ with normal form $1$ for $\ep$, a well-founded 
strict partial ordering $<$ on the set $\ves$ of recursive
edges,
and a constant $k$, such that
whenever $g \in G$, $a \in A$, and $\ega \in \ves$, then
there exists a directed path from $g$ to $ga$ in
$\ga$ labeled by a word $a_1 \cdots a_n \in A^n$
of length $n \le k$ satisfying the property that 
for each $1 \le i \le n$, either
$e_{ga_1 \cdots a_{i-1},a_i} \in \ves$ and 
$e_{ga_1 \cdots a_{i-1},a_i} < \ega$, or else
$e_{ga_1 \cdots a_{i-1},a_i} \in \dgd$.
\end{definition}

We call a group $G$ {\em \stkbl} if there is a finite
generating set for $G$ with respect to which
the group is \stkbl.

Given a group $G$ that is \stkbl\  with respect to 
a generating set $A$, one can define a function 
$c:\ves \ra A^*$ by choosing, for 
each $\ega \in \ves$, a label $c(\ega) =a_1 \cdots a_n \in A^*$
of a directed path in $\ga$ satisfying the property above; that is,
$c(\ega) =_G a$,
$n \le k$, and 
either 
$e_{ga_1 \cdots a_{i-1},a_i} < \ega$ or
$e_{ga_1 \cdots a_{i-1},a_i} \in \dgd$ for each $i$.
The image $c(\ves) \subseteq \cup_{n=0}^k A^n$ is a finite
set.
This function will be called a {\em \stkg\ map}.

On the other hand, given any set $\cc$ of normal forms
for $G$ over $A$ and any function $c:\ves \ra A^*$,
we can define a a relation $<_c$ on $\ves$
as follows.  Whenever $e',e$ are both in $\ves$ and $e'$ 
lies in
the path in $\ga$ that starts at the initial vertex of $e$
and is labeled by $c(e)$
(where $e'$ is oriented 
in the same direction as this path), 
write $e' <_c e$.  Let
$<_c$ be the transitive closure of this relation.
If $c$ is a \stkg\ map obtained from a \stkbl\ structure $(\cc,<,k)$
for $G$ over $A$, then the relation $<_c$ is a subset of the
well-founded strict partial ordering $<$,
and so $<_c$ is also a well-founded strict partial ordering.
Moreover, by K\"onig's Infinity Lemma, $<_c$ satisfies
the property that for each $e \in \ves$, there are
only finitely many $e'' \in \ves$ with $e'' <_c e$.

\begin{definition}\label{def:fstkg}
A {\em \stkg} for a group $G$ with respect to a finite 
symmetric generating set $A$ is a pair $(\cc,c)$
where
$\cc$ is a set 
of normal forms for $G$ over $A$ such that 
the normal form of the identity is the empty word
and $c:\ves \ra A^*$ is a function satisfying
\begin{description}
\item[(S1)] For each $\ega \in \ves$ we have
$c(\ega)=_G a$.
\item[(S2)] The relation $<_c$ on $\ves$ is a strict partial
ordering satisfying the property that for each $e \in \ves$,
there are only finitely many $e'' \in \ves$
with $e'' <_c e$.
\item[(S3)] The image 
$c(\ves)=\{c(e) \mid e \in \ves\} \subset A^*$ 
is a finite set.
\end{description}
\end{definition}

\noindent (Note that Property (S2) implies that $c(\ega) \neq_G a$ 
for all $\ega \in \ves$.)

From the discussion above, the following is immediate.

\begin{lemma}\label{lem:stkbldefs}
A group $G$ is stackable with respect to a finite symmetric 
generating set $A$
if and only if $G$ admits a stacking with respect to $A$.
\end{lemma}

For a \stkg\ $(\cc,c)$, let $R_c$ be the closure of the
set $\{c(\ega)a^{-1} \mid \ega \in \ves\}$
under inversion, cyclic conjugation, and free reduction.

\begin{notation}\label{not:stksets}
For a \fstkg\  $(\cc,c)$, the function
$c$ is the {\em \stkg\ map}, 
the set $\hr$ is the {\em \stkg\ image},
the set $R_c$ is the {\em \stkg\ relation set},
and $<_c$ is the {\em \stkg\ ordering}.
\end{notation}

In essence, the two equivalent definitions of stackability
in Lemma~\ref{lem:stkbldefs} are written to display 
connections to two other properties:  The
form of Definition~\ref{def:fstkbl} closely follows
 Definition~\ref{def:ac} of
almost convexity, and 
the \stkg\ map in Definition~\ref{def:fstkg}
gives rise to rewriting operations, which we discuss next.

Starting from a \fstkg\  $(\cc,c)$ for
a group $G$ with generators $A$, we 
describe a {\em \stkg\ reduction procedure} for 
finding the normal form
for the group element associated to any word, 
by defining a rewriting operation on words over $A$, as follows. 
Whenever a word $w \in A^*$ has a decomposition 
$w=xay$ such that $x,y \in A^*$, $a \in A$, and the directed
edge $\xa$ of $\ga$ 
lies in $\ves$,
then we rewrite $w \ra x c(\xa) y$.
The definition of the
\stkg\  ordering $<_c$ says that for every directed edge $e'$ 
in the Cayley graph $\ga$ that lies along the path
labeled $c(\xa)$ from the vertex labeled $x$,
either $e'$ is
a degenerate edge in $\dgd$, or else $e' \in \ves$ and $e'<_c e$.
Then Property (S2) shows that starting from the word
$w$, there can be at most finitely many rewritings 
$w \ra w_1 \ra \cdots \ra w_m=z$ 
until a word $z$ is obtained
which cannot be rewritten with this procedure.  
The final step of the \stkg\ reduction procedure is to freely reduce 
the word $z$, resulting in a word $w'$.

Now $w =_G w'$, and the word $w'$ (when input into 
this procedure) is not
rewritten with the \stkg\ reduction procedure.  
Write $w'=a_1 \cdots a_n$ with
each $a_i \in A$.  Then for all $1 \le i \le n$, the
edge $e_i:=e_{a_1 \cdots a_{i-1},a_i}$ of $\ga$ 
does not lie in $\ves$, and so must be in $\dgd$.  
In the case that $i=1$, this implies that
either $y_\ep a_1 = y_{a_1}$ or $y_{a_1} a_1^{-1} = y_\ep$.
Since the normal form of the identity
is the empty word, i.e.~$y_\ep = 1$, we must have 
$y_{a_1}=a_1$.  Assume inductively that
$y_{a_1 \cdots a_i}=a_1 \cdots a_i$.  The inclusion
$e_{i+1} \in \dgd$ implies that either 
$a_1 \cdots a_i \cdot a_{i+1}=y_{a_1 \cdots a_{i+1}}$
or $y_{a_1 \cdots a_{i+1}} a_{i+1}^{-1} = a_1 \cdots a_{i}$.
However, the latter equality on words would imply
that the final letter $a_{i+1}^{-1}=a_i$, which contradicts
the fact that $w'$ is freely reduced.  Hence we have
that $w'=y_{w'}=y_w$ is in normal form, and moreover every
prefix of $w'$ is also in normal form.

That is, we have shown the following.

\begin{lemma}\label{lem:prefixclosed}
Let $G$ be a group with generating set $A$ and
let $(\cc,c)$ be a \fstkg\ for $(G,A)$.
If $R_c$ is the associated \stkg\ relation set, then
$\langle A \mid R_c\rangle$ is a finite presentation for $G$.
Moreover, the set $\cc$ of normal forms of a \fstkg\ is
closed under taking prefixes.
\end{lemma}

We call the presentation $\langle A \mid R_c \rangle$
the {\em \stkg\ presentation}.
In the Cayley 2-complex
$X$ corresponding to this presentation,
for each edge
$e$ labeled $a$ in the set $\ves$, 
the word $c(e)a^{-1}$ is the 
label of the boundary path for a 2-cell in $X$,
that traverses the
reverse of the edge $e$.



Any prefix-closed set $\cc$ of normal
forms for $G$ over $A$
yields a maximal
tree $\ct$ in the Cayley graph $\ga$,
namely the set of edges in the paths 
in $\ga$ 
starting at $\ep$ and labeled by the words in $\cc$. 
The associated set $\dgd$ of degenerate edges 
is exactly
the set $\vec E(\ct)$ of directed
edges lying in this tree, and the edges of
$\ves$ are the edges of $\ga$ that do not lie
in the tree $\ct$. 
Each element $w$
of $\cc$ must be a  
{\em simple word}, meaning that $w$ labels 
a simple path, that does not repeat 
any vertices or edges,
in the Cayley graph.

We note that our \stkg\ reduction procedure for finding normal forms
for words may not be an effective
algorithm.   
To ensure that this process 
is algorithmic, we would need to be able to
recognize, given $x \in A^*$ and $a \in A$,
whether or not $\xa \in \ves$, and if so,
be able to find $c(\xa)$.  
If we extend the map $c$ to 
a function $c':\vec E(\ga) \ra 
A^*$ on
all directed edges in $\ga$, by defining $c'(e):=c(e)$
for all $e \in \ves$ and
$c'(e):=a$  whenever $e \in \dgd$ and $e$ has label $a$, 
then essentially this means that the 
graph of the function $c'$, as described by the subset
$$\alg:=\{(w,a,c'(\wa)) \mid w \in A^*, a \in A\}$$
of $A^* \times A \times 
A^*$, should be computable.
In that case, given any $(w,a) \in A^* \times A$, 
by enumerating the words $z$ in $A^*$ and
checking in turn whether $(w,a,z) \in \alg$, we
can find $c'(w,a)$.
(Note that 
the set $\alg$ is computable if and only if
the set $\{(w,a,c(\wa)) \mid w \in A^*, a \in A, e_{w,a} \in \ves\}$
describing the graph of $c$ is computable.  
However, using the latter set in the
\stkg\ reduction algorithm has the drawback of requiring us to 
enumerate the finite (and hence enumerable) set $\hr$,
but we may not have an algorithm to find this
set from the \stkg.)  

\begin{definition}
A group $G$ is {\em \afstkbl}
if $G$ has a finite symmetric
generating set $A$ with a \fstkg\ 
$(\cc,c)$ such that the set $\alg$ is recursive.
\end{definition}

We have shown the following.

\begin{proposition}\label{prop:solvwp}
If $G$ is algorithmically \fstkbl,
then $G$ has solvable word problem.
\end{proposition}

As with many other algorithmic classes of groups,
it is natural to discuss formal language theoretic
restrictions on the associated languages, and in
particular on the set of normal forms.
Computability of the set $\alg$ implies
that the set $\cc$ is computable as well 
(since any word $a_1 \cdots a_n \in A^*$ lies
in $\cc$ if and only if the word is freely
reduced and for each $1 \le i \le n$
the tuple $(a_1 \cdots a_{i-1},a_i,a_i)$ lies
in $\alg$).  Many of the examples we consider
in Section~\ref{sec:examples} will satisfy
stronger restrictions on the set $\cc$.

\begin{definition}
A group $G$ is {\em \rstkbl}
if $G$ has a finite symmetric
generating set $A$ with a \fstkg\ 
$(\cc,c)$ such that the set $\cc$ is a regular
language and the set $\alg$ is recursive.
\end{definition}

Before discussing the details of the inductive procedure 
for building fillings from {\fstkg}s, we first reduce the 
set of diagrams required.

For a group $G$ with symmetrized
presentation $\pp=\langle A \mid R \rangle$
and a set 
$\cc=\{y_g \mid g \in G\} \subseteq A^*$ 
of normal forms for $G$,
a {\em \edg} is a van Kampen diagram for a word
of the form $y_gay_{ga}^{-1}$ where
$g \in G$ and $a$ in $A$.  We can associate
this {\edg} with the directed edge
of the Cayley complex $X$ labeled by $a$ with
initial vertex labeled by $g$.
A {\em \nff} for the pair $(G,\pp)$ 
consists of a set $\cc$ normal
forms for $G$ that are simple words
(i.e. labeling simple paths in the Cayley complex for $\pp$),
together with a collection
$\{\dd_e \mid e \in E(X)\}$ of {\edg}s, where
for each undirected edge $e$ in $X$, the \edg\ 
$\dd_e$ is associated to one of the two
possible directions of $e$.

Every \nff\ induces a filling, using the ``seashell''
(``cockleshell'' in \cite[Section~1.3]{riley})
method, illustrated in Figure~\ref{fig:seashell}, as follows.
\begin{figure}
\begin{center}
\includegraphics[width=3.6in,height=1.6in]{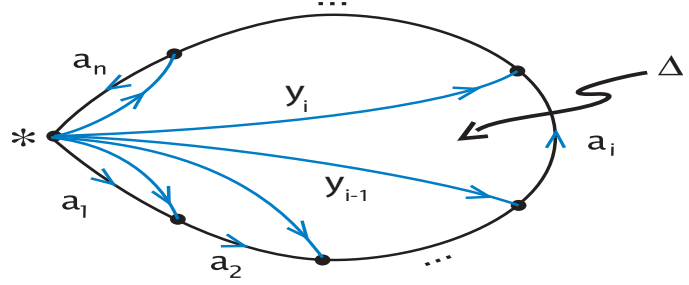}
\caption{Van Kampen diagram built with seashell procedure}\label{fig:seashell}
\end{center}
\end{figure}
Given a word $w=a_1 \cdots a_n$ representing
the identity of $G$, with each $a_i \in A$, then
for each $1 \le i \le n$, there is a \edg\ 
$\dd_i$ in the \nff\ that is associated
to the edge of $X$ with endpoints labeled by the group
elements represented by the words
$a_1 \cdots a_{i-1}$ and $a_1 \cdots a_i$.
Letting $y_i$ denote the normal form in $\cc$ representing 
$a_1 \cdots a_i$,
then the counterclockwise boundary of this diagram
is labeled by either 
$y_{i-1} a_i y_{i}^{-1}$
or $y_{i} a_i^{-1} y_{i-1}^{-1}$;
by replacing $\dd_i$ by its mirror image if necessary, we 
may take $\dd_i$ to have counterclockwise boundary 
word $x_i:=y_{i-1} a_i y_{i}^{-1}$.
We next iteratively build a van Kampen diagram 
$\dd_i'$ for the word 
$y_{\ep}a_1 \cdots a_i y_{i}^{-1}$,
beginning with $\dd_1':=\dd_1$.  For $1<i \le n$, 
the planar diagrams $\dd_{i-1}'$ and $\dd_{i}$
have boundary subpaths
sharing a common label $y_{i}$.
The fact that this word is simple, and so labels
a simple path in $X$, 
implies that any path in a van Kampen
diagram labeled by $y_i$ must also be simple, 
and hence each of these boundary paths
is an embedding.
These paths are also oriented in the same direction,
and so the diagrams $\dd_{i-1}'$ and $\dd_{i}$ can be
glued, starting at their basepoints and
folding along these subpaths,
to construct the 
planar diagram $\dd_i'$.  
Performing these gluings consecutively for each
$i$ results in a van Kampen diagram $\dd_n'$ with
boundary label $y_{\ep}wy_{w}^{-1}$.
Note that we have allowed the
possibility that some of the boundary edges of $\dd_n'$
may not lie on the boundary of a 2-cell in $\dd_n'$;
some of the words $x_i$ may freely
reduce to the empty word, and the corresponding
van Kampen diagrams $\dd_i$ may have no 2-cells.
Note also that
the only simple word representing the identity of $G$
is the empty word; that is, $y_\ep =y_{w}=1$.  
Hence $\dd_n'$ is the
required van Kampen diagram for $w$.

Again starting from a \fstkg\  $(\cc,c)$ for a group
$G$ over a finite generating set $A$, we now 
give an inductive procedure for
constructing a filling for $G$ over the \stkg\ 
presentation $\pp=\langle A \mid R_c \rangle$ as follows.  
Let $X$ be the Cayley graph of this presentation.
From the argument above, an inductive process for 
constructing a \nff\ 
from the \fstkg\  will suffice. 
The set of normal forms
for the \nff\ will  be the set 
$\cc$ from the \stkg.

We will define a \edg\  corresponding to each
directed edge in $\vec E(X)=\ves \cup \dgd$.
Let $e$ be an edge in $\vec E(X)$, oriented
from a vertex $g$ to a vertex $h$ and labeled by $a \in A$,
and let $w_e:=y_g a y_{h}^{-1}$.

In the case that $e$ lies in $\dgd$, the
word $w_e$ freely reduces to the empty
word.  Let $\dd_e$ be the van Kampen diagram for 
$w_e$ consisting of a line segment of edges,
with no 2-cells.

In the case that $e \in \ves$, we
will use Noetherian induction
to construct the \edg.
Write $c(e)=a_1 \cdots a_n$ with each $a_i \in A^*$, and for each
$1 \le i \le n$, let $e_i$
be the edge in $X$ from $ga_1 \cdots a_{i-1}$
to $ga_1 \cdots a_i$ labeled by $a_i$ in the
Cayley graph.  
For each $i$, either the directed edge $e_i$ is
in $\dgd$, or else $e_i \in \ves$ and 
$e_i <_c e$; in both cases we
have, by above or by Noetherian induction,
a van Kampen diagram $\dd_i:=\dd_{e_i}$ with boundary label
$y_{ga_1 \cdots a_{i-1}} a_i y_{ga_1 \cdots a_i}^{-1}$.
By using the ``seashell'' method, we successively 
glue the diagrams $\dd_{i-1}$, $\dd_i$ along their
common boundary words $y_{ga_1 \cdots a_{i-1}}$.
Since all of these gluings are along simple paths, 
this results in a planar van Kampen diagram $\dd_e'$ with boundary
word $y_g c(e) y_{h}^{-1}$.  
(Note that by our assumption that no generator represents
the identity, $c(e)$ must contain at least one letter.)
Finally, glue a polygonal 2-cell with boundary
label given by the relator $c(e)a^{-1}$ 
along the boundary subpath $c(e)$ in $\dd_e'$, in order to 
obtain the diagram $\dd_e$ with boundary 
word $w_e$.  Since in this step we have glued a disk onto $\dd_e'$
along an arc, the diagram $\dd_e$ is again planar, and is a \edg\ 
corresponding to $e$.
(See Figure~\ref{fig:genericstacking}.)
\begin{figure}
\begin{center}
\includegraphics[width=3.6in,height=1.6in]{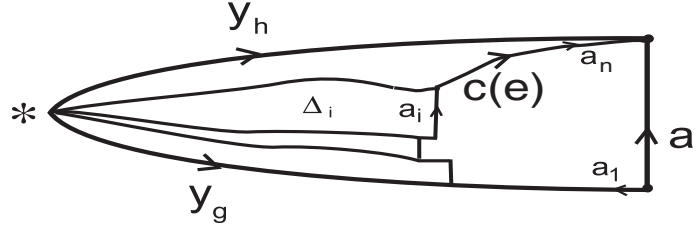}
\caption{Van Kampen diagram $\dd_e$ built from stacking}\label{fig:genericstacking}
\end{center}
\end{figure}

The final step to obtain the \nff\ associated to
the \fstkg\  is to eliminate repetitions.
Given any undirected edge $e$ in $E(X)$ choose 
$\dd_e$ to be a \edg\ constructed above for one of the
orientations of $e$.
Then the collection $\cc$ of normal forms,
together with the collection 
$\{\dd_e \mid e \in E(X)\}$
of {\edg}s, is a \nff\ for the \fstkbl\ group $G$.

\begin{definition}
A {\em \rnf} is a \nff\ that can be
constructed from a \fstkg\ by the
above procedure.  A {\em \rf} is
a filling induced by a \rnf\ using seashells.
\end{definition}

\begin{remark}\label{rmk:anfhasnf}
{\em
This \rnf\ and \rf\ both satisfy a further property which
we will exploit in our applications:
}
For every van Kampen diagram $\dd$ in the filling and 
every vertex $v$ in $\dd$,
there is an edge path in $\dd$ from the basepoint $*$ to $v$
labeled by the normal form in $\cc$ for the element
$\pi_{\dd}(v)$ in $G$.  
\end{remark}

As with our previous procedure, we have an effective
algorithm in the case that the set $\alg$ is computable.

\begin{proposition}\label{prop:fillalg}
If $G$ is \afstkbl\ over the finite 
generating set $A$, then there is an
algorithm which, upon input of a word $w \in A^*$
that represents the identity in $G$, will construct a van 
Kampen diagram for $w$ over the \stkg\ presentation.
\end{proposition}

Although our \stkg\ reduction procedure above for finding 
normal forms from a \fstkg\ can be used to describe the 
van Kampen diagrams in this \rf\ more directly, it is this
inductive view which will allow us to obtain
bounds on filling inequalities for \fstkbl\ groups 
in Section~\ref{sec:rnf}.

\begin{remark} {\em
For finitely generated groups that are not finitely
presented, the concept of a \stkg\ can still 
be defined, although in this case 
it makes sense to discuss {\stkg}s in terms
of a presentation for $G$, to avoid the (somewhat
degenerate) case in
which every relator is included in the presentation.
A group $G$ with symmetrized presentation
$\pp=\langle A \mid R\rangle$ is } \wstkbl\ 
{\em if there is a set $\cc=\{y_g \mid g \in G\}$
of normal forms over $A$ with $y_\ep=1$ and
a function $c:\ves \ra A^*$ satisfying
properties (S1) and (S2) of Definition~\ref{def:fstkg} together
with the condition that \stkg\ relation
set $R_c$ is a subset of $R$. 
The pair $(G,\pp)$ is} \awstkbl\ {\em if
again the set $\alg$ is computable, and}
\rwstkbl\ {\em if the set $\cc$ is a regular language
and the set $\alg$ is recursive.
The \stkg\ reduction procedure and the
inductive method for constructing van Kampen diagrams
over the presentation $\langle A \mid R_c\rangle$ 
of $G$ (and hence over $\pp$)
still hold in this more general setting.  
}
\end{remark}


\vspace{.1in}


\subsection{Tame filling inequalities: Definitions and 
  motivation} \label{subsec:tfdef}


$~$

\vspace{.1in}

Throughout this section, we assume that $G$ is 
a finitely presented group, with finite symmetrized presentation
$\pp=\langle A \mid R \rangle$.
We begin with a description of the \dmc\ 
filling inequalities which motivate the tame filling invariants
introduced in this paper.

Let $X$ be the Cayley complex for the presentation $\pp$, 
let $\xx$ be the 1-skeleton of $X$ (i.e., the Cayley graph)
and let $d_X$ be the path metric on $\xx$.
Given any word $w \in A^*$, let $l(w)$ denote the length
of this word in the free monoid.  By slight abuse of notation,
$d_X(\ep,w)$ then denotes the length of the element of $G$
represented by the word $w$, where as usual 
$\ep$ denotes the
identity element of the group $G$.
For any van Kampen diagram $\dd$ with basepoint $*$, let
$d_\dd$ denote the path metric on the 1-skeleton $\dd^1$. 
Recall that $\pi_\dd:\dd \ra X$ is the canonical map
that maps edges to edges preserving both label and
and direction, and satisfies $\pi_\dd(*)=\ep$.

\begin{definition} \label{def:idi}
A group $G$ with finite presentation $\pp$ satisfies an {\em intrinsic \dmi}
for a nondecreasing function $f:\N \ra \N$ if 
for all $w \in A^*$ with $w=_G \ep$,
there exists a van Kampen diagram $\dd$ for $w$ over $\pp$ such that
for all vertices $v$ in $\dd^0$
we have 
$d_\dd(*,v) \le f(l(w))$.

The pair $(G,\pp)$ satisfies an {\em extrinsic \dmi}
for a nondecreasing function $f:\N \ra \N$ if 
for all $w \in A^*$ with $w=_G \ep$,
there exists a van Kampen diagram $\dd$ for $w$ over $\pp$ such that
for all vertices $v$ in $\dd^0$
we have 
$d_X(\ep,\pi_\dd(v)) \le f(l(w))$.
\end{definition}

There are minimal such nondecreasing functions for any
pair $(G,\pp)$, namely the {\em intrinsic diameter function}
or {\em isodiametric function}, and the   
{\em extrinsic diameter function}.
See, for example, the exposition
in \cite[Chapter II]{riley} for more details on these two diameter functions.

We build a 3-dimensional view of the 1-skeleton $\dd^1$ of
a van Kampen diagram $\dd$ for a word $w$ 
by considering the following subsets of
$\dd^1 \times \R$:
\begin{eqnarray*}
 \dd^i & := & \{(p,d_\dd(*,p)) \mid p \in \dd^1\} \\
 \dd^e & := & \{(p,d_X(\ep,\pi_\dd(p))) \mid p \in \dd^1\} 
\end{eqnarray*}
Viewing the set of points $\dd^1 \times \{0\}$
at ``ground level'', and points $(p,r)$ at height $r$, 
the sets $\dd^i$ and $\dd^e$ transform the planar van Kampen
diagram into a topographical landscape of hills and valleys,
as in Figure~\ref{fig:hills}.
\begin{figure}
\begin{center}
\includegraphics[width=5in,height=1.6in]{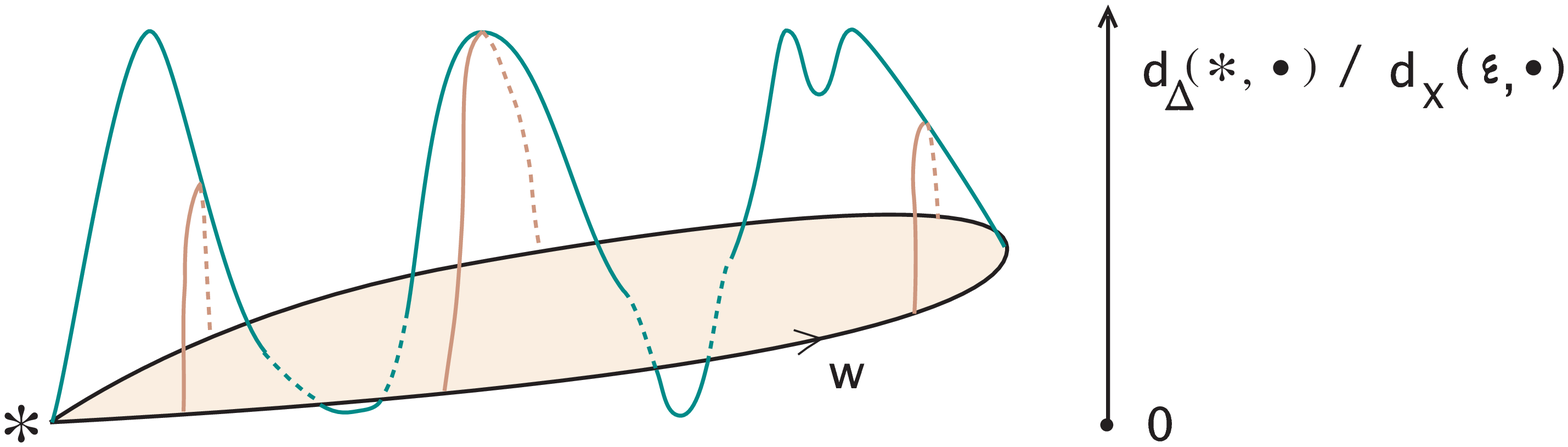}
\end{center}
\caption{Topographic view of van Kampen diagram}\label{fig:hills}
\end{figure}
A pair $(G,\pp)$ satisfies an intrinsic (resp. extrinsic)
\dmi\  for a
function $f:\N \ra \N$ if for every word $w$ representing the
identity of $G$, there is a van Kampen diagram $\dd$ for $w$ such
that the height of the highest peak in $\dd^i$ (resp. $\dd^e$) is
at most $f(l(w))+\frac{1}{2}$ (the constant $\frac{1}{2}$ takes
into account that we have kept the edges in our picture). 

These \dmc\  functions are quite coarse, in that they
do not measure whether only a few such peaks occur or whether
there are many peaks near this maximum height with deep valleys
between.  Saying this in another way, the \dmc\  functions
do not distinguish how wildly or tamely the peaks and
valleys occur in van Kampen diagrams.  In order to do this,
we refine the notion of a \dmc\  function to that of
a tame filling function, as follows.  

To begin, we define a collection of paths along which
we will measure the tameness of the hills. Intuitively,
these paths are a continuously chosen ``combing'' of
the boundary of the van Kampen diagram, as illustrated
in Figure~\ref{fig:vkhtpy}.   More formally, we have
the following.
\begin{figure}
\begin{center}
\includegraphics[width=2.2in,height=1.1in]{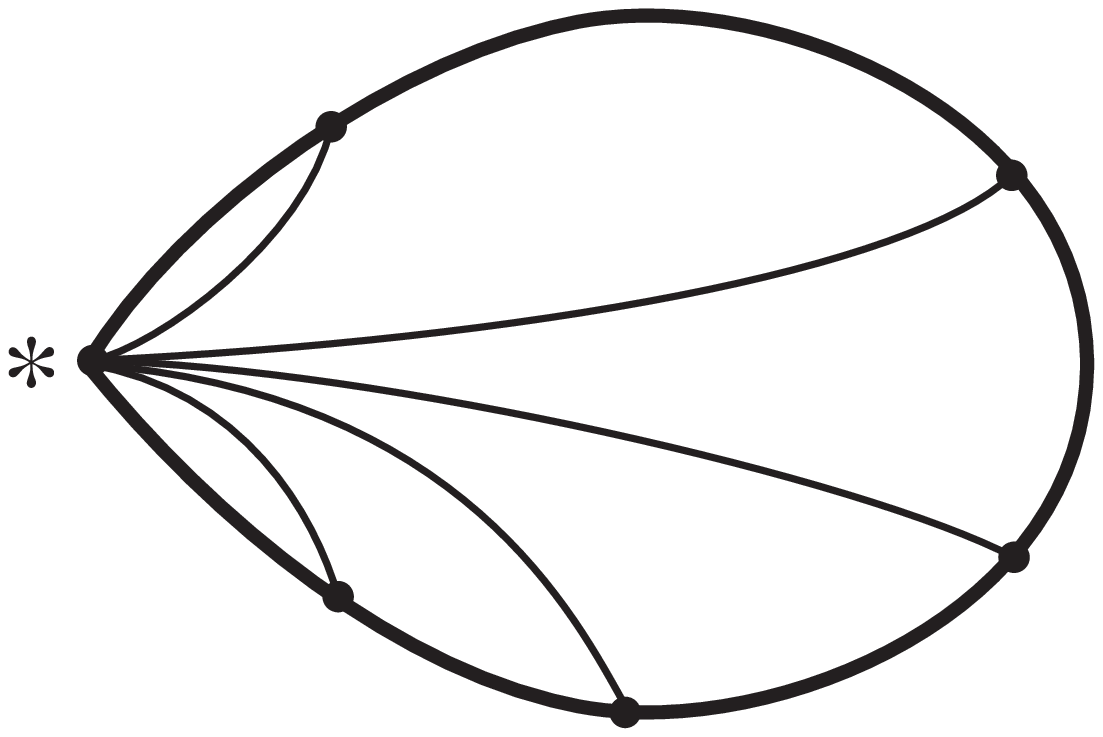}
\hspace{.3in} \includegraphics[width=2.2in,height=1.1in]{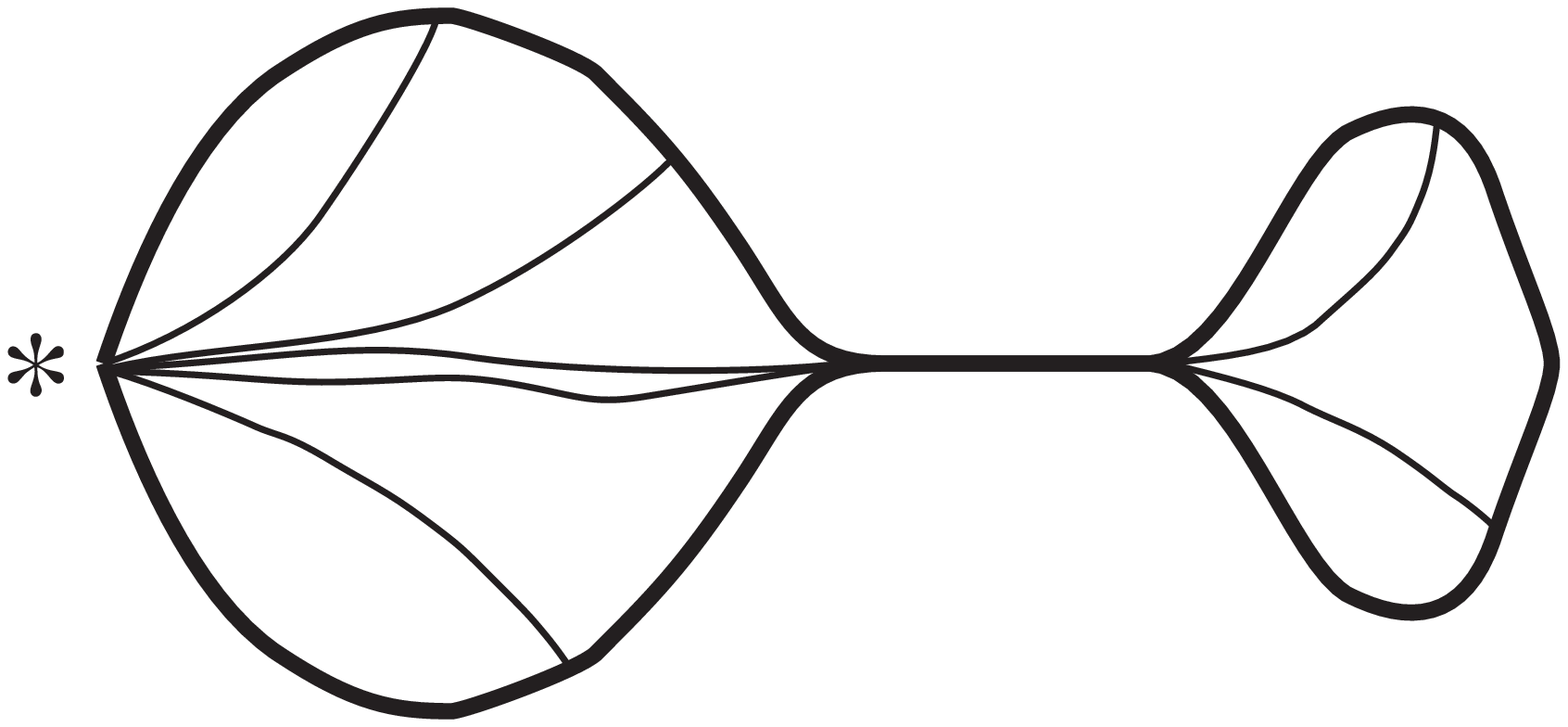}
\end{center}
\caption{Van Kampen homotopies}\label{fig:vkhtpy}
\end{figure}

\begin{definition} \label{def:vkh}
A {\em van Kampen homotopy} of a van Kampen diagram $\dd$ is a continuous
function $\ps:\bo \dd \times [0,1] \ra \dd$ satisfying:
\begin{enumerate} 
\item whenever $p \in \bo \dd$, then $\ps(p,0)=*$ and 
$\ps(p,1)=p$,
\item whenever $t \in [0,1]$, then $\ps(*,t)=*$, 
and
\item whenever $p \in (\bo \Delta)^0$, 
then $\ps(p,t) \in \Delta^1$ for all $t \in [0,1]$.
\end{enumerate}
\end{definition}

The \dms\  require a filling; that is,
a collection
of van Kampen diagrams for all words representing $\ep$.
Analogously, our refinement will require
a collection $\{(\dd_w,\ps_w) \mid w \in A^*, w=_G \ep\}$
such that for each $w$, $\dd_w$ is a van Kampen diagram
with boundary word $w$, and $\ps_w:\bo \dd_w \times [0,1] \ra \dd_w$ 
is a van Kampen homotopy.
We call such a collection a {\em \cfl} for the pair $(G,\pp)$.

To streamline notation later, it will be helpful to 
be able to measure the height (i.e., distance to the
basepoint) of each point in the van Kampen diagram, 
rather than just the height of points in the 1-skeleton.
Although we do not necessarily have a metric on a
Cayley 2-complex or van Kampen diagram, we can define a
coarse notion of distance in any
2-complex $Y$ as follows.

\begin{definition} \label{def:td}
Let $Y$ be a combinatorial 2-complex with basepoint vertex $y \in Y^0$, and
let $p$ be any point in $Y$.  Define the {\em coarse distance}
$\td_Y(y,p)$ by:
\begin{itemize}
\item If $p$ is a vertex, then
$\td_Y(y,p):=d_Y(y,p)$ is the path metric distance between
the vertices $y$ and $p$ in the graph $Y^1$.  
\item If $p$ is 
in the interior $Int(e)$ of an edge $e$ of $Y$, then
\\
$\td_Y(y,p) := \min\{\td_Y(y,v) \mid v \in \bo(e)\} + \frac{1}{2}$.
\item If $p$ is in the interior of a 2-cell $\sigma$ of $Y$, then 
\\
$~$ \hspace{-1in} $\td_Y(y,p) := \max\{\td_Y(y,q) \mid q \in Int(e)$
for edge $e$ of $\bo(\sigma)\}-\frac{1}{4}$.
\end{itemize}
\end{definition}


We can now expand the
3-dimensional view above of a
van Kampen diagram $\dd$, by considering the following
subsets of $\dd \times \R$:
\begin{eqnarray*}
 \tdd^i & := & \{(p,\td_\dd(*,p)) \mid p \in \dd\} \\
 \tdd^e & := & \{(p,\td_X(\ep,\pi_\dd(p))) \mid p \in \dd\} 
\end{eqnarray*}
Our measure of tameness is defined by how high
a van Kampen homotopy path can climb in the second
coordinate of these sets,
and yet still return to a much lower elevation later.


\begin{definition}\label{def:iti}
A group $G$ with finite
presentation $\pp$ satisfies an {\em intrinsic \tfi}
for a nondecreasing function $f:\nn \ra \nn$ if 
for all $w \in A^*$ with $w=_G \ep$,
there exists a van Kampen diagram $\dd$ for $w$ over $\pp$ 
and a van Kampen homotopy $\ps:\bo \dd \times [0,1] \ra \dd$
such that
\begin{description}
\item[$(\dagger^i)$] 
for all 
$p \in \bo \dd$ and $0 \le s < t \le 1$, we have 
\[ \td_\dd(*,\ps(p,s)) \le f(\td_\dd(*,\ps(p,t)))~. \]
\end{description}
\end{definition} 

In the topographic view, we compose the van Kampen homotopy
$\ps$ with the vertical projection $\nu: \dd \ra \tdd^i$.  Given
any point $p$ in the boundary of $\dd$, the composition 
$\nu \circ \ps(p,\cdot):[0,1] \ra \dd^i$ gives a 
(discontinuous, in general) path from $(*,0)$ 
to $(p,\td_\dd(*,p))$, i.e. from height 0 to height $\td_\dd(*,p)$,
such that
if at any ``time'' $s \in [0,1]$ the path has reached a height above
$f(q)$ for some $q \in \nn$, then at all later times $t>s$,
the path cannot return downward to a height at or below $q$.
Essentially, the \tfi\  implies that the paths in $\tdd^i$ rising
from the basepoint up to the boundary must 
go upward steadily, and not keep returning to significantly
lower heights.

As with the diameter functions above, we also consider
the extrinsic version of this filling function, which has
a similar interpretation using the projection to $\tdd^e$. 

\begin{definition}\label{def:eti}
A group $G$ with finite presentation $\pp$
satisfies an {\em extrinsic \tfi}
for a nondecreasing function $f:\nn \ra \nn$ if 
for all $w \in A^*$ with $w=_G \ep$,
there exists a van Kampen diagram $\dd$ for $w$ over $\pp$
and a van Kampen homotopy $\ps:\bo \dd \times [0,1] \ra \dd$
such that 
\begin{description}
\item[$(\dagger^e)$] 
for all $p \in \bo \dd$ and $0 \le s < t \le 1$,
we have 
\[ \td_X(\ep,\pi_\dd(\ps(p,s))) \le f(\td_X(\ep,\pi_\dd(\ps(p,t))))~.
\]
\end{description}
\end{definition} 


In contrast to the definition of \dms,
the definition of \tfi\  does not depend on
the length $l(w)$ of the word $w$.  Indeed, the property that
a homotopy path $\ps(p,\cdot)$ cannot return to an elevation
below $q$ after it has reached a height above $f(q)$ is 
uniform for all reduced words over $A$ representing $\ep$.
As a consequence, it is not clear whether every pair
$(G,\pp)$ has a (intrinsic or extrinsic) \tfi\ 
for a finite-valued function.


\section{Relationships among filling invariants}\label{sec:relationships}


The following proposition shows that for finitely
presented groups, tameness inequalities imply
\dms.

\begin{proposition} \label{prop:itimpliesid}
If a group $G$ with finite presentation $\pp$ satisfies an 
intrinsic [resp. extrinsic]
\tfi\  for a nondecreasing function $f:\nn \ra \nn$,
then the pair $(G,\pp)$ also satisfies an intrinsic [resp. extrinsic] 
\dmi\  for the function $\hat f:\N \ra \N$ defined
by $\hat f(n) = \lceil f(n) \rceil$.
\end{proposition}

\begin{proof}
We prove this for intrinsic tameness; the extrinsic proof is similar.
Let $w$ be any word over the generating set $A$ of the
presentation $\pp$ representing the trivial element $\ep$
of $G$, and let $\dd$, $\ps$ be a van Kampen diagram and
homotopy for $w$ satisfying the condition $(\dagger^i)$. 
Since the function $\ps$ is continuous,
each vertex $v \in \dd^0$ satisfies $v=\ps(p,s)$ for some $p \in \bo \dd$
and $s \in [0,1]$.  There is an edge
path along $\bo \dd$ from $*$ to $p$ labeled by at most
half of the word $w$, and so 
$\td_\dd(*,p) \le \frac{l(w)}{2}$.  Using the facts that 
$p=\ps(p,1)$ and $s \le 1$, 
condition $(\dagger^i)$ implies that $\td_\dd(*,v) \le f(\frac{l(w)}{2})$.
Since $f$ is nondecreasing, the result follows.
\end{proof}

In \cite{bridsonriley}, Bridson and Riley give an example
of a finitely presented group $G$ whose (minimal) intrinsic and
extrinsic diameter functions are not Lipschitz
equivalent.  
While we have not resolved the relationship
between \tfs\  in general, we give bounds
on their interconnections in Theorem~\ref{thm:itversuset}. 
These relationships between
the intrinsic and extrinsic \tfs\ 
are applied in examples later in this paper. 

\begin{theorem}\label{thm:itversuset}
Let $G$ be a finitely presented group with Cayley complex $X$ 
and \cfl\ $\cld$.
Suppose that $j:\N \ra \N$ is a nondecreasing function such that
for every vertex $v$ of a van Kampen diagram $\dd$ in $\cld$,
$d_{\dd}(*,v) \le j(d_X(\ep,\pi_{\dd}(v)))$, and
let $\tj:\nn \ra \nn$ be
defined by $\tj(n):=j(\lceil n \rceil)+1$.  
\begin{enumerate}
\item If $G$ satisfies an extrinsic \tfi\ for the function 
$f$ with respect to $\cld$, then $G$ 
satisfies an intrinsic \tfi\ for the function $\tj \circ f$.
\item If $G$ satisfies an intrinsic \tfi\ for the function 
$f$ with respect to $\cld$, then $G$ 
satisfies an extrinsic \tfi\ for the function $f \circ \tj$.
\end{enumerate}
\end{theorem}

\begin{proof}
We begin by showing that the inequality restriction
for $j$ on vertices holds for the function $\tj$ on all
points in the van Kampen diagrams in $\cld$, using the fact that
coarse distances on edges and 2-cells are closely
linked to those of vertices.
Let $(\dd,\ps)
 \in \cld$
and let $p$ be any point in $\dd$.
Among the vertices
in the boundary of the open cell of $\dd$ containing $p$, 
let $v$ be the vertex whose coarse
distance to the basepoint $*$ is maximal.
Then $\td_{\dd}(*,p) \le \td_{\dd}(*,v)+1$.
Moreover, $\pi_{\dd}(v)$ is again a vertex
in the boundary of the open cell of $X$ 
containing $\pi_{\dd}(p)$, and so
$\td_X(\ep,\pi_{\dd}(v)) \le \lceil \td_X(\ep,\pi_{\dd}(p)) \rceil$.
Applying the fact that $j$ is nondecreasing, then 
$\td_{\dd}(*,p) \le \td_{\dd}(*,v)+1
\le j(\td_{X}(\ep,\pi_{\dd}(v)))+1 \le 
j(\lceil \td_X(\ep,\pi_{\dd_w}(p)) \rceil) +1~.$
Hence the second inequality in 
\[ \td_X(\ep,\pi_{\dd}(p)) \le \td_{\dd}(*,p) 
\le \tj(\td_X(\ep,\pi_{\dd}(p))) \hspace{0.8in} (a)
\]
follows. The first inequality is a consequence
of the fact that coarse distance can only be
preserved or decreased by the natural map $\pi_\dd$
from any van Kampen diagram to the Cayley complex.

Now suppose that $G$ (with its finite presentation) satisfies an 
extrinsic \tfi\ for the function 
$f:\nn \ra \nn$ with respect to $\cld$.
Then for all $p \in \dd$ and for all $0 \le s < t \le 1$, we have
$
\td_{\dd}(*,\ps(p,s)) \le \tj(\td_X(\ep,\pi_{\dd}(\ps(p,s))))
\le \tj(f(\td_X(\ep,\pi_{\dd}(\ps(p,t)))))
\le \tj(f(\td_{\dd}(*,\ps(p,t))))
$
where the first and third inequalities follow from $(a)$,
and the second uses the extrinsic \tfi\ and the nondecreasing
property of $\tj$.  This completes the proof of (1).

The proof of (2) is similar.
\end{proof}



\section{Alternative views for \tfs}\label{sec:relax}


The definitions of intrinsic and extrinsic
\tfs\ for finitely presented groups
require a van Kampen diagram
for each word over the generators that represents
the trivial element of the group; i.e., a filling.  
In our first alternative view,
we show that the required 
collection of diagrams can be reduced to a \nff, at 
the cost of altering the direction of the homotopy paths.
We will apply this view in Section~\ref{sec:rnf} in
developing an algorithm to bound \tfs\ for \fstkbl\ groups.

In Definition~\ref{def:vkh},
our definition of a van Kampen homotopy 
$\ps:\bo \dd \times [0,1] \ra \dd$
is ``natural'', in the sense that the
first factor in the domain of this function 
is a subcomplex of the associated van Kampen
diagram $\dd$.  This
requires that for each point $p$ on an
edge $e$ of $\bo \dd$,
there is a unique choice of path from
the basepoint $*$ to $p$ via this homotopy.
However, when traveling along the boundary word 
$w$ counterclockwise around $\dd$, this
point $p$ (and undirected edge $e$) may be traversed more
than once, and it can be convenient to
have different combings of this edge corresponding
to the different traversals.  
In this section, we also show that 
a \tfi\ with respect to this
more relaxed condition is equivalent to
the \tfi\ defined in Section~\ref{subsec:tfdef}.
This second alternative view of the invariants will prove
useful in Section~\ref{sec:qiinv}, in our proof of the
quasi-isometry invariance of \tfs.

We begin with definitions to make both of
these statements more precise. 
First we extend
the notion of tameness to homotopies with other domains.

\begin{definition}
Let $f:\nn \ra \nn$ be a nondecreasing function, let
$Z$ be a space, and let
$\alpha:Z \times [0,1] \ra \dd$ be a continuous function
onto a van Kampen diagram $\dd$ with basepoint $*$, satisfying 
$\alpha(p,0)=*$ for all $p \in Z$.  
The map $\alpha$
is called {\em intrinsically $f$-tame} if 
for all $p \in Z$ and $0 \le s < t \le 1$,
we have 
\[ 
\td_\dd(*,\alpha(p,s)) \le f(\td_\dd(*,\alpha(p,t)))~.
\]
Similarly, the map $\alpha$
is {\em extrinsically $f$-tame} if 
for all $p \in Z$ and $0 \le s < t \le 1$,
we have 
\[ 
\td_X(\ep,\pi_\dd(\alpha(p,s))) \le f(\td_X(\ep,\pi_\dd(\alpha(p,t))))~.
\]
\end{definition}

Next we reroute the homotopy paths in
a diagram $\dd$, so that all paths travel
from the basepoint $*$ to a single edge of $\bo \dd$.
Suppose that $\dd$ is a van Kampen diagram for a 
word $w=w_1aw_2^{-1}$
such that each $w_i$ is a word over $A$ and $a \in A$.
An {\em \ehy} in $\dd$ of the directed
edge $e_a$ in $\bo \dd$, from vertex $v_1$ to
vertex $v_2$, corresponding to $a$ is
a continuous function $\tht:e_a \times [0,1] \ra \dd$
satisfying
\begin{description}
\item[(e1)] whenever $p$ is a point in $e_a$, then $\tht(p,0)=*$ and $\tht(p,1)=p$,
\item[(e2)] for $i=1,2$ the path 
$\tht(v_i,\cdot):[0,1] \ra \dd$ follows the
path labeled $w_i$ in $\bo \dd$.
\end{description}
See Figure~\ref{fig:edgefilling}.
\begin{figure}
\begin{center}
\includegraphics[width=2.2in,height=1.1in]{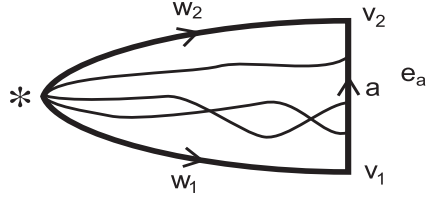}
\caption{Edge homotopy}\label{fig:edgefilling}
\end{center}
\end{figure}

A {\em \cnf} for a group $G$ with symmetrized 
presentation $\pp=\langle A \mid R \rangle$
consists of a collection $\cc \subseteq A^*$ of
simple word normal forms for $G$ together with a collection
$\cle=\{(\dd_e,\tht_e) \mid e \in E(X) \}$,
where $E(X)$ is the set of edges in the Cayley
complex $X$, satisfying the following:
\begin{description}
\item[(n1)] For each edge $e$, there is a choice
of direction for $e$ so that if the initial vertex
is $g$, the terminal vertex is $h$, the label on $e$
in the Cayley graph  
is $a$, and $y_g$,$y_{h}$ are the representatives of $g$,$h$
in $\cc$, then
$\dd_e$ is a van Kampen diagram for the word
${w_e}:=y_gay_{h}^{-1}$.
\item[(n2)] For each $e$, the
map $\tht_e:\he \times [0,1] \ra \dd_e$
is an \ehy\ of the directed edge
$\he$, from vertex $\hat g_e$ to vertex $\hat h_e$,
in $\bo \dd_e$ that corresponds to $a$ 
in the factorization of the boundary word ${w_e}$.  
\item[(n3)] For every pair of edges $e,e' \in E(X)$
with a common endpoint $g$, we require that  
$\pi_{\dd_e} \circ \tht_e(\hat g_e,t)=
\pi_{\dd_{e'}} \circ \tht_{e'}(\hat g_{e'},t)$ for all 
$t$ in $[0,1]$; that is, $\tht_e$ and $\tht_{e'}$
project to the same path, with the same
parametrization, in the Cayley complex $X$.
\end{description}


Note that, as in the case of fillings, every
\cnf\ induces a \cfl, again using the ``seashell''
method, as follows.  Given a \cnf\ 
$(\cc,\cle)$ and
a word $w=a_1 \cdots a_n$ with $w=_G \ep$ and
each $a_i \in A$, let $(\dd_i,\tht_i)$ be the
element of $\cle$ corresponding to the edge of $X$
with endpoints $a_1 \cdots a_{i-1}$ and
$a_1 \cdots a_i$.  If necessary replacing $\dd_i$ by
its mirror image and altering $\tht_i$ accordingly,
we may assume that $\dd_i$ has boundary label
$y_{i-1} a_i y_i$, where $y_i$ is the normal
form in $\cc$ of $a_1 \cdots a_i$.  As usual, let $\dd_w$
be the van Kampen diagram for $w$ obtained by successively
gluing these diagrams along their $y_i$ boundary
subpaths.  
This procedure yields a quotient map 
$\alpha:\coprod \dd_i \ra \dd_w$, such that each
restriction $\alpha|:\dd_i \ra \dd_w$ is an embedding.
Let $\he_i$ be the edge in the boundary path
of $\dd_i$ (and by slight abuse of notation also in the
boundary of $\dd_w$) corresponding to the letter $a_i$.
In order to build a van Kampen homotopy on
$\dd_w$, we note that the \ehs\ 
$\tht_i$ give a continuous function 
$\alpha \circ \coprod \tht_i: 
\coprod \he_i \times [0,1] \ra \dd_w$.
Recall that 
property (n3) of the definition of \cnf\ says that
on the common endpoint 
$v_i$ of the edges $\he_i$
and $\he_{i+1}$ of $\dd_w$, the paths
$\pi_{\dd_i} \circ \tht_i(v_i,\cdot)$ and
$\pi_{\dd_{i+1}} \circ \tht_{i+1}(v_i,\cdot)$
follow the edge path in $X$ labeled $y_i$ with the same parametrization.
Hence the same is true for the functions
$\tht_i(v_i,\cdot)$ and $\tht_{i+1}(v_{i+1},\cdot)$
following the edge paths labeled $y_i$ that were glued by
$\alpha$.
Moreover, if an
$\he_i$ edge and (the reverse of) an $\he_j$ edge are glued
via $\alpha$, the maps $\tht_i$ and $\tht_j$ have been
chosen to be consistent.
Hence the collection of maps $\tht_i$
are consistent on points identified by the
gluing map $\alpha$,
and we obtain
an induced function
$\ps_w:\bo \dd_w \times [0,1] \ra \dd_w$.
The \ehy\  conditions of the $\tht_i$ maps
imply that the function $\ps_w$ satisfies all
of the properties needed for the required van Kampen
homotopy on the diagram $\dd_w$.

\begin{definition}\label{def:nfi}
A group $G$ with finite presentation $\langle A \mid R \rangle$ 
satisfies an {\em intrinsic} [resp.~{\em extrinsic}]
{\em \cni} for a nondecreasing function
$f:\nn \ra \nn$ if there is a \cnf\ given by $\cc \subseteq A^*$,
and $\cle=\{(\dd_u,\tht_u) \mid u \in E(X)\}$ such that
each \ehy\ $\tht_{u}$ is intrinsically 
[resp.~extrinsically] $f$-tame.
\end{definition}

A \cnf\ is 
{\em geodesic} if all of the words in the
normal form set $\cc$ 
label geodesics in the associated Cayley graph.
We call a \cni\ {\em geodesic} if the associated
\cnf\ is geodesic.
A set of geodesic normal forms that we will 
use several times in this paper is the set of
{\em shortlex} normal forms.  
Choose a (lexicographic) total ordering on the
finite set $A$.   For any two words $z,z'$ over
$A$, we write $z \slex z'$ if $z$ is less than 
$z'$ in the
corresponding shortlex ordering on $A^*$.

Next we turn to relaxing the boundary condition.
Let $S^1$ denote the unit circle in the
${\mathbb R}^2$ plane.
For any natural number $n$, let 
$C_n$ be $S^1$
with a 1-complex structure consisting of
$n$ vertices (one of which is the basepoint
$(-1,0)$) and $n$ edges.

Given any van Kampen diagram $\dd$ over $\pp$ for a word
$w$ of length $n$, let $\vt_\dd:C_n \ra \bo \dd$
be the function that maps $(-1,0)$ to $*$
and, going counterclockwise once around $C_n$,
maps each subsequent edge of $C_n$ homeomorphically onto
the next edge in the counterclockwise path labeled $w$
along the boundary of $\dd$.

A {\em disk homotopy} of a van Kampen diagram $\dd$
over $\pp$ for a word $w$ of length $n$ is a continuous 
function $\ph:C_n \times [0,1] \ra \dd$ satisfying:
\begin{description}
\item[(d1)] whenever $p \in C_n$, then $\ph(p,0)=*$ and
$\ph(p,1)=\vt_\dd(p)$,
\item[(d2)] whenever $t \in [0,1]$, then $\ph((-1,0),t)=*$, 
and
\item[(d3)] whenever $p \in C_n^0$, 
then $\ph(p,t) \in \dd^1$ for all $t \in [0,1]$.
\end{description}

\begin{definition}\label{def:rtfi}
A group $G$ with finite presentation $\pp$ satisfies an
{\em intrinsic} [respectively, {\em extrinsic}]
{\em relaxed \tfi} for a nondecreasing function 
$f:\nn \ra \nn$ if
for all $w \in A^*$ with $w=_G \ep$,
there exists a van Kampen diagram $\dd$ for $w$ over $\pp$
and an intrinsically [resp.~extrinsically] $f$-tame 
disk homotopy $\ph:C_{l(w)} \times [0,1] \ra \dd$.
\end{definition}

Now we are ready to show that in the intrinsic 
case, these concepts are effectively equivalent.

\begin{proposition}\label{prop:htpydomain}
Let $G$ be a group with a 
finite symmetrized presentation $\pp$,
and let $f:\nn \ra \nn$ be a nondecreasing function.
The following are equivalent, up to Lipschitz equivalence
of the function $f$:
\begin{enumerate}
\item $(G,\pp)$ satisfies an intrinsic 
\tfi\ with respect to $f$.
\item $(G,\pp)$ satisfies an intrinsic 
relaxed \tfi\  with respect to $f$.
\item $(G,\pp)$ satisfies an intrinsic geodesic \cni\  
with respect to $f$.
\end{enumerate}
\end{proposition}

\begin{proof}
Write the presentation for $G$ as $\pp=\langle A \mid R \rangle$.

\noindent $(1) \Rightarrow(2)$: 

Given a van Kampen diagram
$\dd$ for a word $w$ 
and a van Kampen homotopy $\ps:\bo \dd \times [0,1] \ra \dd$,
the composition $\ph=\ps \circ (\vt_\dd \times id_{[0,1]}) :
C_n \times [0,1] \ra \dd$ is a disk homotopy for this diagram.
The fact that the identity function is used on the
[0,1] factor implies the result on the inequalities.

\noindent $(2) \Rightarrow(3)$:  

Suppose that
$\cld=\{(\dd_w,\ph_w) \mid w \in A^*, w=_G \ep\}$
is a collection of van Kampen diagrams and disk homotopies
such that 
each $\ph_w$ is intrinsically $f$-tame.

Let $\cc:=\{y_g \mid g \in G\}$ be the set of 
shortlex normal forms with respect to some lexicographic
ordering of $A$.

For any edge $e \in E(X)$, we orient the edge
$e$ from vertex $g$ to $h=ga$ if $y_g \slex y_{ga}$. 
There is a pair $(\dd_{w_e},\ph_{w_e})$ in $\cld$
associated to the word $w_e:=y_gay_{h}^{-1}$.
Define $\dd_e:=\dd_{w_e}$, and
let $\he$ be the edge in 
the boundary path of $\dd_{w_e}$ corresponding to the
letter $a$ in the concatenated word $w_e$.

We construct an \ehy\  
$\tht_{e}: \he \times [0,1] \ra \dd_e$ as follows.
Associated with the map $\ph_{w_e}$ we have
a canonical map $\vt_{\dd_e}:C_{l({w_e})} \ra \bo \dd_e$,
given by $\vt_{\dd_e}(q)=\ph_{w_e}(q,1)$, using disk homotopy
condition (d1).
Recall that this map wraps the simple edge circuit $C_{l({w_e})}$ 
cellularly along the edge path of $\bo \dd_e$.
Let $\gamma: \he \ra C_{l({w_e})}$ be a continuous map that
wraps the edge $\he$ once (at constant speed)
in the counterclockwise direction
along the circle, with the endpoints of $\he$
mapped to $(-1,0)$.  
For each point $p$ in $\he$, and for all
$t \in [0,\frac{1}{2}]$, define 
$\tht_e(p,t):=\ph_{w_e}(\gamma(p),2t)$.
Then $\tht_e(p,\frac{1}{2})=\vt_{\dd_e}(\gamma(p)) \in \bo \dd_e$. 

Let $\hhe$ be the (directed) edge of $C_{l({w_e})}$ corresponding
to the edge $\he$ of the boundary path in $\bo \dd_e$, 
with endpoint $v_1$ of $\hhe$ occurring earlier than endpoint $v_2$
in the counterclockwise path from $(-1,0)$.
Also let $\hhr_1$, $\hhr_2$ be the arcs of $C_{l({w_e})}$ mapping
via $\vt_{\dd_e}$ to the paths labeled by the subwords 
$y_g$, $y_h^{-1}$,
respectively, of ${w_e}$ in $\bo \dd_e$.
For each point $p$ in the interior $Int(\he)$ of the edge $\he$,
there is a unique point $\hhp$ in $\hhe$ with
$\vt_{\dd_e}(\hhp)=p$.
There is an arc (possibly a single point) in $C_{l({w_e})}$
from $\gamma(p)$ to $\hhp$ that is disjoint from the point
$(-1,0)$; let $\delta_p:[\frac{1}{2},1] \ra C_{l({w_e})}$
be the constant speed path following this arc.
That is, $\vt_{\dd_e} \circ \delta_p$ is a path in
$\bo \dd_e$ from $\vt_{\dd_e}(\gamma(p))$ to $p$.
In particular, if $\gamma(p)$ lies in $\hhr_1$, then
the path $\vt_{\dd_e} \circ \delta_p$ follows the end portion
of the boundary path labeled by $y_g$ from 
$\vt_{\dd_e}(\gamma(p))$ to the endpoint $\vt_{\dd_e}(v_1)$ 
of $\he$
and then follows a portion of $\he$ to $p$.  If $\gamma(p)$
lies in $\hhr_2$, the path $\vt_{\dd_e} \circ \delta_p$
follows a portion of the boundary path $y_h$
and $\he$  clockwise
from  $\vt_{\dd_e} \circ \delta_p$ via $\vt_{\dd_e}(v_2)$ to $p$,
and if $\gamma(p)$ is in $\hhe$, then
the path $\vt_{\dd_e} \circ \delta_p$ remains in $\he$.
Finally, for each point $p$
that is an endpoint $p=\vt_{\dd_e}(v_i)$ (with $i=1,2$),
let $\delta_p:[\frac{1}{2},1] \ra C_{l({w_e})}$
be the constant speed path along the arc $\hhr_i$ in $C_{l({w_e})}$
from $(-1,0)$ to $v_i$.  
Now for all $p$ in $\he$ and $t \in [\frac{1}{2},1]$,
define
$\tht_e(p,t):=\vt_{\dd_e}(\delta_p(t))$.

Combining the last sentences of the previous
two paragraphs, we have constructed a continuous function
$\tht_e:\he \times [0,1] \ra \dd_e$.
See Figure~\ref{fig:vktoedgehtpy} for an illustration 
of this map.
\begin{figure}
\begin{center}
\includegraphics[width=4.4in,height=1.8in]{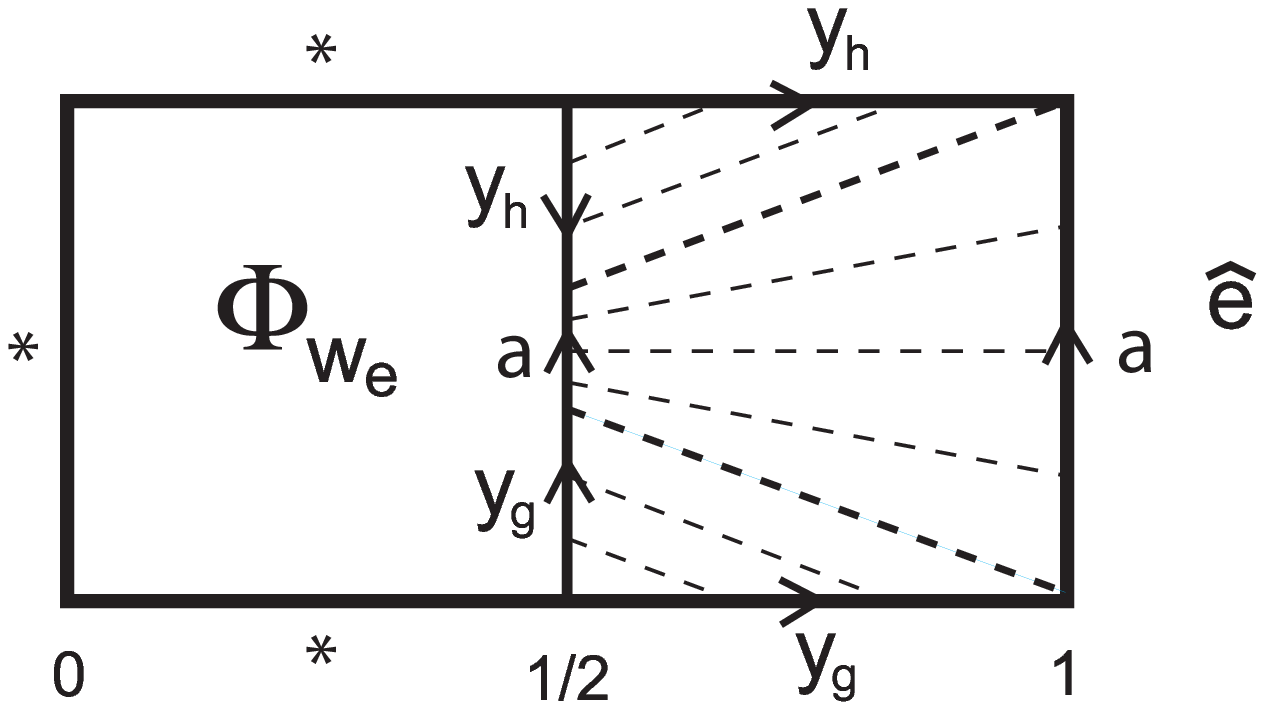}
\caption{Edge homotopy $\tht_e$ in Proposition~\ref{prop:htpydomain} 
proof (of (2) $\Rightarrow$ (3))}\label{fig:vktoedgehtpy}
\end{center}
\end{figure}
The disk homotopy conditions satisfied by $\ph_{w_e}$
imply that $\tht_e$ is an \ehy.  
Let $\cle$ be the collection
$\cle=\{(\dd_e,\tht_e) \mid 
e \in E(X) \}$.
Then $\cc$ together with $\cle$ define a geodesic \cnf\ of the
pair $(G,\pp)$.

Now we turn to analyzing the tameness of the
\ehy\  $\tht_e:\he \times [0,1] \ra \dd_e$.
Suppose that $p$ is any point in 
$\he$.
The fact
that the disk homotopy $\ph_{w_e}$ is intrinsically
$f$-tame 
implies that for all
$0 \le s<t \le \frac{1}{2}$, we have
$\td_{\dd_e}(*,\tht_e(p,s)) \le f(\td_{\dd_e}(*,\tht_e(p,t)))$.

The path
$\tht_e(p,\cdot)=\vt_{\dd_e}(\delta_p(\cdot)):[\frac{1}{2},1] \ra \dd_e$ 
on the second half of the interval $[0,1]$ follows a
geodesic in $\dd_e$, with the possible exception
of the end portion of this path that
lies completely contained in the edge $\he$.
Hence for all $\frac{1}{2} \le s<t \le 1$, we have
$\td_{\dd_e}(*,\tht_e(p,s)) \le \td_{\dd_e}(*,\tht_e(p,t))+1$.

Finally, whenever $0 \le s<\frac{1}{2}<t \le 1$,
we have
$\td_{\dd_e}(*,\tht_e(p,s)) \le f(\td_{\dd_e}(*,\tht_e(p,\frac{1}{2})))
\le f(\td_{\dd_e}(*,\tht_e(p,t))+1)$,
where the latter inequality uses the nondecreasing property
of $f$.

Putting these three cases together, the \ehy\ 
 $\tht_e$ is intrinsically $g$-tame
with respect to the nondecreasing function
$g:\nn \ra \nn$ given by $g(n)=f(n+1)$ for all $n \in \nn$,
and this function is Lipschitz equivalent to $f$.

\noindent $(3) \Rightarrow(1)$:  

Suppose that the
set $\cc=\{y_g \mid g \in G\}$ of geodesic normal
forms for $G$ together with the collection
$\cle=\{(\dd_e,\tht_e) \mid e \in E(X)\}$
is a \cnf\ with each $\tht_e$ intrinsically $f$-tame.
Let $\cld=\{(\dd_w',\ps_w)\}$ be the induced \cfl\ 
from the seashell method.

Fix a word $w \in A^*$ with $w=_G \ep$.
Recall that there is a quotient map
$\alpha:\coprod \dd_i \ra \dd_w'$ from
the seashell construction, where each $\dd_i$ 
is a \edg\ from $\cle$, and the restriction
of $\alpha$ to each $\dd_i$ is an embedding.   
In particular, for each $i$,  
paths labeled $y_i \in \cc$ in $\bo \dd_{i}$
and $\bo \dd_{i+1}$ are glued via $\alpha$.
Moreover, the
homotopy $\ps_w$ is induced by the map 
$\alpha \circ \coprod \tht_i:\coprod \he_i \times [0,1] \ra \dd_w'$. 

In order to analyze coarse distances in
the van Kampen diagram $\dd_w'$, 
we begin by supposing that $p$ is any vertex in $\dd_w'$.  
Then $p=\alpha(q)$ for some
vertex $q \in \dd_i$ (for some $i$).  The identification
map $\alpha$ cannot increase distances to the 
basepoint, so we have $d_{\dd_w'}(*,p) \le d_{\dd_i}(*,q)$.
Suppose that $\beta:[0,1] \ra \dd_w'$ is an edge path 
in $\dd_w'$ from $\beta(0)=*$
to $\beta(1)=p$ of length strictly less than $d_{\dd_i}(*,q)$.
This path cannot stay in the (closed)
subcomplex $\alpha(\dd_i)$ of $\dd'_w$,
and so there is a minimum time $0 < s \le 1$ such
that $\beta(t) \in \alpha(\dd_i)$ for all $t \in [s,1]$.
Then the point $\beta(s)$ must lie on the image
of the boundary of $\dd_i$ in $\dd_w'$.  Since the
words $y_{i-1}$ and $y_i$ label geodesics in the
Cayley complex of the presentation $\pp$, 
these words must also label geodesics in
$\dd_i$ and $\dd_w'$.  Hence we can replace the
portion of the path $\beta$ on the interval $[0,s]$
with the geodesic path along one of these words
from $*$ to $\beta(s)$, to obtain a new edge path
in $\alpha(\dd_i)$ from $*$ to $p$ of length 
strictly less than $d_{\dd_i}(*,q)$.  Since
$\alpha$ embeds $\dd_i$ in $\dd_w'$, this
results in a contradiction.  
Thus for each vertex $p$ in $\dd_w'$,
we have $d_{\dd_w'}(*,p)=d_{\dd_i}(*,q)$.

For any point in the interior of an
edge or 2-cell in $\dd_w'$, the coarse
distance in $\dd_w'$ to the basepoint is computed
from the path metric  distances of the vertices
in the boundary of the cell.  Hence the result
of the previous paragraph shows that for any
point $p$ in $\dd_w'$ with $p=\alpha(q)$ for
some point $q \in \dd_i$, we have
$\td_{\dd_w'}(*,p)=\td_{\dd_i}(*,q)$.  That is,
the map $\alpha$ preserves coarse distance.

Finally, for each $p \in \bo \dd_w'$, 
we have $p=\alpha(q)$ for some point 
$q$ in $\he_i \subseteq \dd_i$, and so 
$\ps_w(p,t)=\alpha(\tht_i(q,t))$ for all
$t \in [0,1]$.  Since each $\tht_i$ is an
intrinsically $f$-tame map, the
homotopy $\ps_w$ is also $f$-tame.
\end{proof}

In the extrinsic setting, a \cnf\ gives rise to
another type of homotopy, namely the notion of
a 1-combing, first defined by Mihalik and Tschantz in 
\cite{mihaliktschantz}.
A {\em 1-combing} of the Cayley complex
$X$ is a continuous function
$\up:\xx \times [0,1] \ra X$ satisfying that 
whenever $p \in \xx$, then $\up(p,0)=\ep$ and 
$\up(p,1)=p$, and
whenever $p \in X^0$, 
then $\up(p,t) \in \xx$ for all $t \in [0,1]$.
That is, a 1-combing is a continuous choice
of paths in the Cayley 2-complex $X$ from the
vertex $\ep$ labeled by the identity of $G$ 
to each point of the Cayley graph $\xx$, such
that the paths to vertices are required
to stay inside the 1-skeleton.

For a \cnf\ given 
by a set $\cc$ of normal forms together with a collection 
$\cle=\{(\dd_u,\tht_u) \mid u \in E(X)\}$
of van Kampen diagrams and \ehs,
there is a canonical associated 1-combing $\up$ of the
Cayley complex $X$ defined as follows.  For 
any point $p$ in the Cayley graph $\xx$, let $u$ be
an edge in $X$ containing $p$.  Then for any
$t \in [0,1]$, define $\up(p,t):=\pi_{\dd_u}(\tht_u(p,t))$.
The consistency condition (n3) of the definition of
a \cnf\ ensures that $\up$ is well-defined.

In fact, this associated 1-combing 
satisfies more restrictions than those of
Mihalik and Tschantz, in that the 1-combing factors through
\ehs\  of {\edg}s.  
We refer to these extra properties as \dia.
That is, a {\em \dia\  1-combing} of $X$ is a 1-combing
$\up:\xx \times [0,1] \ra X$ that also satisfies:
\begin{description}
\item[(c1)] whenever $t \in [0,1]$, then $\up(\ep,t)=\ep$.
\item[(c2)] whenever $v$ is a vertex in $X$,
the path $\up(v,\cdot)$ follows an embedded edge path
(i.e., no repeated vertices or edges)  
from $\ep$ to $v$ labeled by word $w_v$, and
\item[(c3)] whenever $e$ is a directed edge from 
vertex $u$ to vertex $v$ in $X$ labeled by $a$, then
there is a van Kampen diagram $\dd$ with respect to
$\pp$ for the word $w_uaw_v^{-1}$, 
together with
an \ehy\  $\tht:\hat e \times [0,1] \ra \dd$ 
associated to the edge $\hat e$ of $\bo \dd$ corresponding
to the letter $a$ in this boundary word,
such that 
\\
$\up \circ (\pi_\dd \times id_{[0,1]})|_{\hat e \times [0,1]}
=\pi_{\dd} \circ \tht$.
\end{description}

Mihalik and Tschantz~\cite{mihaliktschantz} defined a notion of tameness
of a 1-combing, which Hermiller and Meier~\cite{hmeiermeastame} 
refined to the idea of the 1-combing homotopy being $f$-tame
with respect to a function $f$, which they call a 
``radial tameness function''.
(In \cite{hmeiermeastame}, coarse distance in $X$
is described in terms of ``levels'', and the 
definition of coarse distance for a 2-cell 
is defined slightly differently from that of Definition~\ref{def:td}.)

\begin{definition}\label{def:radial} \cite{hmeiermeastame}
A group $G$ with finite presentation $\pp$ 
satisfies a {\em radial \tci}
for a nondecreasing function $\rho:\nn \ra \nn$ if
there is a \dia\ 1-combing $\up$ of the associated
Cayley 2-complex $X$ such that
\begin{description}
\item[$(\dagger^r)$] 
for all $p \in \xx$ and $0 \le s < t \le 1$,
we have 
\[ \td_X(\ep,\up(p,s)) \le \rho(\td_X(\ep,\up(p,t)))~.
\]
\end{description}
\end{definition}



Note that the radial \tci\  property is 
fundamentally
an extrinsic property,
utilizing (coarse) distances measured in the Cayley
complex.  Effectively, Proposition~\ref{prop:etistc}
below shows that 
the logical intrinsic analog of a radial \tci\ 
is the concept of an intrinsic \tfi\ or
intrinsic \cni.

In Proposition~\ref{prop:etistc}, we 
show that in the extrinsic setting, 
a stronger set of equivalences hold.
In particular, not only
are the intrinsic properties discussed in the
Proposition~\ref{prop:htpydomain} also
equivalent in the extrinsic case, they are also equivalent to
a radial \tci, and to an extrinsic \cni\ 
without the geodesic restriction.

\begin{proposition}\label{prop:etistc}
Let $G$ be a group with 
a finite symmetrized presentation $\pp$,
and let $f:\nn \ra \nn$ be a nondecreasing function.
The following are equivalent, up to Lipschitz equivalence
of the function $f$:
\begin{enumerate}
\item $(G,\pp)$ satisfies an extrinsic 
\tfi\ with respect to $f$.
\item $(G,\pp)$ satisfies an extrinsic 
relaxed \tfi\  with respect to $f$.
\item $(G,\pp)$ satisfies an extrinsic geodesic \cni\ 
with respect to $f$.
\item $(G,\pp)$ satisfies an extrinsic \cni\ 
with respect to $f$.
\item $(G,\pp)$ satisfies a radial \tci\ 
with respect to $f$.
\end{enumerate}
\end{proposition}

\begin{proof}
We first note that the proofs of (1) $\Rightarrow$ (2) $\Rightarrow$ (3)
are analogous to the proofs of the same intrinsic properties
in Proposition~\ref{prop:htpydomain}.  The implication
(3) $\Rightarrow$ (4) is immediate.

The implication (4) $\Rightarrow$ (1) follows the seashell
method as in the proof of (3) $\Rightarrow$ (1) in
Proposition~\ref{prop:htpydomain}. 
In the extrinsic setting,
the seashell quotient map $\alpha$ preserves 
extrinsic distances (irrespective of whether or
not the normal forms are geodesics).
That is, for any point $p$ in $\dd_i$, 
we have $\pi_{\dd_i}(p)=\pi_{\dd_w'}(p)$, and 
hence $\td_X(\ep,\pi_{\dd_i}(p))=\td_X(\ep,\pi_{\dd_w'}(p))$.
The rest of the proof follows.

The implication (4) $\Rightarrow$ (5) utilizes the
canonical \dia\ 1-combing $\up$ associated to a \cnf\ discussed
above.  For $f$-tame \ehs\  $\tht_e$ in 
the \cnf, the definition 
$\up(p,t):=\pi_{\dd_e}(\tht_e(p,t))$, 
implies that the condition ($\dagger^r$) (with
respect to the same function $f$) also holds.

Finally, for the implication (5) $\Rightarrow$ (4),
given a \dia\ 1-combing $\up$, the definition of
\dia\ implies that there is a canonically associated
\cnf\ through which $\up$ factors, as well.
Again we have ($\dagger^r$) implies that each of
the \ehs\  is $f$-tame, with
respect to the same function $f$, as an immediate
consequence.
\end{proof}

We note  that each of the properties in 
Propositions~\ref{prop:htpydomain} and~\ref{prop:etistc}
must also have the same quasi-isometry invariance
as the respective \tfi, from Theorem~\ref{thm:itisqi}.




\section{Combed fillings for stackable groups}\label{sec:rnf}


In this section we give an inductive procedure for
constructing a \cnf\ for any \fstkbl\ group.
Recall that in Section~\ref{subsec:stackdef}, an inductive
procedure was described for building a \nff\ 
from a \fstkg; in this section we extend this
process to include \ehs.  
Because
\ehs\  are built in this recursive fashion,
we will have finer control on their tameness than
for a more general \cnf.  We will utilize this
extra restriction to prove in Theorem~\ref{thm:rnftame}
that every \fstkbl\ group admits finite-valued
intrinsic and extrinsic \tfs.

Let $G$ be a \fstkbl\ group with 
\fstkg\ $(\cc,c)$ over a
finite inverse-closed generating set $A$,
and let $\pp=\langle A \mid R_c \rangle$
be the
(symmetrized) \stkg\ presentation, with Cayley complex $X$.
Let $\dgd$ be the set of degenerate edges in $X$, let $\ves$
be the set of recursive edges, and let
$<_c$ be the \stkg\  ordering.
We construct 
a \cnf\ for $G$ as follows.

The set $\cc$ will also be the set of normal forms
for the \cnf.  For each $g \in G$, let
$y_g$ denote the normal form for $g$ in $\cc$.

For each directed edge $e$ in $\vec E(X)=\dgd \cup \ves$, oriented
from a vertex $g$ to a vertex $h$ and labeled by $a \in A$,
let $w_e:=y_g a y_{h}^{-1}$.  
Let $\dd_e$ be the \edg\  (with
boundary word $w_e$) associated to $e$, obtained from
the \fstkg\ by using the construction in Section~\ref{subsec:stackdef}.

In the case that $e$ lies in $\dgd$, the 
diagram $\dd_e$ contains no 2-cells.
Let $\he$ be the edge of $\bo \dd_e$
corresponding to $a$ in the factorization of $w_e$; see Figure~\ref{fig:eintree}.
\begin{figure}
\begin{center}
\includegraphics[width=2.3in,height=0.7in]{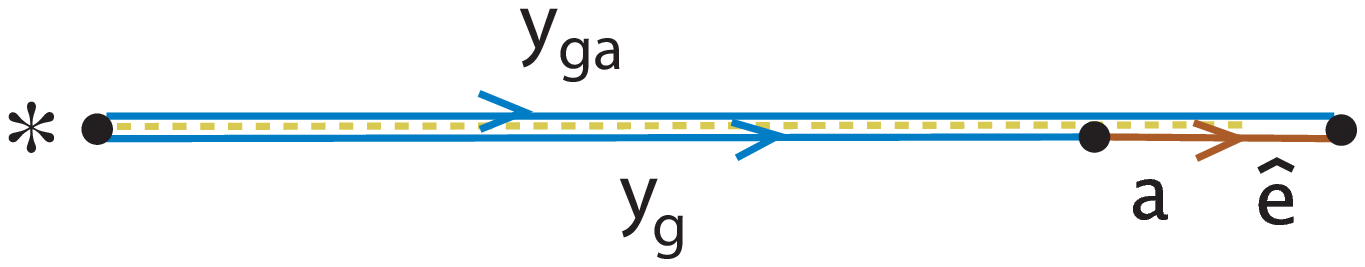}
\hspace{.2in}
\includegraphics[width=2.3in,height=0.7in]{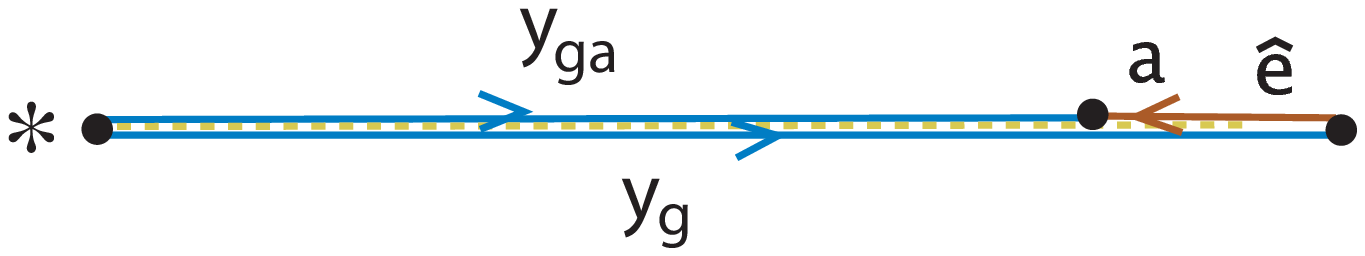}
\caption{$(\dd_e,\tht_e)$ for 
$e$ in $\dgd$}\label{fig:eintree}
\end{center}
\end{figure} 
Define the \ehy\  $\tht_e:\he \times [0,1] \ra \dd_e$
by taking $\tht_e(p,\cdot):[0,1] \ra \dd_e$ 
to follow the
shortest length (i.e. geodesic with respect to
the path metric) path from the basepoint $*$ to $p$
at a constant speed,
for each point $p$ in $\he$.

Next we use the recursive construction of the
van Kampen diagram $\dd_e$ to
recursively construct the 
\ehy\  in the case that 
$e \in \ves$.
Recall that if we
write $c(e)=a_1 \cdots a_n$ with each $a_i \in A^*$,
then the \edg\ $\dd_e$ is constructed  
from {\edg}s $\dd_i$
with boundary labels 
$y_{ga_1 \cdots a_{i-1}} a_i y_{ga_1 \cdots a_{i}}$,
obtained by induction or from degenerate edges.
These diagrams are glued along their common boundary paths
$y_i:=y_{ga_1 \cdots a_{i}}$ 
(to obtain the ``seashell'' diagram $\dd_e'$), 
and then a
single 2-cell with boundary label 
$c(e)a^{-1}$ is glued onto $\dd_e'$
along the $c(e)$ subpath of $\bo \dd_e'$,
to produce $\dd_e$.

A slightly alternative view of this construction 
of $\dd_e$ will allow us more flexibility in
constructing the \ehy\ associated to this diagram, which 
in turn will lead to better tameness bounds later.
Factor $c(e)=x_g^{-1} \tc_e x_h$ such that the directed edges 
in the paths in $X$ labeled
by $x_g^{-1}$ starting at $g$,
and labeled by $x_h^{-1}$ starting at $h$, all
lie in $\dgd$, and such that $y_g=y_q x_g$ and $y_h=y_r x_h$
where $q:=_G gx_g^{-1}$ and $r:=_G hx_h^{-1}$.
There are indices $j,k$ such that
$\tc_e:=a_j \cdots a_k$. 
If the word $\tc_e$ is
nonempty, then $\dd_e$ can also be constructed by a seashell
gluing of the {\edg}s $\dd_j, \cdots, \dd_k$ 
to produce
a diagram $\dd_e''$ with boundary labeled $y_q \tc_e y_r^{-1}$, 
after which
a single 2-cell $\hc_e$ with boundary label $c(e)a^{-1}$
is glued onto $\dd_e''$, along the $\tc_e$ subpath in $\bo \dd_e''$,
to produce $\dd_e$.  If the word $\tc_e$ is empty,
then $q=r$,
and $\dd_e$ is obtained by taking a simple edge
path from a basepoint labeled by the word
$y_{q}$ (i.e., the van Kampen diagram for the
word $y_qy_q^{-1}$ with no 2-cells), 
and attaching a single 2-cell $\hc_e$
with boundary label $c(e)a^{-1}$, gluing the
end of the $y_q$ edge path to the vertex of $\bo \hc_e$
separating the $x_g^{-1}$ and $x_h$ subpaths.
It follows from this construction that the diagrams $\dd_i$ 
and the
cell $\hc_e$ can be considered to be subsets of $\dd_e$.

Let $\he$ be the directed edge in $\bo \dd_e$ from 
vertex $\hat g$ to vertex $\hat h$ corresponding
to $a$ in the factorization of $w_e$.
Let $\hat q$ and $\hat r$ be the vertices of the
2-cell $\hc_e$ at the start and end, respectively,
of the path in $\bo \hc_e$
labeled by $\tc_e$.  
Let $J:\he \ra [0,1]$ be a homeomorphism, with
$J(\hat g)=0$ and $J(\hat h)=1$.
Since $\hc_e$ is a disk, there is
a continuous function $\Xi_e:\he \times [0,1] \ra \hc_e$
such that: (i) For each $p$ in the interior $Int(\he)$, we have
$\Xi_e(p,(0,1)) \subseteq Int(\hc_e)$ and $\Xi_e(p,1)=p$.
(ii) $\Xi_e(J^{-1}(\cdot),0):[0,1] \ra \hc_e$ follows the path
in $\bo \hc_e$ labeled $\tc_e$ from $\hat q$ to $\hat r$
at constant speed.
(iii) $\Xi_e(\hat g,\cdot)$ follows the path
in $\bo \hc_e$ labeled $x_g$ from $\hat q$ to $\hat g$
at constant speed.  (iv) $\Xi_e(\hat h,\cdot)$ follows the path
in $\bo \hc_e$ labeled $x_h$ from $\hat r$ to $\hat h$
at constant speed.  
Let $l_g$, $m_g$, $l_h$, and $m_h$ be the lengths 
of the words $y_q$, $x_g$, $y_r$, and $x_h$ in $A^*$,
respectively.

We give a piecewise definition of the \ehy\  
$\tht_e:\he \times [0,1] \ra \dd_e$
as follows. For any point $p$ in $\he$, 
if $\tc_e$ is a nonempty word, then there is an index
$j \le i \le k$ such that the point $\Xi_e(p,0)$ lies in $\dd_i$.
In the case that $\tc_e=1$, let $\tht_i(\hat q,\cdot)$ 
in the following formula denote the
constant speed path following the geodesic
in $\dd_e$ from $*$ to $\hat q=\hat r$.
Define 
\[ \tht_e(p,t):= \left\{
  \begin{array}{ll}
  \tht_i(\Xi_e(p,0),\frac{1}{a_p}t)          
        & \mbox{if $t \in [0,a_p]$} \\
  \Xi_e(p,\frac{1}{1-a_p}(t-a_p))          
        & \mbox{if $t \in [a_p,1]$} 
  \end{array} \right.
\]
where
\[ a_p:= \left\{
  \begin{array}{ll}
  \frac{2l_g}{l_g+m_g}(\frac{1}{2}-J(p))+J(p)          
        & \mbox{if $J(p) \in [0,\frac{1}{2}]$} \\
  (1-J(p))+\frac{2l_h}{l_h+m_h}(J(p)-\frac{1}{2})          
        & \mbox{if $J(p) \in [\frac{1}{2},1]$} 
  \end{array} \right.
\]
Note that if $a_p=0$
and $J(p) \in [0,\frac{1}{2}]$, then we must also
have $J(p)=0$ and $l_g=0$.  In this case $p=\hat g$
and $y_q$ is the empty word, and so 
$\tht_i(\Xi_e(p,0),\cdot)=\tht_i(\hat q,\cdot)$ is
a constant path at the basepoint $*$ of $\dd_e$;
hence $\tht_e$ is well-defined in this case.  
The other instances in which $a_p$
can equal 0 or 1 are similar.

The complication in this definition of $\tht_e$ stems from
the need to
ensure that for the endpoint vertices $\hat g$
and $\hat h$ of $\he$, the projections to $X$ of the
paths $\tht_e(\hat g, \cdot)$ and $\tht_e(\hat h,\cdot)$
via the map $\pi_{\dd_e}$ are consistent with
the paths defined for all other edges to these points;
that is, to ensure that the property (n3) of the 
definition of \cnf\ will hold.  
In particular, we ensure that the paths
$\tht_e(\hat g,\cdot)$, $\tht_e(\hat h,\cdot)$
follow the words $y_g$, $y_h$, respectively, in $\bo \dd_e$
at constant speed. 
The van Kampen diagram $\dd_e$ and \ehy\ 
$\tht_e$ are illustrated in Figure~\ref{fig:rnfebad}.
\begin{figure}
\begin{center}
\includegraphics[width=4.3in,height=1.7in]{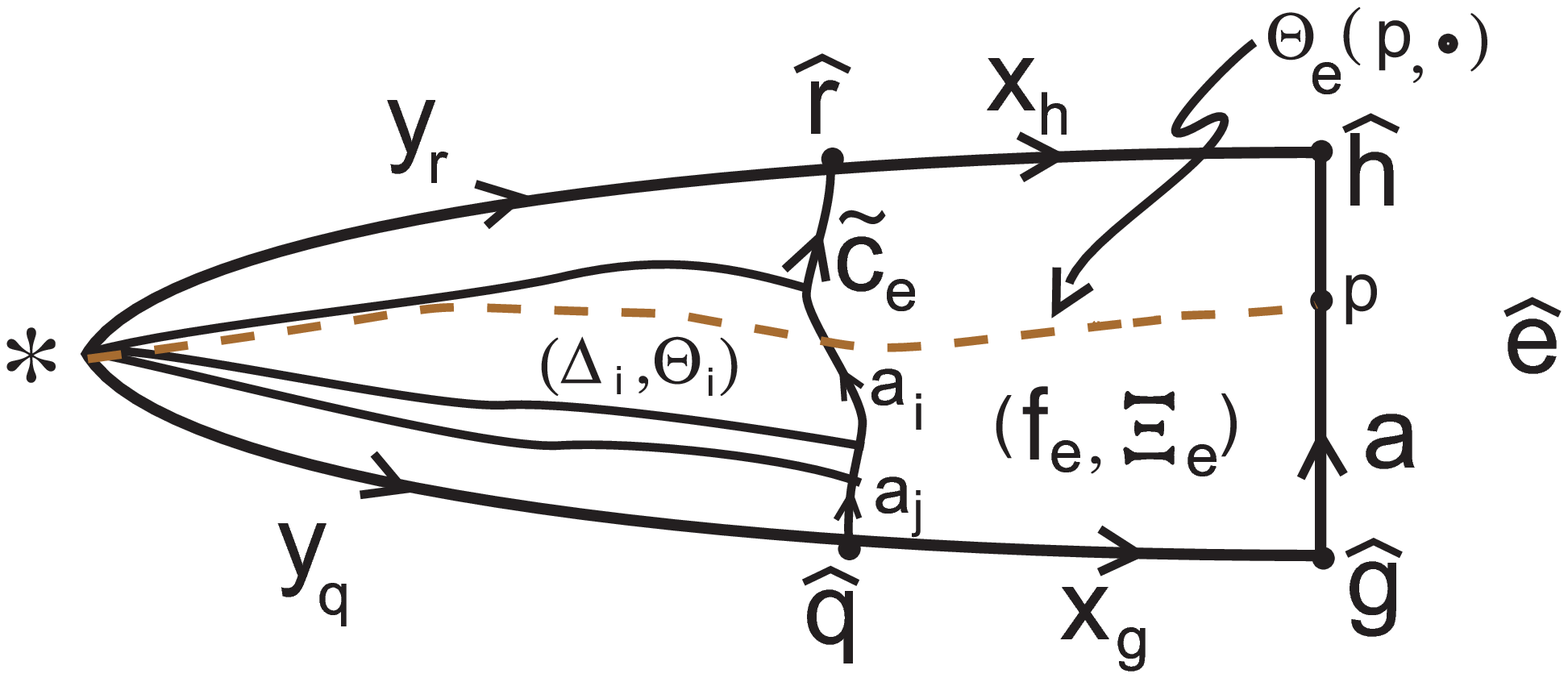}
\caption{$(\dd_e,\tht_e)$ for 
$e$ in $\ves$}\label{fig:rnfebad}
\end{center}
\end{figure}

We now have a collection of van Kampen diagrams
and \ehs\  for the elements of $\vec E(X)$.
To obtain the \cnf\ associated to
the \fstkg, the final step again is to eliminate repetitions.
Given any undirected edge $e$ in $E(X)$, 
let $(\dd_,\tht_e)$ be a \edg\ and \ehy\ constructed
above for one of the orientations of $e$.
Then the collection $\cc$ of prefix-closed normal forms
from the \fstkg,
together with this collection 
$\cle:=\{(\dd_e,\tht_e) \mid e \in E(X)\}$
of van Kampen diagrams and \ehs, is a \cnf\ for $G$.

\begin{definition}
A {\em \rcnf} is a \cnf\ that can be 
constructed from a \fstkg\  by the
above procedure. A {\em \rcf} is a \cfl\ induced
by a \rcnf\ using seashells.
\end{definition}

 
The extra structure of this recursively
defined \cnf\ $(\cc,\cle)$ 
allows us to compute finite-valued
\tfs\ for $G$.
To analyze the tameness of the \ehs, we 
consider the {\em intrinsic diameter} $\idi(\dd_e)$ and
{\em extrinsic diameter} $\edi(\dd_e)$
of each van Kampen diagram in the collection $\cle$;
that is, 
$\idi(\dd_e) = \max\{d_{\dd_e}(\ep,v) \mid v \in \dd_e^0\}$ and
$\edi(\dd_e) = \max\{d_X(\ep,\pi_{\dd_e}(v)) \mid v \in \dd_e^0\}$, where
$X$ is the Cayley complex of the \stkg\ presentation.
Let $B(n)$ be the ball of radius $n$ (with respect to
path metric distance) in the Cayley
graph $\xx$ centered at $\ep$.
Define the functions 
$\kti, \kte, \kxi, \kxe:\N \ra \N$ by
\begin{eqnarray*}
\kti(n) &:=& \max\{l(y_g) \mid g \in B(n)\}~,\\
\kte(n) &:=& \max\{d_X(\ep,x) \mid \exists~g \in B(n) \text{ such that } 
      x \text{ is a prefix of } y_g \}~,\\
\kxi(n) &:=& \max\{\idi(\dd_e) \mid e \in \ves
              \text{ and the initial vertex of } e \text{ is in } B(n)\}~,\\
\kxe (n) &:=& \max\{\edi(\dd_e) \mid e \in \ves 
              \text{ and the initial vertex of } e \text{ is in } B(n)\}~. 
\end{eqnarray*}
Note that we do not assume that prefixes are proper.
(Also note that the van Kampen diagrams in the 
\cfl\ induced by the \rcnf\ $(\cc,\cle)$ may
not realize the minimal possible intrinsic or 
extrinsic diameter among all van Kampen diagrams
for the same boundary words.)

We will need to consider coarse distances throughout
the Cayley complex $X$.  
To that end, define the
functions $\mui,\mue:\nn \ra \nn$ by
\begin{eqnarray*}
\mui(n) &:=& \max \{\kti(\lceil n \rceil +  1) + 1, n+1, 
         \kxi(\lceil n \rceil + \maxr+1)\} \text{ and} \\
\mue(n) &:=& \max \{\kte(\lceil n \rceil +  1) + 1, n+1, 
         \kxe(\lceil n \rceil + \maxr+1)\}~,
\end{eqnarray*}
where $\maxr$ is the length of the longest relator
in the \stkg\ presentation $\pp$.  
It follows directly from the definitions that $\kti,\kte,\kxi,\kxe$ 
are nondecreasing functions, and therefore so are $\mui$ and $\mue$.

\begin{theorem}\label{thm:rnftame}
If $G$ is a \fstkbl\ group,  
then $G$ admits an intrinsic \tfi\ for the finite-valued function
$\mui$, and an extrinsic \tfi\ for the finite-valued function
$\mue$.
\end{theorem}

\begin{proof}
Let $\cld=\{(\dd_w,\ps_w)\}$ be the \cfl\ 
obtained via the seashell method
from the \rcnf\ $(\cc,\cle)$ associated to
a \fstkg\  $(\cc, c)$ for $G$.
As usual, we write $\cc=\{y_g \mid g \in G\}$.
Let $\dd_w$ be any of the van Kampen diagrams 
in $\cld$,
let $p$ be any point in $\bo \dd_w$, and let
$0 \le s<t \le 1$.
To simplify notation later, we also 
let $\sigma:=\ps_w(p,s)$ and $\tau:=\ps_w(p,t)$.

If $\tau$ is in the 1-skeleton $\dd_w^1$, then
let $\tau':=\tau$ and $t':=t$.  Otherwise, $\tau$
is in the interior of a 2-cell, and there is a
$t \le t' \le 1$ such that $\ps_w(p,[t,t'))$ is
contained in that open 2-cell, and 
$\tau':=\ps_w(p,t') \in \dd_w^1$.

\smallskip

{\em Case I.  $\tau' \in \dd_w^0$ is a vertex.}
In this case 
the path $\ps_w(p,\cdot):[0,t'] \ra X$
follows the edge path labeled $y_{\pi_{\dd_w}(\tau')}$ from $*$, 
through $\sigma$, to $\tau=\tau'$ (at constant speed).  
There is a vertex $\sigma'$ on this path lying on the
same edge as $\sigma$ (with $\sigma'=\sigma$ if $\sigma$ is
a vertex) satisfying
$\td_{\dd_w}(*,\sigma) < d_{\dd_w}(*,\sigma')+1$
and $\td_{X}(\ep,\pi_{\dd_w}(\sigma)) < d_{X}(\ep,\pi_{\dd_w}(\sigma'))+1$.
The subpath from $*$ to $\sigma'$ is labeled by a prefix $x$ of 
the word $y_{\pi_{\dd_w}(\tau)}$.
Then 
\[
\td_{\dd_w}(*,\sigma) < d_{\dd_w}(*,\sigma')+1 \le
l(y_{\pi_{\dd_w}(\tau)})+1 \le \kti(d_X(\ep,\pi_{\dd_w}(\tau)))+1
\le \kti(d_{\dd_w}(*,\tau))+1
\]
\[
\text{ and }
 \td_{X}(\ep,\pi_{\dd_w}(\sigma)) 
< d_{X}(\ep,\pi_{\dd_w}(\sigma'))+1 \le
\kte(d_X(\ep,\pi_{\dd_w}(\tau)))+1~.
\]

\smallskip

{\em Case II. $\tau'$ is in the interior of an edge $\he$ of $\dd_w$.} 
From the seashell construction, the path $\ps_w(p,\cdot):[0,1] \ra \dd_w$
lies in a subdiagram $\dd'$ of $\dd_w$ such that
$\dd'$ is a \edg\ in
$\cle$.  From the construction of the \rcnf,
the subpath $\ps_w(p,\cdot):[0,t'] \ra \dd_w$
lies in a subdiagram $\dd_e$ of $\dd'$ for some
pair  $(\dd_e,\tht_e) \in \cle$ associated to
a directed edge $e \in \dgd \cup \ves$.
Moreover, $\he$ is the edge of $\dd_e$ corresponding to $e$, 
and the path $\ps_w(p,\cdot):[0,t'] \ra \dd_w$ is
a bijective (orientation preserving)
reparametrization of the path $\tht_e(\tau',\cdot):[0,1] \ra \dd_e$.

\smallskip

{\em Case IIA.  $e \in \dgd$.} 
The van Kampen diagram $\dd_e$ contains no 2-cells, and
the path $\tht_e(\tau',\cdot):[0,1] \ra \dd_e$
follows the edge path labeled by a normal form
$y_g \in \cc$ from $*$ to 
$\hat g$ (at constant speed), 
and then follows the portion of $\he$ from $\hat g$ to $\tau'$,
where $\hat g$ is the endpoint of $\he$
closest to $*$ in the diagram $\dd_e$.
In this case, $\tau$ must also lie in $\dd_w^1$, and
so again we have $\tau=\tau'$.
Since $\hat g$ and $\tau$ lie in the same closed 1-cell,
we have $d_{\dd_w}(*,\hat g)< \lceil \td_{\dd_w}(*,\tau) \rceil +1$,
and similarly for their images (via $\pi_{\dd_w}$) lying in the same
closed edge of $X$.

If $\sigma$ lies in the $y_{g}$ path, then 
Case I applies to that path, with $\tau$ replaced by the 
vertex $\hat g$.  Combining
this with the inequality above and applying the nondecreasing
property of the functions $\kti$ and $\kte$ yields
\[
\td_{\dd_w}(*,\sigma) <  \kti(d_{\dd_w}(*,\hat g))+1
\le \kti(\lceil \td_{\dd_w}(*,\tau) \rceil+1)+1
\text{ and }
\]
\[
 \td_{X}(\ep,\pi_{\dd_w}(\sigma)) < \kte(d_X(\ep,\pi_{\dd_w}(\hat g)))+1
\le \kte(\lceil \td_X(\ep,\pi_{\dd_w}(\tau)) \rceil +1)+1~.
\]

On the other hand, if $\sigma$ lies in $\he$, 
then $\sigma$ and $\tau$ are contained in a common edge.
Hence 
\[
\td_{\dd_w}(*,\sigma) <  \td_{\dd_w}(*,\tau) +1
\text{  and  } \td_{X}(\ep,\pi_{\dd_w}(\sigma)) <  
\td_X(\ep,\pi_{\dd_w}(\tau)) +1~.
\]

\smallskip

{\em Case IIB. $e \in \ves$.}  
In this case either $\tau=\tau'$, or $\tau$ is
in the interior of the cell $\hc_e$ of the
diagram $\dd_e$.  
Let $g$ be the initial vertex 
of the directed edge $e$.
Then $g$ and $\pi_{\dd_w}(\tau)$ lie in a common
edge or 2-cell of $X$, and so
$d_{X}(\ep,g) < \lceil \td_{X}(\ep,\pi_{\dd_w}(\tau)) \rceil +\maxr+1$,
where $\maxr$ is the length of the longest relator in 
the presentation of $G$.

Note that distances in the  subdiagram $\dd_e$
are bounded below by distances in $\dd_w$.
In this case, combining these inequalities and the
nondecreasing properties of $\kxi$ and $\kxe$ yields
\begin{eqnarray*}
\td_{\dd_w}(*,\sigma) &\le&  \td_{\dd_e}(*,\sigma) 
\le \idi(\dd_e) \le \kxi(d_{X}(\ep,g)) \\
&\le&  \kxi(\lceil d_{X}(\ep,\pi_{\dd_w}(\tau)) \rceil +\maxr+1)
\le \kxi(\lceil d_{\dd_w}(*,\tau)) \rceil +\maxr+1)
\text{ and } \\
 \td_{X}(\ep,\pi_{\dd_w}(\sigma)) 
&=&\td_{X}(\ep,\pi_{\dd_e}(\sigma)) 
\le \edi(\dd_e) \le \kxe(d_{X}(\ep,g)) \\
&\le& \kxe(\lceil \td_{X}(*,\pi_{\dd_w}(\tau)) \rceil +\maxr+1)~.
\end{eqnarray*}


Therefore in all cases, we have
$\td_{\dd_w}(*,\sigma) \le \mui(\td_{\dd_w}(*,\tau))$
and $\td_{X}(\ep,\pi_{\dd_w}(\sigma)) \le \mue(\td_{X}(\ep,\pi_{\dd_w}(\tau)))$,
as required.
\end{proof}

The \tfi\ 
bounds in Theorem~\ref{thm:rnftame} are not sharp in general.
In particular, we will improve upon these bounds for the example of
almost convex groups in Section~\ref{subsec:ac}.


Recall from Section~\ref{subsec:stackdef}
that the group $G$ is \afstkbl\ if 
there is a \fstkg\ $(\cc,c)$ over a 
finite generating set $A$ of $G$
for which the subset

$\alg=\{(w,a,x) \mid w \in A^*, a \in A, x=c'(\wa) \}$

\noindent of $A^* \times A \times A^*$ is computable,
where 
$\wa$ denotes the directed edge in $X$ labeled $a$ from $w$ to $wa$,
and $c'(\wa)=c(\wa)$ for $\wa \in \ves$ and $c'(\wa)=a$ for 
$\wa \in \dgd$.
For \afstkbl\ groups, the procedure 
described above for building a \rcnf\ from
the \stkg\ is again algorithmic.  

Note that
whenever the group $G$ admits a \tfi\ for
a function $f:\nn \ra \nn$, and $g:\nn \ra \nn$
satisfies the property that $f(n) \le g(n)$
for all $n \in \nn$, then $G$ also admits the
same type of \tfi\ for the function $g$.
Applying this, we obtain a computable bound on \tfs\ for
\afstkbl\  groups.  

\begin{theorem}\label{thm:astktf}
If $G$ is an \afstkbl\ group, then $G$ satisfies
both intrinsic and extrinsic \tfs\ with respect to
a recursive function.
\end{theorem}

\begin{proof}
From 
Theorem~\ref{thm:rnftame}, it suffices to show that the functions
$\kti$, $\kte$, $\kxi$, and $\kxe$ are 
bounded above by recursive functions.

We can write the functions 
\begin{eqnarray*}
\kti(n) &=& \max\{l(y) \mid y \in \cc_n\} \hspace{0.1in} \text{ and}\\
\kte(n) &=& \max\{d_X(\ep,x) \mid x \in P_n\} 
\end{eqnarray*}
where 
$\cc_n := \{ y_g \in \cc \mid g \in G \text{ and } d_X(\ep, g) \le n \}$
and $P_n$ is the set of prefixes of words in $\cc_n$.
Now for any prefix $x$ of a word $y \in \cc_n$, we have 
$l(x) \le l(y)$, and so we can also write
\begin{eqnarray*}
\kti(n) &=& \max\{l(x) \mid x \in P_n\}.
\end{eqnarray*}
Since distance in a van Kampen diagram $\dd$ always
gives an upper bound for distance, via the map $\pi_\dd$,
in the Cayley complex $X$,
then for all $n \in \N$,
we have $\kte(n) \le \kti(n)$.  
Moreover, $\idi(\dd)$ must always be an
upper bound for $\edi(\dd)$, and so
$\kxe(n) \le \kxi(n)$ for all $n$.  Thus it
suffices to find recursive upper bounds for $\kti$ and $\kxi$.

For each word $w$ over $A$, a \stkg\ reduction algorithm
for computing the the
associated word $y_w$ in $\cc$ was given in
Section~\ref{subsec:stackdef}.
The set of words 
$\cc_n$ 
is also the set 
$\cc_n=\{y_u \mid u \in \cup_{i=0}^n A^i\}$ of normal
forms for words of length up to $n$.  
By enumerating the finite set of words of length
at most $n$, computing their normal forms in $\cc$ with
the reduction algorithm, 
and taking the maximum word length that
occurs, we obtain $\kti(n)$.  Hence the function
$\kti$ is computable.

Given $w \in A^*$ and $a \in A$, we compute 
the two words $y_w$ and $y_{wa}$ and store them
in a set $L_e$.
Next we follow the definition of 
the \edg\ $\dd_e$
for the edge $e=e_{w,a}$
in the recursive construction of the \nff\ from
Section~\ref{subsec:stackdef}.
If $(w,a,a) \in \alg$,
then $e \in \dgd$ and we add no
other words to $L_e$.  On the other hand, if 
$(w,a,a) \notin \alg$,
then $e \in \ves$.  In the latter case,
by enumerating the finitely many words $x \in \hr$,
and checking whether or not $(w,a,x)$ lies in the computable
set $\alg$, 
we can compute the word $c(e)=x$.
Write $x=a_1 \cdots a_n$ with each $a_i \in A$.
For $1 \le i \le n$, we compute the normal forms $y_i$ in $\cc$
for the words $wa_1 \cdots a_i$, and add these words to
the set $L_e$.  For each pair $(y_{i-1},a_i)$, we determine 
the word $x_i$ such that
$(y_i,a_i,x_i) \in \alg$. 
If $x_i \neq a_i$, we write $x_i=b_1 \cdots b_m$ 
with each $b_j \in A$, and add the normal forms
for the words $y_{i-1}b_1 \cdots b_j$ to $L_e$ for each $j$.
Repeating this process through all of the
steps in the construction of $\dd_e$, we must, after finitely many
steps, have no more words to add to $L_e$. 
The set $L_e$ now contains the normal form
$y_{\pi_{\dd_e}(v)}$ for each vertex $v$ of the
diagram $\dd_e$.
Calculate $k(w,a):=\max\{l(y) \mid y \in L_e\}$.

Now as in
Remark~\ref{rmk:anfhasnf}, for 
each vertex $v$ of the \edg\ $\dd_e$ there is
a path in $\dd_e$ from the basepoint to $v$
labeled by a word in the set $L_e$.
Then $\idi(\dd_e) \le k(w,a)$, and we have
an algorithm to compute $k(w,a)$.

Now we can write $\kxi(n) \le k_r'(n)$ for all $n \in \N$, where
\begin{eqnarray*}
k_r'(n) &:=& \max \{k(w,a) \mid w \in \cup_{i=0}^n A^i, a \in A
\}.
\end{eqnarray*}
Repeating the computation of  $k(w,a)$
above for all words $w$ of length at most $n$
and all $a \in A$, we can compute
this upper bound $k_r'$ for $\kxi$, as required.
\end{proof}

\begin{remark}\label{rmk:quasiastk}
{\em Although the proof of Theorem~\ref{thm:astktf}
shows in the abstract that an algorithm must exist
to compute $k_r'(n)$, this proof does not give a
method to find this algorithm starting from the
computable set $\alg$.  In particular, although
every finite set is recursively enumerable,
it is not clear how
to enumerate the finite set $\hr$.  In practice, however,
for every example we will discuss, we start with both
a finite presentation $\langle A \mid R \rangle$
for the group $G$ and a \stkg\ that
(re)produces that presentation.  In that case, the set
$\hr$ must be contained in the finite set 
$R':=\{x \in A^* \mid \exists a \in A$ with $xa \in R\}$.  
Then we can replace
the enumeration of $\hr$ with an enumeration of $R'$, 
which can be computed from $R$.}
\end{remark}

\eject


\section{Examples of stackable groups, and their tame filling 
                   inequalities}\label{sec:examples}



\subsection{Groups admitting complete
     rewriting systems}\label{sec:rs}


$~$

\vspace{.1in}

Recall that a {\em finite complete rewriting system}
(finite {\em CRS}) for a group $G$ consists of a finite set $A$
and a finite set of rules $R \subseteq A^* \times A^*$
(with each $(u,v) \in R$ written $u \ra v$)
such that 
as a monoid, $G$ is presented by 
$G = Mon\langle A \mid u=v$ whenever $u \ra v \in R \rangle,$
and the rewritings
$xuy \ra xvy$ for all $x,y \in A^*$ and $u \ra v$ in $R$ satisfy:
\begin{itemize}
\item {\em Normal forms:} Each $g \in G$ is 
represented by exactly one {\em irreducible} word 
(i.e. word that cannot be rewritten)
over $A$.
\item {\em Termination:} The (strict) partial ordering
$x>y$ if $x \rightarrow
x_1 \rightarrow ... \rightarrow x_n \rightarrow y$ is 
well-founded.
($\not\exists$ infinite chain $w \ra x_1 \ra x_2 \ra \cdots$.)
\end{itemize}

Given any finite CRS $(A,R)$ for $G$, there is another
finite CRS $(A,R')$ for $G$ with the same
set of irreducible words such that the CRS is {\em minimal}.  
That is, for each $u \ra v$ in $R'$, the word $v$
and all proper subwords of the word $u$ 
are irreducible (see, for example, \cite[p.~56]{sims}).
Let $A'$ be the closure of $A$ under inversion.
For each letter $a \in A' \setminus A$,
there is an irreducible word $z_a \in A^*$ with $a =_G z_a$.
Let $R'':=R' \cup \{a \ra z_a \mid a \in A' \setminus A\}$.
Then  $(A',R'')$ is also a minimal finite CRS for $G$,
again with the same set of irreducible normal forms
as the original CRS $(A,R)$.
For the remainder of this paper, we will assume that
all of our complete rewriting systems are minimal
and have an inverse-closed alphabet.

For any complete rewriting system $(A,R)$, there is a natural
associated symmetrized group presentation 
$\langle A \mid R' \rangle$, where 
$R'$ is the closure of the relator set 
$\{uv^{-1} \mid u \ra v \in R\} \cup \{aa^{-1} \mid a \in A'\}$
under free reduction (except the empty word), 
inversion, and cyclic conjugation.

Given any word $w \in A^*$, we write
$w \ras w'$ if there is any sequence of rewritings
$w=w_0 \ra w_1 \ra \cdots \ra w_n=w'$ (including
the possibility that $n=0$ and $w'=w$).
A {\em prefix rewriting}
of $w$ with respect to the complete rewriting system $(A,R)$
is a sequence of rewritings $w=w_0 \ra \cdots \ra w_n=w'$,
written $w \prs w'$,
such that at each $w_i$,
the shortest possible reducible prefix is rewritten to obtain
$w_{i+1}$.  
When $w_n$ is irreducible,
the number $n$ is the {\em prefix rewriting length} of $w$,
denoted $prl(w)$.

In~\cite{hmeiertcacrs} Hermiller and Meier constructed a
\dia\  1-combing associated to a finite complete rewriting
system. 
In Theorem~\ref{thm:crsrecit}, we 
utilize
the analog of their construction 
to build a \stkg\ from a finite CRS   
(whose canonical 1-combing built from the
\rcnf\ 
is the one defined in \cite{hmeiertcacrs}).

\begin{theorem}\label{thm:crsrecit}
If the group $G$ admits a
finite complete rewriting system,
then 
$G$ is \rstkbl.
\end{theorem}

\begin{proof}
Let $\cc=\{y_g \mid g \in G\}$ be the set of irreducible words from 
a minimal
finite CRS $(A,R)$ for $G$, where $A$ is inverse-closed.
Then $\cc=A^* \setminus \cup_{u \ra v \in R}A^*uA^*$ is a regular language.
Let $\ga$ be the Cayley graph for the pair $(G,A)$.
Note that prefixes of irreducible
words are also irreducible, and so $\cc$ is a 
prefix-closed set of normal 
forms for $G$ over $A$.

As usual, whenever $e$ is a directed edge in $\ga$
with label $a$
and initial vertex $g$,  then $e$ lies in the set
$\dgd$ of degenerate edges if and only if
$y_gay_{ga}^{-1}$ freely reduces to the empty word,
and otherwise $e \in \ves$.


Given a directed edge $e \in \ves$ with
initial vertex $g$ and label $a$, the word
$y_ga$ is reducible (since this edge is not
in $\dgd$).
Since $y_g$ is
irreducible, the shortest reducible prefix of
$y_ga$ is the entire word.
Minimality of the rewriting system $R$ implies that
there is a unique factorization $y_g=w\tilde u$ 
such that $\tilde ua$ is the left
hand side of a unique rule $\tilde ua \ra v$ in $R$; that is,
$y_ga \ra wv$ is a prefix rewriting.
Define $c(e):=\tilde u^{-1}v$.

Property (S1) of the definition of \stkg\ is 
immediate.  To check property (S2), 
we first let $p$ be the path in $\ga$ that starts at $g$
and follows the word $c(e)$.
Since the word $\tilde u$ is a suffix $x_g$ of the normal form $y_g$,
then the edges in the path $p$
that correspond to the letters in 
$\tilde u^{-1}$ all lie in the set $\dgd$ of degenerate edges.  
Hence we can choose $\tc_e=v$.
For each directed edge $e'$ in the subpath of $p$ labeled by $v$,
either $e'$ also lies in $\dgd$, or else $e' \in \ves$, and
there is a factorization $v=v_1a'v_2$ so that
$e'$ is the directed edge along $p$ corresponding 
to the label $a' \in A$.
In the latter case, if we denote
the initial vertex of $e'$ by $g'$, then
the prefix rewriting sequence
from $y_{g'}a'v_2$ to its irreducible form
is a (proper) subsequence of the prefix
rewriting of $y_ga$.  
That is, if we define a function
$prl:\ves \ra \N$ by $prl(e):=prl(y_ga)$
whenever $e$ is an edge with initial vertex $g$ and
label $a$, we have $prl(e')<prl(e)$.
Hence the ordering $<_c$ corresponding to our
function $c:\ves \rightarrow A^*$ satisfies 
the property that $e'<_c e$ implies $prl(e')<prl(e)$,
and the well-ordering property on $\N$ implies that
$<_c$ is a well-founded strict partial ordering.
Thus (S2) holds as well.

The image set $c(\ves)$ is the set
of words $\hr=\{\tilde u^{-1}v \mid \exists a \in A$
with $\tilde u a \ra v$ in $R\}$.  Thus 
Property (S3) 
follows from finiteness of the set $R$
of rules in the rewriting system.
We now have a tuple $(\cc,c)$ of data satisfying
properties (S1-S3) of Definition~\ref{def:fstkg}, 
i.e., a \fstkg.
The \stkg\ presentation in this case is the
symmetrized presentation associated to the
rewriting system.

To determine whether a tuple $(w,a,x)$ lies in the 
associated set 
$\alg$, we begin by computing the normal forms
$y_w$ and $y_{wa}$ from $w$ and $wa$, using the
rewriting rules of our finite system.
Then $(w,a,x) \in \alg$ if and only if
either at least one of the words 
$y_wa$ and $y_{wa}a^{-1}$ is irreducible 
and $a=x$, or
else both of the words $y_wa$ and $y_{wa}a^{-1}$ are 
reducible and there exist both a factorization
$y_w=z\tilde u$ for some
$z \in A^*$ and a rule $\tilde ua \ra v$ in $R$ such that
$x=\tilde u^{-1}v$.  Since there are only
finite many rules in $R$ to check for such a
decomposition of $y_w$,
it follows that the set $\alg$ is 
also computable, and
so this \stkg\  is algorithmic.
\end{proof}

Theorem~\ref{thm:astktf} now shows that any group
with a finite complete rewriting system admits
intrinsic and extrinsic \tfs\ with respect to a
recursive function.
By relaxing the bounds on \tfs\  further, we can write 
bounds on 
filling inequalities in terms
of another important function in the study of rewriting systems.

\begin{definition}
The {\em string growth complexity} function $\gamma:\N \ra \N$
associated to a finite complete rewriting system $(A,R)$
is defined by 
\[ \gamma(n) := \max \{l(x) \mid \exists~w \in A^* \text{ with }
l(w) \le n \text { and } w \ras x\} \]
\end{definition}

This function $\gamma$ is an upper bound for the
intrinsic (and hence also extrinsic) diameter function
of the group $G$ presented by the rewriting system.  
In the following, we show that $G$ also satisfies \tfs\ with
respect to a function Lipschitz equivalent to $\gamma$.

\begin{corollary}\label{cor:rsgrowth}
Let $G$ be a group with a finite complete
rewriting system. 
Let $\gamma$ be the
string growth complexity function 
for the associated
minimal system and let $\maxr$ denote the length of 
the longest rewriting rule for this system.
Then $G$ 
satisfies both intrinsic and extrinsic \tfs\ 
for the recursive function 
$n \mapsto \gamma(\lceil n \rceil +\maxr+2)+1$.
\end{corollary}

\begin{proof}
Let $(A,R)$ be a minimal finite complete rewriting system
for $G$ such that $A$ is inverse-closed.
Let $(\cc,c)$ be the \stkg\ for $G$ constructed
in the proof of Theorem~\ref{thm:crsrecit}, and
let $X$ be the Cayley complex of the \stkg\ presentation $\pp$.
Let $\cle=\{\dd_e,\tht_e) \mid e \in E(X)\}$
be the associated \rcnf (where we note that the 
choice of subword $\tc_e$ of $c(e)$ for each $e \in \ves$,
used in the construction of $\tht_e$,
is given in the proof of Theorem~\ref{thm:crsrecit}).
For the rest of this proof, we rely heavily
on the result and
notation developed in the proof of Theorem~\ref{thm:astktf}
to obtain the tameness bounds for these \ehs.

From that proof, we have 
$\kte(n) \le \kti(n) = \max\{l(y) \mid y \in \cc_n\}$ for all $n$,
where $\cc_n$ is the set of irreducible normal forms obtained by
rewriting words over $A$ of length at most $n$.  Therefore
$\kte(n) \le \kti(n) \le \gamma(n)$.

Also from that earlier proof, we have $\kxe(n) \le \kxi(n) \le k_r'(n)$
for all $n \in \N$.  
Suppose that $w \in A^*$ is a word of length at most $n$, $a \in A$,
and $e=e_{w,a}$ is the directed edge in $X$ from $w$ to $wa$
labeled by $a$.
In this case we analyze the van Kampen diagram $\dd_e$ more carefully.
This diagram is built by successively
applying prefix rewritings to the word $y_w a$
and/or by applying free reductions (which must also result
from prefix rewritings).
Hence for every vertex $v$ in the diagram $\dd_e$,
there is a path from the basepoint $*$ to $v$ labeled by an
irreducible prefix $y$ of a word $x \in A^*$ such that 
$y_w a \prs x$,
and this word $y$ is the element of the set $L_e$ corresponding
to the vertex $v$.
Then the maximum $k(w,a)$ of the lengths of
the elements of $L_e$ is bounded above
by $\max \{l(y) \mid y$ is a prefix of $x$ and $y_w a \prs x\}$.
Since the length of a prefix of a word $x$ is at most
$l(x)$, we have
$k(w,a) \le \max \{l(x) \mid y_w a \prs x\}$.

Plugging this into the formula for $k_r'$, we obtain
\begin{eqnarray*}
k_r'(n) &=& \max\{k(w,a) \mid w \in \cup_{i=0}^n A^i, a \in A, e_{w,a} \in \ves\} \\
& \le &   \max \{l(x) \mid \exists~ w \in \cup_{i=0}^n A^i, a \in A, y_w a \prs x\}.
\end{eqnarray*}
Now for each word $w$ of length at most $n$ and each $a \in A$, we have 
$wa \prs y_wa$, and so

\hspace{-0.1in} $k_r'(n) \le \max \{l(x) \mid \exists~ w \in \cup_{i=0}^n A^i, 
   a \in A$ with 
  $w a \prs x\} \le  \gamma(n+1)$.

Putting these inequalities together, we obtain $\mue(n) \le \mui(n)$ and

\hspace{-.1in} $\mui(n) = \max \{\kti(\lceil n \rceil +  1) + 1, 
         n+1, \kxi(\lceil n \rceil + \maxr+1)\} $

\hspace{.3in} $ \le \gamma(\lceil n \rceil +\maxr+2)+1.$
\end{proof}

\begin{remark}
{\em We note that every instance of rewriting in the proofs
in this Section
was a prefix rewriting, and so $G$ also satisfies \tfs\ 
with $\gamma$ replaced by the potentially
smaller} prefix rewriting string growth complexity
{\em function $\gamma_p(n)= \max \{l(x) \mid \exists~w \in A^* \text{ with }
l(w) \le n \text { and } w \prs x\} \le \gamma(n)$.}
\end{remark}



\subsection{Thompson's group $F$}\label{subsec:f}


$~$

\vspace{.1in}

Thompson's group 
\[F=\langle x_0,x_1 \mid [x_0x_1^{-1},x_0^{-1}x_1x_0],
   [x_0x_1^{-1},x_0^{-2}x_1x_0^2] \rangle\]
is the group of orientation-preserving piecewise linear
homeomorphisms of the unit interval [0,1], satisfying that
each linear piece has a slope of the form $2^i$ for some
$i \in \Z$, and all breakpoints occur in the 2-adics.
In \cite{chst}, Cleary, Hermiller, Stein, and Taback
show that Thompson's group with the generating set
$A=\{x_0^{\pm 1},x_1^{\pm 1}\}$ is \fstkbl, 
with \stkg\ presentation given by 
the symmetrization of the presentation above. 
Moreover, in~\cite[Definition~4.3]{chst} they give an 
algorithm for computing the \stkg\ map, which 
can be shown to yield an \astkg\ for $F$.

Although we will not repeat their proof here, 
we describe 
the normal form set $\cc$ associated to the \stkg\ 
constructed for Thompson's group in~\cite{chst}
in order to
discuss its formal language theoretic properties.
Given a word $w$ over the generating set 
$A=\{x_0^{\pm 1},x_1^{\pm 1}\}$, denote the number of occurrences in $w$ of the letter
$x_0$ minus the number of occurrences in $w$ of the letter $x_0^{-1}$ by $expsum_{x_0}(w)$;
that is, the exponent sum for $x_0$.
The authors of that paper show (\cite[Observation~3.6(1)]{chst}) that the set 

\smallskip

$\cc:=\{w \in A^* \mid \forall \eta \in \{\pm 1\}$, the words
$x_0^\eta x_0^{-\eta}$, $ x_1^\eta x_1^{-\eta}$, and $x_0^2x_1^\eta$ 
are  not subwords of $w$, 

\hspace{1in} and $\forall $ prefixes $ w' $ of $ w,
expsum_{x_0}(w') \le 0\}$,

\smallskip


\noindent is a set of normal forms for $F$.  
Moreover, each of these words
labels a (6,0)-quasi-geodesic path in the
Cayley complex $X$~\cite[Theorem~3.7]{chst}.

This set $\cc$ is the intersection of the regular language
$A^* \setminus \cup_{u \in U}A^*uA^*$, where 
$U:=\{x_0x_0^{-1},x_0^{-1}x_0,x_1x_1^{-1},x_1^{-1}x_1,x_0^2x_1,x_0^2x_1^{-1}\}$,
with the language $L:=\{w \in A^* \mid \forall $ prefixes $ w' $ of $ w,
expsum_{x_0}(w') \le 0\}$.  
We refer the reader to the text of Hopcroft and Ullman~\cite{hu}
for definitions and results on context-free and regular languages we now
use to analyze the set $L$.
The language $L$ can be recognized by a deterministic push-down automaton (PDA)
which pushes an $x_0^{-1}$ onto its stack whenever
an $x_0^{-1}$ is read, and pops an $x_0^{-1}$ off of its stack whenever
an $x_0$ is read.  When $x_1^{\pm 1}$ is read, the PDA
does nothing to the stack, and does not change its state.
The PDA remains in its initial state unless an $x_0$ is read 
when the only symbol on 
the stack is the stack start symbol $Z_0$, in which case
the PDA transitions to a fail state (at which it must then
remain upon reading the remainder of the input word).  
Ultimately the PDA accepts
a word whenever its final state is its initial state.
Consequently, $L$ is a deterministic context-free language.  
Since the intersection of a regular language with a deterministic
context-free
language is deterministic context-free, 
the set $\cc$ is also a deterministic context-free language.

The authors of~\cite{chst} construct the
\stkg\ of $F$ as a stepping stone to showing
that $F$ with this presentation also
admits a radial \tci\ with
respect to a linear function.
We note that although the definition of \dia\ 1-combing
is not included in that paper, and the coarse distance
definition differs slightly, the constructions
of 1-combings in the proofs are \dia\  and
admit Lipschitz equivalent radial \tci\ functions.
Hence by Proposition~\ref{prop:etistc}, this group
satisfies a linear extrinsic \tfi.


Let $\cle$
be the \rcnf\ associated to the \stkg\ 
in ~\cite{chst}, and 
let ${\mathcal D}=\{(\dd_w,\ph_w) \mid w \in A^*, w=_F \ep\}$ 
be the 
\cfl\ 
induced by $\cle$ by the seashell procedure.
As noted above, 
the van Kampen homotopies in the collection ${\mathcal D}$ are
extrinsically $f$-tame for a linear function $f$.
A consequence of Remark~\ref{rmk:anfhasnf} and the
seashell construction is that for 
each word $w \in A^*$ with $w=_F \ep$ and for each vertex
$v$ in $\dd_w$, there is a path in $\dd_w$ from the
basepoint $*$ to the vertex $v$ labeled by the (6,0)-quasi-geodesic
normal form in $\cc$ representing $\pi_{\dd_w}(v)$. 
Then we have 
$d_{\dd_w}(*,v)
\le 6d_X(\ep,\pi_{\dd_w}(v))$.
Let $\tj:\nn \ra \nn$ be the (linear) function
defined by $\tj(n)=6 \lceil n \rceil + 1$.
Theorem~\ref{thm:itversuset} then shows that 
Thompson's group $F$ also satisfies a linear 
intrinsic \tfi, for the linear function $\tj \circ f$.

On the other hand, we note that Cleary and Taback~\cite{clearytaback}
have shown that Thompson's group $F$ is not almost convex
(in fact, Belk and Bux~\cite{belkbux}
have shown that $F$ is
not even minimally almost convex).  Combining this
with Theorem~\ref{thm:aceti} below, Thompson's group $F$ cannot
satisfy an intrinsic or extrinsic 
\tfi\ for the identity function.


\subsection{Iterated Baumslag-Solitar groups}\label{subsec:iteratedbs}


$~$

\vspace{.1in}

The iterated
Baumslag-Solitar group
$$G_k=\langle a_0,a_1,...,a_k \mid a_i^{a_{i+1}}=a_i^2; 0 \le i \le k-1\rangle$$
admits a finite complete rewriting
system for each $k \ge 1$ (first described by Gersten;
see \cite{hmeiermeastame} for details),
and so Theorem~\ref{thm:crsrecit}
shows that this group is \rstkbl.

Gersten~\cite[Section~6]{gerstenexpid} showed that $G_k$
has an isoperimetric function that grows at least as fast as a tower 
of exponentials
$$
E_k(n) := \underbrace{2^{2^{.^{.^{.^{2^n}}}}}}_{\hbox{k times}}~.
$$
It follows from his proof that the (minimal) extrinsic diameter function 
for this group is at least $O(E_{k-1}(n))$.
Hence this is also a lower bound for the (minimal) intrinsic diameter
function for this group.  
Then by Proposition~\ref{prop:itimpliesid}, $G_k$ 
cannot satisfy an intrinsic or extrinsic \tfi\ for the function 
$E_{k-2}$.  (In the extrinsic case, this was shown in the
context of tame combings in~\cite{hmeiermeastame}.)
Combining this with Corollary~\ref{cor:rsgrowth}, for $k \ge 2$
the  group $G_k$ is an example of a \rstkbl\ group which admits intrinsic
and extrinsic recursive \tfs\ but which cannot satisfy a \tfi\ 
for $E_{k-2}$.  



\subsection{Solvable Baumslag-Solitar groups}\label{subsec:bs}


$~$

\vspace{.1in}

The solvable Baumslag-Solitar groups are presented by  
$G=BS(1,p)=\langle a,t \mid tat^{-1}=a^p \rangle$ with $p \in \Z$.
In~\cite{chst} Cleary, Hermiller, Stein, and Taback 
show that for $p \ge 3$, the groups $BS(1,p)$ admit
a linear radial \tci, 
and hence (from Proposition~\ref{prop:etistc})
a linear extrinsic \tfi.

We note that the \cfl\ in their proof is induced by
the \rcnf\ associated to a \rstkg, which we 
describe here in order to obtain an intrinsic \tfi\ for
these groups.
The set of normal forms over the generating set
$A=\{a,a^{-1},t,t^{-1}\}$ is
$$\cc:=\{t^{-i}a^mt^k \mid i,k \in \N \cup \{0\}, m \in \Z,
\text{ and either } p \not| m \text{ or } 0 \in \{i,k\}\}.$$
The recursive edges in $\ves = \dire \setminus \dgd$
are the directed edges of the form $e_{w,b}$
with initial point $w$ and label $b \in A$
satisfying one of the following:
\begin{enumerate}
\item $w=t^{-i}a^m$ and $b=t^\eta$ with $m \neq 0$, $\eta \in \{\pm 1\}$,
and $-i+\eta \le 0$, or
\item $w=t^{-i}a^mt^k$ and $b=a^\eta$ with $k>0$ and $\eta \in \{\pm 1\}$.
\end{enumerate}
In case (1), we define
$c(e_{t^{-i}a^m,t^{\eta}}):=(a^{-\nu p}ta^{\nu})^\eta$, 
where $\nu:=\frac{m}{|m|}$ is 1 if $m>0$ and $-1$ if $m<0$.
In case (2) we define
$c(e_{t^{-i}a^mt^k,a^{\eta}}):=t^{-1}a^{\eta p}t$.

Properties (S1) and (S3) of the definition of \stkg\ 
follow directly.  To show that the pair \stkg\ map $c$
also satisfies property (S2), we first briefly describe the 
Cayley complex $X$ for the finite presentation above; see
for example \cite[Section~7.4]{echlpt} for more details.
The Cayley complex $X$ is homeomorphic to the product
$\R \times T$ of the real line with a regular tree $T$,
and there are canonical projections 
$\Pi_\R:X \ra \R$ and $\Pi_T:X \ra T$.
The projection $\Pi_T$ takes each edge labeled by 
an $a^{\pm 1}$ to a vertex of $T$.  Each edge of $T$ is
the image of infinitely many $t$ edges of $X^1$, with
consistent orientation, and so we may consider the
edges of $T$ to be oriented and labeled by $t$, as well.
For the normal form $y_g=t^{-i}a^mt^k \in \cc$ of an element
$g \in G$, the projection onto $T$ of the path in $X^1$ starting at
$\ep$ and labeled by $y_g$ is the unique geodesic
path, labeled by $t^{-i}t^k$,
in the tree $T$ from $\Pi_T(\ep)$ to $\Pi_T(g)$. 
For any directed edge $e$ in 
$\ves$ in case (2) above, 
there are
$p+1$ 2-cells in the Cayley complex $X$ 
that contain $e$ in their boundary,
and the path $c(e)$ starting from 
the initial vertex of $e$ is the portion 
of the boundary, disjoint from $e$, 
of the only one of those 2-cells $\sigma$ that satisfies 
$d_T(\Pi_T(\ep),\Pi_T(q)) \le d_T(\Pi_T(\ep),\Pi_T(e))$
for all points $q \in \sigma$, where $d_T$
is the path metric in $T$.  For any edge $e'$ that
lies both in this $c(e)$ path and in $\ves$, then
$e'$ is again a recursive edge of type (2), and we
have $d_T(\Pi_T(\ep),\Pi_T(e'))<d_T(\Pi_T(\ep),\Pi_T(e))$.  
Thus the well-ordering on 
$\N$ applies, to show that there are
at most finitely many $e'' \in \ves$ with $e'' <_c e$
in case (2).

The other projection map $\Pi_\R$ takes each vertex $t^{-i}a^mt^k$
to the real number $p^{-i}m$, and so takes each edge
labeled by $t^{\pm 1}$ to a single real number, and takes
each edge labeled $a^{\pm 1}$ to an interval in $\R$.
For an edge $e \in \ves$ in case (1) above, there are
exactly two 2-cells in $X$ containing $e$, and
the path $c(e)$ starting at the initial vertex $w=t^{-i}a^m$ of $e$
travels around the boundary of the one of these
two cells (except for the edge $e$)
whose image, under the projection $\Pi_\R$,
is closest to $0$.  The only possibly recursive edge 
$e'$ in this $c(e)$
path must also have type (1), and moreover
the initial vertex of $e'$ is $w'=t^{-i}a^{m-\nu}$ and
satisfies $|\Pi_\R(w')|=|\Pi_\R(w)|-p^{-i}$.
Then in case (1) also there are only finitely many recursive edges
that are $<_c e$, completing the proof of property (S2).

Therefore the tuple $(\cc,c)$ is a \stkg, 
and the symmetrization of the
presentation above is the \stkg\ presentation.
The canonical diagrammatic 1-combing
built from the associated \rcnf\ is the 1-combing constructed
in~\cite{chst}.

Let ${\mathcal D}=\{(\dd_w,\ph_w) \mid w \in A^*, w=_F \ep\}$ 
be the 
\cfl\ 
induced by this \rcnf\ via the seashell procedure.
From Remark~\ref{rmk:anfhasnf}, we know that
for each vertex $v$ of a van Kampen diagram
$\dd_w$ in this collection, there is a path in 
$\dd_w$ from $*$ to $v$ labeled by the normal form
of the element $\pi_{\dd_w}(v)$ of $BS(1,p)$.
The normal form $y_g$ of $g \in G$ can be obtained from
a geodesic representative by applying (the infinite set of)
rewriting rules of the form $ta^\eta \ra a^{\eta p}t$ and
$a^\eta t^{-1} \ra t^{-1}a^{\eta p}$ for $\eta=\pm 1$
together with $t^{-1}a^{pm}t \ra a^m$ for $m \in \Z$
and free reductions.  
Then $d_{\dd_w}(*,v) \le l(y_g)
\le j(d_X(\ep,\pi_{\dd_w}(v))$
for the function $j:\N \ra \N$
given by $j(n)=p^n$.  Theorem~\ref{thm:itversuset} 
and the linear extrinsic \tfi\ result above now
apply, to show that the group $BS(1,p)$ with $p \ge 3$ also satisfies 
an intrinsic \tfi\ with respect to a function $\nn \ra \nn$
that is Lipschitz equivalent to the exponential function 
$n \mapsto p^n$ with base $p$.



\subsection{Almost convex groups}\label{subsec:ac}


$~$

\vspace{.1in}

One of the original motivations for the definition
of a radial tame combing inequality in~\cite{hmeiermeastame} 
was to capture Cannon's~\cite{cannon} notion of
almost convexity in a quasi-isometry invariant
property.  Let $G$ be a group with an
inverse-closed generating set $A$, and
let $d_\ga$ be the path metric on the associated
Cayley graph $\ga$.  For $n \in \N$, 
define the sphere $S(n)$ of radius $n$
to be the set of points in $\ga$ a distance
exactly $n$ from the vertex labeled by the
identity $\ep$.  Recall that the ball $B(n)$ of radius $n$
is the set of points in $\ga$ whose path metric distance
to $\ep$ is less than or equal to $n$.

\begin{definition}\label{def:ac}
A group $G$ is {\em almost convex} with respect
to the finite symmetric
generating set $A$ if there is a constant $k$
such that
for all $n \in \N$ and for all $g,h$ in the
sphere $S(n)$ satisfying
$d_\ga(g,h) \leq 2$
(in the corresponding Cayley
graph), there is a path inside 
the ball $B(n)$ from
$g$ to $h$ of length no more than $k$.
\end{definition}

Cannon~\cite{cannon} showed that every group 
satisfying an almost convexity condition
over a finite generating set is also finitely presented.
Thiel~\cite{thiel} showed that almost convexity
is a property that depends upon the finite
generating set used.  

In Theorem~\ref{thm:aceti}, we show that 
a pair $(G,A)$ that is almost convex 
is \afstkbl\ and (applying Theorem~\ref{thm:itisqi}) must also
lie in the quasi-isometry invariant 
class of groups admitting linear intrinsic
and extrinsic \tfs.  Moreover almost convexity
of $(G,A)$ is exactly characterized by 
admitting a finite set $R$ of
defining relations for $G$ over $A$ such that
the pair $(G,\langle A \mid R \rangle)$ satisfies an intrinsic
or extrinsic \tfi\ 
with respect to the identity function 
$\iota:\nn \ra \nn$ (i.e. $\iota(n)=n$ for all $n$).
In the extrinsic case, equivalence of almost
convexity and an extrinsic \cni\ for $\iota$
follows directly from
the equivalence of 
almost convexity with a radial \tci\ for the 
identity shown by Hermiller and Meier
in~\cite[Theorem C]{hmeiermeastame}, 
together with Proposition~\ref{prop:etistc}.
We give some details here which
include a description of the \stkg\  
involved, 
and a minor correction to the proof in that
earlier paper.

\begin{theorem} \label{thm:aceti}
Let $G$ be a group with finite generating set $A$, and
let $\iota: \nn \ra \nn$ denote the identity
function.  The following
are equivalent:
\begin{enumerate}
\item The pair $(G,A)$ is almost convex 
\item There is a finite
presentation $\pp=\langle A \mid R \rangle$ for $G$ that
satisfies an intrinsic \tfi\ with respect to $\iota$.
\item There is a finite
presentation $\pp=\langle A \mid R \rangle$ for $G$ that
satisfies an extrinsic \tfi\ with respect to $\iota$.
\end{enumerate}
Moreover, if any of these hold, then $G$ is \afstkbl\ over $A$.
\end{theorem}

\begin{proof}
Suppose that the group $G$ has finite symmetric
generating set $A$, and let $\ga$ be the corresponding
Cayley graph.  

\smallskip

\noindent{\em Almost convex $\Rightarrow$ \afstkbl:}

Suppose that the group $G$ 
is almost convex with respect to $A$, 
with an almost convexity constant $k$.
Let $\cc=\{z_g \mid g \in G\}$ be the 
set of shortlex  normal forms over $A$ for $G$. 
Let $\dgd$ be the corresponding set of
directed degenerate edges, and let
$\ves=\vec E(\ga) \setminus \dgd$ be
the set of recursive edges.

Let $e$ be any element of $\ves$ and
suppose that $e$ is oriented from endpoint $g$
to endpoint $h$.
If $\td_\ga(\ep,g)=\td_\ga(\ep,h)=n$, then 
the points $g$ and
$h$ lie in the same sphere.  Almost convexity of
$(G,A)$ 
implies that there is a directed edge path in $X$ from $g$ to $h$
of length at most $k$ 
that lies in the ball $B(n)$.  
We define $\tc_e=c(e)$ to be the shortlex least word over $A$
that labels a path in $B(n)$ from $g$ to $h$.
If $\td_\ga(\ep,g)=n$ and $\td_\ga(\ep,h)=n+1$,
then 
we can write $z_h =_{A^*} z_{h'}b$ for some $h' \in G$
and $b \in A$.
Hence $g,h' \in S(n)$ and $d_\ga(g,h') \le 2$.  
Again in this case we 
define $\tc_e$ to be the shortlex least
word over $A$ that labels 
a path in $X$
of length at most $k$ inside of the ball $B(n)$
from $g$ to $h'$.  The almost convexity property
shows that 
the word $c(e):=\tc(e)b$ has length at 
most $k+1$, this word labels a path from $g$ to $h$,
and $c(e)$ decomposes as the
word $\tc(e)$ followed by a suffix $x_h=b$ of $z_h$.
Similarly,
if $\td_\ga(\ep,g)=n+1$ and $\td_\ga(\ep,h)=n$, then
$z_g=z_{g'}b$ for some $b \in A$ and $g' \in G$, and we define
$\tc_e$ to be the shortlex least word labeling a path
in $B(n)$ from $g'$ to $h$.  Then $c(e):=b^{-1}\tc_e$
labels a path from $g$ to $h$, and decomposes as
a prefix $b^{-1}$, that is the inverse of a suffix
$x_g=b$ of $z_g$, followed by $\tc_e$.

In each of these three cases, for any point
$p$ in the interior of $e$, we have
$\td_\ga(\ep,p)=n+\frac{1}{2}$.
For any directed edge $e'$ that lies both
in $\ves$ and in the path of $\ga$
starting at $g$ and labeled by $c(e)$, the
edge $e'$ must lie in the subpath
labeled by $\tc(e)$, and hence $e'$ is contained
in $B(n)$. 
Therefore any point $p'$ in the
interior of $e'$ must satisfy 
$\td_X(\ep,p') \le n-\frac{1}{2} < \td_X(\ep,p)$.
That is, if we define the function $f_\ga:\vec E(\ga) \rightarrow \nn$
by $f_\ga(u):=\td_\ga(\ep,q)$ for any (and hence all) $q \in Int(u)$,
we have that $e'<_c e$ implies $f_\ga(e')<f_\ga(e)$ in the
standard well-ordering on $\nn$.  Hence the relation $<_c$  
is a well-founded strict partial ordering.  Properties
(S1) and (S2) of Definition~\ref{def:fstkg} hold for the function $c$.

The image set $c(\ves)$ of this function $c$
is contained in the finite set of all nonempty
words over $A$ of length up to $k+1$ that represent
the identity element of $G$, and so (S3) also holds.
We now have that the tuple $(\cc,c)$ is a \fstkg.  

We are left with showing computability
for the set $\alg$ defined by
$\alg=\{(w,a,x) \mid c'(\wa)=x\} \subset A^* \times A \times 
A^*$ 
where $\wa$ denotes
the edge in $\ga$ from $w$ to $wa$, and 
$c'(\wa)=c(\wa)$ for all $\wa \in \ves$,
and $c'(\wa)=a$ for all $\wa \in \dgd$.
Suppose that $(w,a,x)$ is any element of 
$A^* \times A \times A^*$.
Cannon~\cite[Theorem~1.4]{cannon} has shown
that the word problem is solvable for $G$, and
so by enumerating the words in $A^*$ in increasing
shortlex order, and checking whether each in turn
is equal in $G$ to $w$, we can find the shortlex
normal form $z_w$ for $w$.  Similarly we compute $z_{wa}$.  
If the word $z_w a z_{wa}^{-1}$ freely reduces
to $1$, then the tuple $(w,a,x)$ lies in $\alg$ 
if and only if $x=a$.

Suppose on the other hand that the word
$z_w a z_{wa}^{-1}$ does
not freely reduce to $1$.  
If $l(z_w)=l(z_{wa})$ is the natural number $n$, then
we enumerate the
elements of the finite set $\cup_{i=0}^n A^i$
of words of length up to $n$ in increasing
shortlex order.  For each word $y=a_1 \cdots a_m$
in this enumeration, with each $a_i \in A$,
we use the word problem solution again 
to compute the word length $l_{y,i}$
of the normal form $z_{wa_1 \cdots a_i}$ for
each $0 \le i \le m$.
If each $l_{y,i} \le n$, and equalities $l_{y,i} = n$
do not hold for two consecutive indices $i$, then
$(w,a,y)$ lies in $\alg$ and we halt the enumeration; otherwise, we go on to 
check the next word in our enumeration.
The tuple $(w,a,x)$ lies in $\alg$ if and only if $x$
is the unique word $y$ that results when this
algorithm stops.  The cases that $l(z_w)=l(z_{wa}) \pm 1$ are
similar.  

Combining the algorithms in the previous two paragraphs,
we have that the set $\alg$ is computable.

\noindent {\em (1) implies (3)}: 
 
Suppose that the group $G$ 
is almost convex with respect to $A$, 
with almost convexity constant $k$.
Let $(\cc,c)$ be the \stkg\ obtained above,
and let $X$ be the Cayley complex for the
\stkg\ presentation $\pp=\langle A \mid R_c\rangle$, 
with 1-skeleton $\xx=\ga$.
Let $\cle=\{(\dd_e,\tht_e) \mid e \in E(X) \}$
be the set of {\edg}s and \ehs\ 
from the associated \rcnf.

Theorem~\ref{thm:rnftame} can now be applied, but
unfortunately this result is insufficient. Although
the fact that all of the normal forms in $\cc$
are geodesic implies that the functions $\kti$ and $\kxi$
are the identity, the \tfi\ bounds
$\mui$ and $\mue$ are not.
Instead, we follow the 
steps of the algorithm that built the \rcnf\ more carefully.

Let $e$ be any edge of $X$,
again with endpoints $g$ and $h$,
and let $n:=\min\{d_X(\ep,g),d_X(\ep,h)\}$;
that is, either $g,h \in S(n)$, or one of these
points lies in $S(n)$ and the other is in $S(n+1)$.
Let $\he$ 
be the edge corresponding to $e$ in the
van Kampen diagram $\dd_e$, and let $p$ be an
arbitrary point in $\he$.

{\em Case I.  Suppose that $e \in \ugd$.}
Then $\dd_e$ is a line segment with no 2-cells,
and the path $\pi_{\dd_e} \circ \tht_e(p,\cdot)$ follows a
geodesic in $\xx$.  
Hence this path is extrinsically $\iota$-tame.

{\em Case II.  Suppose that $e \in E(X) \setminus \ugd$.}
We prove this case by Noetherian induction.
By construction, the paths $\pi_{\dd_e} \circ \tht_e(\hat g, \cdot)$
and $\pi_{\dd_e} \circ \tht_e(\hat h,\cdot)$ follow the geodesic
paths in $X$ starting from $\ep$ and
labeled by the words $y_g$ and $y_h$ 
at constant speed.

Suppose that $p$ is a point in the interior of $\he$.
We follow the notation of the recursive construction of $\tht_e$
in Section~\ref{sec:rnf}.  
In that construction, \ehs\ are constructed for directed
edges; by slight abuse of notation,
let $e$ also denote the directed edge from $g$ to $h$
that yields the element $(\dd_e,\tht_e)$ of $\cle$.
Recall that this recursive procedure utilizes a factorization of the 
word $c(e)$ as $c(e)=x_g \tc_e x_h$.  In our
definition of $c(e)$ above, we defined this factorization
so that for each edge $e'$ (no matter whether $e'$
is in $\dgd$ or $\ves$)
in the $\tc_e$ path, we have
$f_\ga(e')<f_\ga(e)$.
On the interval $[0,a_p]$,
the path $\tht_e(p,\cdot)$ follows a
path $\tht_i(\Xi_e(p,0),\cdot)$ in a subdiagram of
$\dd_e$ that is either  an \ehy\  for an edge $e_i$
of $X$ that lies in this $\tc_e$ subpath, or
a line segment labeled by a shortlex normal form.  
Hence either by induction or case I,
the homotopy $\tht_i$ is extrinsically $\iota$-tame.

On the interval $[a_p,1]$, the path 
$\tht_e(p,\cdot)$ follows the path $\Xi_e(p,\cdot)$
from the point $\Xi_e(p,0)$ (in the subpath of $\bo \hc_e$
labeled $\tc_e$, whose image in $X$
is contained in $B(n)$) through the interior of
the 2-cell $\hc_e$ of $\dd_e$ to the point $p$.
We have
$\td_X(\ep,\pi_{\dd_e}(\Xi_e(p,0))) \le n$, 
$\td_X(\ep,\pi_{\dd_e}(\Xi_e(p,t))) = n+\frac{1}{4}$
for all $t \in (0,1)$, 
and $\td_X(\ep,\pi_{\dd_e}(\Xi_e(p,1))) = 
\td_X(\ep,p)=f_\ga(e)=n+\frac{1}{2}$.
Hence the path $\Xi_e(p,\cdot)$ is extrinsically $\iota$-tame.
Putting these pieces together, we have that $\tht_e$
is also extrinsically $\iota$-tame in Case II. 

Thus in the \rcnf\  
($\cc$, $\cle$),
each \ehy\  is extrinsically $\iota$-tame, and 
hence the same is true for the van Kampen
homotopies of the \rcf\ $(\cc,\cld)$ induced
by $\cle$, by
the proof of (4) $\Rightarrow$ (1) in 
Proposition~\ref{prop:etistc}. Therefore $(G,\pp)$
satisfies an extrinsic \tfi\ with respect to the same
function $\iota$.

{\em (1) implies (2)}: 
As noted in Remark~\ref{rmk:anfhasnf}, the
\rcf\ 
constructed above from the almost convexity
condition satisfies the 
property that 
for every vertex $v$ in a van Kampen diagram $\dd$ of $\cld$, there is
a path in $\dd$ from $*$ to $v$ labeled by the
shortlex normal form for the element $\pi_{\dd}(v)$
of $G$.  Since these normal forms label geodesics in $X$,
it follows that intrinsic and extrinsic distances (to the basepoints)
in the diagrams $\dd$ of $\cld$ are the same.
Thus the pair $(G,\pp)$ 
satisfies an intrinsic \tfi\ with respect to the same
function $\iota$.


{\em (2) or (3) implies (1)}: 
The proof of this direction in the extrinsic case
closely follows
the proof of~\cite[Theorem~C]{hmeiermeastame}, and
the proof in the intrinsic case is quite similar.
\end{proof}

\begin{remark} {\em 
As in Remark~\ref{rmk:quasiastk}, Cannon's word 
problem algorithm for almost convex groups, which we
applied in the proof of Theorem~\ref{thm:aceti},
requires the use of an enumeration of a
finite set of words over $A$, namely
those that represent $\ep$ in $G$ and have length at most 
$k+2$.  As Cannon
also points out~\cite[p.~199]{cannon}, although this
set is indeed recursive, there may not be an algorithm
to find this set, starting from $(G,A)$ and the constant $k$.}
\end{remark}

Since every word hyperbolic group is almost convex,
and the set of shortlex normal forms (used in the
proof of Theorem~\ref{thm:aceti} to construct a \stkg\ 
for any word hyperbolic group) is a regular language,
we have shown that every
word hyperbolic group is \rstkbl.
Combining Theorem~\ref{thm:crsrecit} with a result
of Hermiller and Shapiro~\cite{hs}, that
the fundamental group of every closed 3-manifold with a
uniform geometry other than hyperbolic must have a
finite complete rewriting system, shows
that these groups are \rstkbl\ as well.  
Hence we obtain the following.

\begin{corollary}\label{cor:3mfd}
If $G$ is the fundamental group of a closed 3-manifold with
a uniform geometry, then $G$ is \rstkbl. 
\end{corollary}


\section{Groups with a fellow traveler property and their
    tame filling inequalities}\label{sec:combable}


In this Section, we consider a class of finitely presented groups which
admit a rather different procedure for constructing van Kampen diagrams.
Let $G$ be a group with a finite
inverse-closed generating set $A$ such that no element
of $A$ represents the identity $\ep$ of $G$,
and let $\ga$ be the Cayley graph of $G$ over $A$.
We also assume that $G$ admits a set $\cc=\{y_g \mid g \in G\}$ 
of simple word normal forms over $A$ for $G$ that satisfies a (synchronous)
{\em $K$-fellow traveler property}.  That is, there is a constant
$K \ge 1$ such that whenever $g,h \in G$ and $a \in A$ with $ga=_G h$, and
we write $y_g=a_1 \cdots a_m$ and $y_h=b_1 \cdots b_n$ with
each $a_i,b_i \in A$
(where, without loss of generality, we assume $m \le n$),
then for all $1 \le i \le m$ we have
$d_\ga(a_1 \cdots a_i,b_1 \cdots b_i) \le K$, and
for all $m< i \le n$ we have $d_X(g,b_1 \cdots b_i) \le K$.

For each $m<i \le n$, let $a_i$ denote the empty word.
Let $c_0$ denote the empty word, let $c_n:=a$, and  
for each $1 \le i \le n-1$, let $c_i$ be a word in $A^*$
labeling a geodesic path in $\ga$ from $a_1 \cdots a_i$ to
$b_1 \cdots b_i$.
Thus each $c_i$ has length at most $K$.

This fellow traveler property implies that the set $R$ of
nonempty words over $A$ of length up to $2K+2$ that represent
the trivial element is a set
of defining relators for $G$.  
Let $\pp=\langle A \mid R \rangle$ be the
symmetrized presentation for $G$, and let
$X$ be the Cayley complex.

A van Kampen diagram $\dd_e$ for the word $y_gay_h^{-1}$
corresponding to the edge $e$ labeled $a$ from $g$ to $h$ in $X$
is built by successively gluing 2-cells 
labeled $c_{i-1}a_ic_i^{-1}b_{i-1}$, for $1 \le i \le n$,
along their common $c_i$ boundaries.  Then the
diagram $\dd_e$ is ``thin'', in that it has only the
width of (at most) one 2-cell.
An \ehy\  $\tht_e$ for this diagram can be constructed to go
successively through each 2-cell in turn
from the basepoint $*$ to the edge $\he$ corresponding to $e$; 
see Figure~\ref{fig:ladder} for an illustration.
\begin{figure}
\begin{center}
\includegraphics[width=4.8in,height=1.2in]{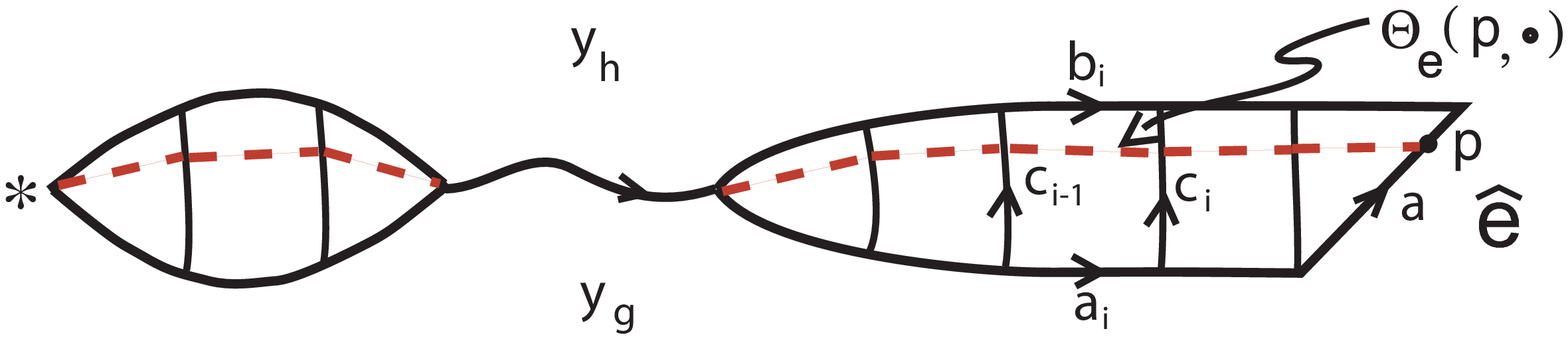}
\caption{``Thin'' van Kampen diagram $\dd_e$}\label{fig:ladder}
\end{center}
\end{figure}
Let $\cle=\{(\dd_e,\tht_e)\}$ be the collection of these 
{\edg}s and \ehs; the pair $(\cc,\cle)$ is a
\cnf.

\begin{proposition}\label{prop:combing}
Let $G$ be a group with a finite generating set $A$ and 
Cayley graph $\ga$.  If $G$ has a
set $\cc$ of simple word normal forms with a $K$-fellow traveler
property such that the set
$$S_n:=\{w \in A^* \mid d_\ga(\ep,w) \le n \text{ and } 
    w \text{ is a prefix of a word in }\cc\}$$
is a finite set for all $n \in \N$, then
$G$ satisfies both intrinsic and extrinsic \tfs\ for
finite-valued functions.
\end{proposition}

\begin{proof}
We utilize the finite presentation $\pp$ for $G$, with Cayley
complex $X$, and the \cnf\ $(\cc,\cle)$ constructed above.
Let $\cld=\{(\dd_w,\ps_w) \mid w \in A^*, w=_G\ep\}$ be the 
\cfl\ obtained from $\cle$ 
using the seashell procedure.
Also let $\dd_w$ be any of the diagrams in $\cld$, let 
$p$ be any point in $\bo \dd_w$, and  let $0 \le s<t\le 1$.

Let $\he$ be an edge of $\bo \dd_w$ containing $p$
(where $p$ may be in the interior or an endpoint).
Then the path $\ps_w(p,\cdot)$ lies in a 
subdiagram $\dd_e$ of $\dd_w$ such that $\dd_e$ is
the diagram in $\cle$ corresponding to the edge $e=\pi_{\dd_w}(\he)$
of $X$, and $\ps_w(p,\cdot)=\tht_e(p,\cdot)$.  
Let $\hat g$ be an endpoint of $\he$, with
$g=\pi_{\dd_e}(\hat g) \in G$ an endpoint of $e$. 
Let 
$y_g$ be the normal form of $g$ in $\cc$.

Applying the ``thinness'' of $\dd_e$, there is a 
path labeled $y_g$ in $\bo \dd_e$ starting at the 
basepoint, and every point of $\dd_e$ lies in some closed cell
of $\dd_e$ that also contains a vertex in this boundary path.
In particular, there are vertices $v_s$ and $v_t$ on the boundary path
$y_g$ of $\dd_e$ such that the point $\ps_w(p,s)=\tht_e(p,s)$ and 
the point $v_s$ occupy
the same closed 0, 1, or 2-cell in $\dd_e$
(and hence also in $\dd_w$), $\ps_w(p,t)=\tht_e(p,t)$ and $v_t$ occupy
a common closed cell, and $v_s$ occurs before
(i.e., closer to the basepoint) or at $v_t$ along
the $y_g$ path.  As usual let $\maxr \le 2K+2$ denote the length
of the longest relator in the presentation $\pp$. 
Then we have
$|\td_{\dd_w}(*,\ps_w(p,s))-\td_{\dd_w}(*,v_s)| \le \maxr+1$ and
$|\td_{X}(\ep,\pi_{\dd_w}(\ps_w(p,s)))-\td_{X}(\ep,\pi_{\dd_w}(v_s))| 
\le \maxr+1$, and similarly for the pair $\ps_w(p,t)$ and $v_t$.
Write the word $y_g=y_1y_2y_3$ where the vertex $v_s$
occurs on the $y_g$ path in $\bo \dd_e \subset \dd_w$ between the $y_1$ and
$y_2$ subwords, and the vertex $v_t$ between the $y_2$ and $y_3$
subwords.  Note that 
$y_1y_2$ is a
prefix of a normal form word in $\cc$,
and so satisfies 
$y_1y_2 \in S_{d_X(\ep,\pi_{\dd_w}(v_t))}$.

Define the function $t^i:\N \ra \N$ by
$t^i(n):=\max\{l(w) \mid w \in S_n\}$.  Since each $|S_n|$ is 
finite, this function is finite-valued.
Utilizing the fact that $t^i$ is a nondecreasing function, we have
\begin{eqnarray*}
\td_{\dd_w}(*,\ps_w(p,s))
&\le&   \td_{\dd_w}(*,v_s)+ \maxr+1 
\text{\ \ \ \ }\le \text{\ \ \ \ }    l(y_1)+\maxr+1 \\
&\le&    l(y_1y_2)+\maxr+1 
\text{\ \ \ \ } \le \text{\ \ \ \ } t^i(d_X(\ep,\pi_{\dd_w}(v_t)))+\maxr+1 \\
&\le&    t^i(\td_{\dd_w}(*,v_t))+\maxr+1 \\
& \le & t^i(\lceil \td_{\dd_w}(*,\ps_w(p,t))\rceil+\maxr+1)+\maxr+1
\end{eqnarray*}
Then $G$ satisfies an intrinsic \tfi\ for the
function $n \ra t^i(\lceil n \rceil+2K+3)+2K+3$.

Next define the function $t^e:\N \ra \N$ by

\hspace{.5in}
  $t^e(n):=\max\{d_X(\ep,v) \mid v \text{ is a prefix of a word in } S_n\}$.

\noindent Again, this is a finite-valued nondecreasing function.
In this case, we note that since $y_1$ is a prefix of $y_1y_2$,
then $y_1$ is a prefix of a word in $S_{d_X(\ep,\pi_{\dd_w}(v_t))}$.
Then
\begin{eqnarray*}
\td_{X}(\ep,\pi_{\dd_w}(\ps_w(p,s)))
&\le&   \td_{X}(\ep,\pi_{\dd_w}(v_s))+\maxr+1 
\text{\ \ } = \text{\ \ } d_X(\ep,y_1)+\maxr+1\\
& \le & t^e(d_X(\ep,\pi_{\dd_w}(v_t)))+\maxr+1 \\
&\le & t^e(\lceil \td_X(\ep,\pi_{\dd_w}(\ps_w(p,t)))\rceil+\maxr+1)+\maxr+1.
\end{eqnarray*}
Then $G$ satisfies an extrinsic \tfi\ for the
function $n \ra t^e(\lceil n \rceil+2K+3)+2K+3$.
\end{proof}

We highlight two special cases in which the hypothesis
of Proposition~\ref{prop:combing}, that each set
$S_n$ is finite, is satisfied.
The first is the case in which the set of normal
forms is prefix-closed.  For this case, the functions
$t^i=\kti$ and $t^e=\kte$ are the functions defined
in Section~\ref{sec:rnf}, and so we have the following.

\begin{corollary}
If $G$ has a prefix-closed set of normal forms that satisfies
a $K$-fellow traveler property, then $G$ admits an
intrinsic \tfi\ for the function 
$f^i(n)=\kti(\lceil n \rceil+2K+3)+2K+3$
and an extrinsic \tfi\ for the function
$f^e(n)=\kte(\lceil n \rceil+2K+3)+2K+3$.
\end{corollary}

The second is the case in which the set of  normal forms
is quasi-geodesic; that is, there are constants
$\lambda,\lambda' \ge 1$ such that every word in this 
set is a   $(\lambda,\lambda')$-quasi-geodesic.
For a group $G$ with generators $A$
and Cayley graph $\ga$,
a word $y \in A^*$ is a {\em $(\lambda,\lambda')$-quasi-geodesic} 
if whenever $y=y_1y_2y_3$, then 
$l(y_2) \le \lambda d_\ga(\ep,y_2)+\lambda'$.
Actually, we will only need a slightly weaker property,
that this inequality holds whenever $y_2$ is
a prefix of $y$ (i.e., when $y_1=1$).
In this case, the set $S_n$ is a subset of
the finite set $\cup_{i=0}^{\lambda n+\lambda'} A^i$
of words of length at most $\lambda n+\lambda'$.
Then $t^e(n) \le t^i(n) \le \lambda n+\lambda'$ for all $n$.
Putting these results together yields the following.

\begin{corollary}\label{cor:combable}
If a finitely generated group $G$ admits a 
quasi-geodesic language of simple word normal forms
satisfying a $K$-fellow traveler property, then $G$ satisfies linear
intrinsic and extrinsic \tfs.
\end{corollary}


\section{Quasi-isometry invariance for \tfs}\label{sec:qiinv}


In this section we show 
that, as with the \dms~\cite{bridsonriley},~\cite{gersten}, 
\tfs\  are also
quasi-isometry invariants, up to Lipschitz equivalence
of functions (and in the intrinsic case, up to 
sufficiently large set of 
defining relations).  In the extrinsic case,
this follows from Proposition~\ref{prop:etistc}
and the proof of Theorem~\cite[Theorem~A]{hmeiermeastame},
but with a slightly different definition of
coarse distance.  We include the details for
both  here, to illustrate the difference between
the intrinsic and extrinsic cases.

\begin{theorem}\label{thm:itisqi}
Suppose that $(G,\pp)$ and $(H,\pp')$ 
are quasi-isometric groups with
finite presentations.
If $(G,\pp)$ satisfies an extrinsic 
\tfi\  with respect to $f$, then $(H,\pp')$
satisfies an extrinsic
\tfi\  with respect to a function
that is Lipschitz equivalent to $f$.
If $(G,\pp)$ satisfies an intrinsic 
\tfi\  with respect to $f$, then after adding
all relators of length up to a sufficiently 
large constant to the 
presentation $\pp'$, 
the pair $(H,\pp')$
satisfies an intrinsic
\tfi\  with respect to a function
that is Lipschitz equivalent to $f$.
\end{theorem}

\begin{proof}
Write the finite presentations $\pp=\langle A \mid R \rangle$
and $\pp'=\langle B \mid S\rangle$; as usual
we assume that these presentations are symmetrized.
Let $X$ be the 2-dimensional Cayley complex for the
pair $(G,\pp)$, and let $Y$ be the Cayley complex associated
to $(H,\pp')$.  
Let $d_X$, $d_Y$ be the path metrics in $X$ and $Y$
(and hence also the word metrics in $G$ and $H$ with respect to the 
generating sets $A$ and $B$), respectively.
  
Quasi-isometry of these groups means that there are
functions $\phi:G \ra H$ and $\theta:H \ra G$ 
and a constant $k>1$ such that for all $g_1,g_2 \in G$
and $h_1,h_2 \in H$, we have
\begin{enumerate}
\item $\frac{1}{k}d_X(g_1,g_2)-k \le d_Y(\phi(g_1),\phi(g_2)) 
          \le kd_X(g_1,g_2)+k$
\item $\frac{1}{k}d_Y(h_1,h_2)-k \le d_X(\theta(h_1),\theta(h_2)) 
              \le kd_Y(h_1,h_2)+k$
\item $d_X(g_1,\theta \circ \phi(g_1)) \le k$
\item $d_Y(h_1,\phi \circ \theta(h_1)) \le k$
\end{enumerate}
By possibly increasing the constant $k$, we may also
assume that $k>2$ and that
$\phi(\ep_G)=\ep_H$ and $\theta(\ep_H)=\ep_G$,
where $\ep_G$ and $\ep_H$ are the identity elements of 
the groups $G$ and $H$, respectively.

We extend the functions $\phi$ and $\theta$ to functions
$\tp:G \times A^* \ra B^*$ and $\tth:H \times B^* \ra A^*$
as follows.  
Let $\tilde A \subset A$ be a subset containing 
exactly one element for each inverse pair $a,a^{-1} \in A$. 
Given a pair $(g,a) \in G \times \tilde A$, using property (1) above
we let
$\tp(g,a)$ be (a choice of) a nonempty word 
of length at most $2k$ labeling a 
path in the Cayley graph $Y^1$ from the vertex $\phi(g)$ to the
vertex $\phi(ga)$ (in the case that $\phi(g)=\phi(ga)$,
we can choose $\tp(g,a)$ to be the nonempty word 
$bb^{-1}$ for some choice of $b \in B$).
We also define $\tp(g,a^{-1}):=\tp(ga^{-1},a)^{-1}$.
Then for any $w = a_1 \cdots a_m$
with each $a_i \in A$, define $\tp(g,w)$ to be the
concatenation
$\tp(g,w) := \tp(g,a_1) \cdots \tp(ga_1 \cdots a_{m-1},a_m)$.
Note that for $w \in A^*$:

(5) the word lengths satisfy
$l(w) \le l(\tp(g,w)) \le 2kl(w)$, and

(6) the word $\tp(\ep_G,w)$ 
represents the element $\phi(w)$ in $H$.

\noindent The function $\tth$ is defined analogously.

Using Propositions~\ref{prop:htpydomain} and~\ref{prop:etistc}, we will
prove this theorem using relaxed \tfs\ 
via disk homotopies.
For the group $G$ with presentation $\pp$, fix a collection 
${\cld} = \{(\dd_w,\ph_w) \mid w \in A^*, w=_G \ep_G\}$
of van Kampen diagrams and associated disk homotopies,
such that all of the $\ph_w$ are intrinsically $f^i$-tame or 
all $\ph_w$ are extrinsically $f^e$-tame, where $f^i,f^e:\nn \ra \nn$
are nondecreasing functions.

\smallskip

\noindent {\it Case A.  Suppose that $G$ is a finite group.}

In this case, $H$ is also finite.
Let ${\mathcal F}$ be a (finite) collection 
of van Kampen diagrams over $\pp'$, one for each
word over $B$ of length at most $|H|$ that
represents $\ep_H$.
Now given any word $u$ over $B$ with $u=_H \ep_H$,
we will construct a van Kampen diagram for $u$
with intrinsic diameter at most 
$|H|+\max\{idiam(\dd) \mid \dd \in {\mathcal F}\}$,
as follows.  
Start with a planar 1-complex that
is a line segment consisting of an edge path
labeled by the word $u$ starting at a basepoint $*$; 
that is, we start with a van Kampen diagram for
the word $uu^{-1}$.
Write $u=u_1'u_1''u_1'''$ where $u_1''=_H \ep_H$
and no proper prefix of $u_1'u_1''$ contains a subword 
that represents $\ep_H$.
Note that $l(u_1'u_1'') \le |H|$.
We identify the vertices in the
van Kampen diagram at the start and end of the
boundary path labeled $u_1''$,
and fill in this loop with the van Kampen
diagram from ${\mathcal F}$ for this word.
We now have a van Kampen diagram for the
word $uu_1^{-1}$ where 
$u_1:=u_1'u_1'''$.
We then begin again, and write $u_1=u_2'u_2''u_2'''$
where $u_2''=_H \ep_H$ and no proper prefix of
$u_2'u_2''$ contains a subword representing the
identity.  Again we identify the vertices at
the start and end of the word $u_2''$ in the
boundary of the diagram, and fill in this
loop with the diagram from ${\mathcal F}$ for
this word, to obtain a van Kampen diagram 
for the word $uu_2^{-1}$ where $u_2:=u_2'u_2'''$.
Repeating this process, since at each step
the length of $u_i$ strictly decreases, we
eventually obtain a word $u_k=u_k''$.  Identifying
the endpoints of this word and filling in
the resulting loop with the van Kampen diagram
in ${\mathcal F}$ yields a van Kampen diagram $\dd_u'$
for $u$.  
See Figure~\ref{fig:finitegp} for an illustration of
this procedure.
\begin{figure}
\begin{center}
\includegraphics[width=3.4in,height=1.4in]{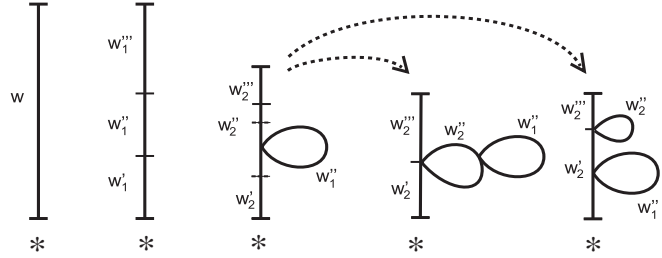}
\caption{Building $\dd_u'$ in the finite group case}\label{fig:finitegp}
\end{center}
\end{figure}
At each step, the maximum distance
from the basepoint $*$ to any vertex in
the van Kampen diagram included from ${\mathcal F}$
is at most $|H|+\max\{idiam(\dd) \mid \dd \in {\mathcal F}\}$,
because this subdiagram is attached at 
the endpoint of a path starting at $*$ and
labeled by the word $u_i'$ of length
less than $|H|$.
At the end of this process, every vertex of the
final diagram lies on one of these subdiagrams.
Hence we obtain the required intrinsic diameter
bound.

Let $\Phi_u'$ be any disk homotopy of
the diagram $\dd_u'$.  Then the 
collection $\{(\dd_u',\Phi_u')\}$
of van Kampen diagrams and disk homotopies $H$ over $\pp'$
satisfies the property that each
homotopy $\Phi_u'$ is intrinsically $f$-tame
for the constant function 
$f(n) \equiv |H|+\max\{idiam(\dd) \mid \dd \in {\mathcal F}\}+\frac{1}{2}$,
since this constant is an upper bound for the
coarse distance from the basepoint to every point of $\dd_u'$.
Hence $H$ also satisfies an intrinsic relaxed \tfi\ 
with respect to the function $f^i+f$, which is
Lipschitz equivalent to $f^i$.  

Similarly, since the extrinsic diameter of
every van Kampen diagram in this collection (or, 
indeed, any other van Kampen diagram) is at
most $|H|$, the pair $(H,\pp')$ satisfies an
extrinsic \tfi\ for the constant function $|H|+\frac{1}{2}$,
and so also satisfies an extrinsic relaxed \tfi\ for
the function $n \ra f^e(n)+|H|+\frac{1}{2}$.

\smallskip

\noindent {\it Case B.  Suppose that $G$ is an infinite group.}

The group $H$ is also infinite, and so
the functions $f^i$ and $f^e$ must grow at
least linearly, in this case.
In particular, we have
 $f^i(n) \ge n-\maxr-1$ and $f^e(n) \ge n-\maxr-1$
for all $n \in \nn$, where $\maxr=\max\{l(r) \mid r \in R\}$
is the maximum length of a relator in the presentation $\pp$.

Now suppose that $u'$ is any word in $B^*$
with $u' =_H \ep_H$.  We will construct
a van Kampen diagram for $u'$,
following the method of
\cite[Theorem 9.1]{bridsonriley}.  At each of the four 
successive steps,
we obtain a van Kampen diagram for a specific word; 
we will also keep track of homotopies and analyze
their tameness, in order to
finish with a diagram and disk homotopy 
for $u'$.

\smallskip

{\em Step I.  For $u:=\tth(\ep_H,u')\in A^*$:} 
Note (6) implies that the word 
$u=_G \theta(u')=_G\theta(\ep_H)=_G\ep_G$, and
so the collection ${\cld}$ contains a
van Kampen diagram $\dd_u$ for $u$ and an associated
disk homotopy 
$\ph_u:C_{l(u)} \times [0,1] \ra \dd_u$.
Note that $\ph_u$ is intrinsically $f_1^i:=f^i$-tame
or extrinsically $f_1^e:=f^e$-tame.

\smallskip

{\em Step II.  For $z'':=\tp(\ep_G,u)=\tp(\ep_G,\tth(\ep_H,u')) \in B^*$:}  
We build a finite, planar, contractible, 
combinatorial 2-complex $\Omega$ from
$\dd_u$ as follows.
As usual, let $\pi_{\dd_u}: \dd_u \ra X$
be the canonical map taking the basepoint $*$ of
$\dd_u$ to $\ep_G$.
Given any edge $e$ in $\dd_u$, choose
a direction, and hence a label $a_e$, for $e$,
and let $v_1$ be the initial
vertex of $e$.
Replace $e$ with a directed edge path
$\hat e$ labeled by the (nonempty) word
$\tp(\pi_{\dd_u}(v_1),a_e)$.  Repeating 
this for every edge of the complex $\dd_u$
results in the 2-complex $\Omega$.

Note that 
$\Omega$ is a 
van Kampen diagram for 
the word $z'':=\tp(\ep_G,u)=\tp(\ep_G,\tth(\ep_H,u')) \in B^*$
with respect to the presentation
$\pp''=\langle B \mid S \cup S'' \rangle$ of $H$,
where $S''$ is the set of all nonempty words over
$B$ of length at most 
$2k \maxr$ that represent $\ep_H$.
Let $Y''$ be the Cayley complex for $\pp''$
and as usual, let $\pi_\Omega:\Omega \ra Y''$ be the canonical map.

Using the fact that the only difference between $\dd_u$ and $\Omega$
is a replacement of edges by edge paths, we 
define $\alpha:\dd_u \ra \Omega$ to be the continuous map
taking each vertex and each interior point of
a 2-cell of $\dd_u$ to the same point of $\Omega$,
and taking each edge $e$ to the corresponding edge path $\hat e$.

Writing $u=a_1 \cdots a_m$ with each $a_i \in A$,
then $z''=c_{1,1} \cdots c_{1,j_1} \cdots c_{m,1} \cdots c_{m,j_m}$
where each $c_{i,j} \in B$ and $c_{i,1} \cdots c_{i,j_i}$ is
the nonempty word labeling the edge path $\widehat{e_i}$
of $\bo \Omega$ that is the image under $\alpha$ of the $i$-th edge
of the boundary path of $\dd_u$.
Recall that $C_{l(u)}$ is the circle $S^1$ with a 
1-complex structure of $l(u)$ vertices and edges.
Let the 1-complex $C_{l(z'')}$ be a refinement of the
complex $C_{l(u)}$, so that the $i$-th edge of 
$C_{l(u)}$ is replaced by $j_i \ge 1$ edges for each $i$, and
let $\hat \alpha:C_{l(z'')} \ra C_{l(u)}$ be the identity
on the underlying circle.
Finally, define the map 
$\omega:C_{l(z'')} \times [0,1] \ra \Omega$ by
$\omega:=\alpha \circ \ph_u \circ (\hat \alpha \times id_{[0,1]})$.
This map $\omega$
satisfies conditions (d1)-(d2) of the definition
of disk homotopy.

Next we analyze the intrinsic tameness of $\omega$.  
Again since in this
step we have only replaced edges by nonempty edge
paths of length at most $2k$, for each vertex $v$ in 
$\dd_u$ we have
$
\td_{\dd_u}(*,v) 
\le \td_{\Omega}(*,\alpha(v)) 
\le 2k\td_{\dd_u}(*,v)~.
$
For a point $q$ in the interior of an edge
of $\dd_u$, let $v$ be a vertex
in the same closed cell; then 
$|\td_{\dd_u}(*,q)-\td_{\dd_u}(*,v)| <1$
and $|\td_{\Omega}(*,\alpha(q))-\td_{\Omega}(*,\alpha(v))| < 2k$.
For a point $q$ in the interior of a 2-cell of $\dd_u$,
let $v$ be a vertex in the closure of this cell with 
$\td_{\dd_u}(*,v) \le \td_{\dd_u}(*,q)+1$.  Then $\alpha(v)$
is a vertex in the closure of the open 2-cell of $\Omega$
containing $\alpha(q)$, and the boundary path of this
cell has length at most $2k\maxr$.
That is, 
$|\td_{\dd_u}(*,q)-\td_{\dd_u}(*,v)| <1$
and $|\td_{\Omega}(*,\alpha(q))-\td_{\Omega}(*,\alpha(v))| < 2k\maxr$.
Thus for all $q\in \dd_u$, we have $\td_{\dd_u}(*,q) 
\le \td_{\Omega}(*,\alpha(q)) \le 2k \td_{\dd_u}(*,q) +4k+2k\maxr$.

Now suppose that $p$ is any point in $C_{l(z'')}$ and 
$0 \le s < t \le 1$.  Combining the inequalities above
with the $f_1^i$-tame property of $\ph_u$ and the fact
that $f_1^i$ is nondecreasing yields
\begin{eqnarray*}
\td_\Omega(*,\omega(p,s)) &=& \td_\Omega(*,\alpha(\ph_u(\hat \alpha(p),s))\\
         &\le& 2k \td_{\dd_u}(*,\ph_u(\hat \alpha(p),s)) +4k+2k\maxr\\
  & \le & 2k f_1^i(\td_{\dd_u}(*,\ph_u(\hat \alpha(p),t)))+4k+2k\maxr \\
  &\le & 2k f_1^i(\td_{\Omega}(*,\alpha(\ph_u(\hat \alpha(p),t))))+4k+2k\maxr\\
      & =& 2k f_1^i(\td_{\Omega}(*,\omega(p,t))+4k+2k\maxr~.
\end{eqnarray*}
Hence $\omega$ is intrinsically $f_2^i$-tame for the nondecreasing function 
$f_2^i(n):=2k f_1^i(n)+4k+2k\maxr$.

In the last part of Step II, we analyze the extrinsic tameness
of $\omega$.  
For any vertex $v$ in $\dd_u$,  let $w_v$ 
be a word labeling
a path in $\dd_u$ from $*$ to $v$.  
Using note (6) above,
we have 
$\phi(\pi_{\dd_u}(v)) =_H 
\phi(w_v) =_H \tp(\ep_G,w_v)=
\pi_\Omega(\alpha(v))$, by our
construction of $\Omega$. 
The quasi-isometry property (1) then gives
\[
\frac{1}{k}d_X(\ep_G,\pi_{\dd_u}(v))-k 
  \le d_Y(\ep_H,\phi(\pi_{\dd_u}(v))) 
    = d_Y(\ep_H,\pi_\Omega(\alpha(v)))
  \le kd_X(\ep_G,\pi_{\dd_u}(v))+k~.
\]
Since the generating sets of the presentations $\pp'$ and
$\pp''$ of $H$ are the same, the Cayley graphs 
and their path metrics $d_Y=d_{Y''}$ are also the same.  
As in the intrinsic case above,
for a point $q$ in the interior of an edge or 2-cell of $\dd_u$,
there is a vertex $v$ in the same closed cell with
$|\td_X(\ep_G,\pi_{\dd_u}(q))-\td_X(\ep_G,\pi_{\dd_u}(v))|<1$ and
$|\td_{Y''}(\ep_G,\pi_\Omega(\alpha(q)))-
         \td_{Y''}(\ep_G,\pi_\Omega(\alpha(v)))|<2k(\maxr+1)$.
Then for all $q \in \dd_u$, we have
\begin{eqnarray*}
\td_X(\ep_G,\pi_{\dd_u}(q))
&\le&  k\td_{Y''}(\ep_H,\pi_\Omega(\alpha(q)))+2k^2\maxr+3k^2+1~,
\text{ and }\\
\td_{Y''}(\ep_H,\pi_\Omega(\alpha(q)))
&\le&  k\td_X(\ep_G,\pi_{\dd_u}(q))+4k+2k\maxr~.
\end{eqnarray*}

Now suppose that $p$ is any point in $C_{l(z'')}$ and 
$0 \le s < t \le 1$. Then
\begin{eqnarray*}
\td_{Y''}(\ep_H,\pi_\Omega(\omega(p,s))) 
   &=& \td_{Y''}(\ep_H,\pi_\Omega(\alpha(\ph_u(\hat \alpha(p),s)))) \\
     & \le&  k \td_{X}(\ep_G,\pi_{\dd_u}(\ph_u(\hat \alpha(p),s))) 
                +4k+2k\maxr \\
  & \le & k f_1^e(\td_{X}(\ep_G,\pi_{\dd_u}(\ph_u(\hat \alpha(p),t))))
                  +4k+2k\maxr\\
 &\le & k f_1^e(k\td_{Y''}(\ep_H,\pi_\Omega(\alpha(\ph_u(\hat \alpha(p),t))))
                     +2k^2\maxr+3k^2+1)+4k+2k\maxr\\
       &=& k f_1^e(k\td_{Y''}(\ep_H,\pi_\Omega(\omega(p,t))+2k^2\maxr+3k^2+1)
              +4k+2k\maxr~.
\end{eqnarray*}
Hence $\omega$ is extrinsically $f_2^e$-tame for the nondecreasing function 
$f_2^e(n):=k f_1^e(kn+2k^2\maxr+3k^2+1)+4k+2k\maxr$.

\smallskip

{\em Step III. For $u'$ over $\pp'''$:}  In this step we
construct another finite, planar, contractible,
and combinatorial 2-complex $\Lambda_{u'}$ starting from $\Omega$,
by adding a ``collar'' around the outside boundary.
Write the word $u'=b_1 \cdots b_n$ with each $b_i \in B$.
For each $1 \le i \le n-1$, let $w_i$ be
a word labeling a geodesic edge path in
$Y$ from $\phi(\theta(b_1 \cdots b_i))$
to $b_1 \cdots b_i$;
the quasi-isometry inequality in (3) above
implies that the length of $w_i$ is at most $k$. 
We add to $\Lambda_{u'}$ a vertex $x_i$ and the vertices
and edges of a directed edge path $p_i$
labeled by $w_i$
from the vertex ${v_i}$ to $x_i$, 
where $v_i$ is the 
vertex in $\bo \dd_u$ at the end of the path
$\tp(\ep_G,\tth(e_H,b_1 \cdots b_i))$ starting at the basepoint.
Note that if $w_i$ is the empty word, we identify
$x_i$ with the vertex ${v_i}$; the path $p_i$
is a constant path at this vertex.
Then $*={v_0}=x_0=x_n$
(and $p_0$ and $p_n$ are the constant path at
this vertex); let this vertex
be the basepoint of $\Lambda_{u'}$.

Next we add to $\Lambda_{u'}$ a directed edge 
$\check e_i$ labeled
by $b_i$ from the vertex $x_{i-1}$ to the vertex $x_i$.
The path $q_i$ from $v_{i-1}$ to $v_i$
along the boundary of 
the subcomplex $\Omega$ is labeled by the nonempty word
$z_i:=\tp(\theta(b_1 \cdots b_{i-1}),\tth(b_1 \cdots b_{i-1},b_i))$.
If both of the paths $p_{i-1},p_i$  are constant
and the label of path $q_i$ is the single letter $b_i$,
then we identify the edge $\check e_i$ with the path $q_i$.
Otherwise, we 
attach a 2-cell $\check \sigma_i$ along the edge circuit
following the edge path starting at $\hat v_{i-1}$ 
that traverses the path
$q_i$,
the path $p_i$, the reverse of the edge $\check e_i$, 
and finally the reverse of
the path $p_{i-1}$. 
See Figure~\ref{fig:qicollar} for a picture of the resulting
diagram.
\begin{figure}
\begin{center}
\includegraphics[width=3.4in,height=1.4in]{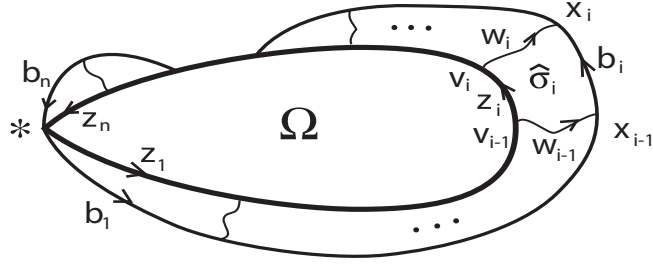}
\caption{The van Kampen diagram $\Lambda_{u'}$}\label{fig:qicollar}
\end{center}
\end{figure}

Now the complex $\Lambda_{u'}$ is a van Kampen
diagram for the original word $u'$,
with respect to the presentation
$\pp'''=\langle B \mid S \cup S''' \rangle$ of $H$,
where $S'''$ is the set of all nonempty words in
$B^*$ of length at most 
$\maxr''':=2k\maxr+(2k)^2+2k+1$ that
represent $\ep_H$.
(Note that the presentation $\langle B \mid S'''\rangle$
also presents $H$, and $\Lambda_{u'}$ is also
a diagram over this more restricted presentation.)
Let 
$Y'''$ be the corresponding Cayley complex.

We define a disk homotopy 
$\lambda_{u'}:C_{l(u')} \times [0,1] \ra \Lambda_{u'}$
by extending the paths of the
homotopy $\omega$ on the subcomplex $\Omega$
as follows.
First we let the cell complex $C_{l(u')}$ be the complex
$C_{l(z''))}$ with each subpath in $C_{l(z''))}$ mapping
to a path $q_i$ in $\bo \Omega$ replaced by a single edge.
From our definitions of $\tp$ and $\tth$, each $q_i$ path
is labeled by a nonempty word, and so
$C_{l(z''))}$ is a refinement of the complex structure
$C_{l(u'))}$ on $S^1$, and we let 
$\hat \beta:C_{l(u'))} \ra C_{l(z''))}$ be the identity
on the underlying circle.
Next define a homotopy 
$\tilde \lambda:C_{l(z'')} \times [0,1] \ra \Lambda_{u'}$
as follows.
For each $1 \le i \le n$, let $\tilde v_i$ be the point in $S^1$
mapped by $\omega$ to $v_i$.
Define $\tilde \lambda(\tilde v_i,t):=\omega(\tilde v_i,2t)$
for $t \in [0,\frac{1}{2}]$, and let $\tilde \lambda(\tilde v_i,t)$
for $t \in [\frac{1}{2},1]$ be a constant speed path
along $p_i$ from $v_i$ to $x_i$.
On the interior of the edge $\tilde e_i$ from $\tilde v_{i-1}$
to $\tilde v_i$, 
define
the homotopy $\tilde \lambda|_{\tilde e_i \times [0,\frac{1}{2}]}$
to follow $\omega|_{\tilde e_i \times [0,1]}$ at double speed,
and let $\tilde \lambda|_{\tilde e_i \times [\frac{1}{2},1]}$
go through the 2-cell $\check \sigma_i$ (or, if there is 
no such cell, let this portion of $\tilde \lambda$ be constant)
from $q_i$ to $\check e_i$.
Finally, we define the homotopy 
$\lambda_{u'}:C_{l(u'))} \times [0,1] \ra \Lambda_{u'}$ by 
$\lambda_{u'}:=\tilde \lambda \circ (\hat \beta \times id_{[0,1]})$.
This map $\lambda_{u'}$ is a disk
homotopy for the diagram $\Lambda_{u'}$.


Next we analyze the intrinsic tameness of $\lambda_{u'}$.
Since $\Omega$ is a subdiagram of $\Lambda_{u'}$, 
for any vertex $v$ in $\Omega$, we have 
$d_{\Lambda_{u'}}(*,v) \le d_\Omega(*,v)$.  
Given any edge path $\beta$ in $\Lambda_{u'}$
from $\ep$ to $q$ that is not completely contained
in the subdiagram $\Omega$, the subpaths of
$\beta$ lying in the ``collar'' can be replaced
by paths along $\bo \Omega$ of length at most
a factor of $4k^2$ longer.  Then
$d_\Omega(*,v) \le 4k^2d_{\Lambda_{u'}}(*,v)$. 
Hence for any point $q \in \Omega$, we have
$\td_{\Lambda_{u'}}(*,q) \le \td_\Omega(*,q) 
\le 4k^2\td_{\Lambda_{u'}}(*,q)+4k^2+1+\maxr'''$.

Now suppose that $p$ is any point of $C_{l(u')}$ and
$0 \le s < t \le 1$.
If $t \le \frac{1}{2}$, then the path 
$\lambda_{u'}(p,\cdot)$ on $[0,t]$ is a reparametrization of
$\omega(p,\cdot)$, and so Step II, the 
fact that $f_2^i$ is nondecreasing,  and the inequalities above
give 
\begin{eqnarray*}
\td_{\Lambda_{u'}}(*,\lambda_{u'}(p,s)) 
&\le& \td_{\Omega}(*,\lambda_{u'}(p,s)) \\
&\le& f_2^i(\td_{\Omega}(*,\lambda_{u'}(p,t)))\\
&\le& f_2^i(4k^2\td_{\Lambda_{u'}}(*,\lambda_{u'}(p,t))+4k^2+1+\maxr''')
\end{eqnarray*}
If $t > \frac{1}{2}$ and $s \le \frac{1}{2}$, then we have 
$ \td_{\Lambda_{u'}}(*,\lambda_{u'}(p,s)) \le 
f_2^i(4k^2\td_{\Lambda_{u'}}(*,\lambda_{u'}(p,\frac{1}{2}))+4k^2+1+\maxr''')$
and
$|\td_{\Lambda_{u'}}(*,\lambda_{u'}(p,t))-
     \td_{\Lambda_{u'}}(*,\lambda_{u'}(p,\frac{1}{2}))|<\maxr'''+1$, so 
$$ \td_{\Lambda_{u'}}(*,\lambda_{u'}(p,s)) \le 
f_2^i(4k^2(\td_{\Lambda_{u'}}(*,\lambda_{u'}(p,t))+\maxr'''+1)+
          4k^2+1+\maxr''').$$
If $s>\frac{1}{2}$, then 
\begin{eqnarray*}
\td_{\Lambda_{u'}}(*,\lambda_{u'}(p,s)) &\le& 
\td_{\Lambda_{u'}}(*,\lambda_{u'}(p,t)) + \maxr'''+1\\
&\le& f_2^i(\td_{\Lambda_{u'}}(*,\lambda_{u'}(p,t)))+\maxr'''+1,
\end{eqnarray*} 
where the latter inequality follows from the
fact that $n \le f^i(n) +\maxr+1 \le f_2^i(n)$ for this
infinite group case.
Then $\lambda_{u'}$ is intrinsically $f_3^i$-tame for the function
$f_3^i(n):=f_2^i(4k^2n+8k^2+1+(4k^2+1)\maxr''')$.

We note that we have now
completed the proof of Theorem~\ref{thm:itisqi} in
the intrinsic case:  
The collection
$\{(\Lambda_{u'},\lambda_{u'}) \mid u' \in B^*, u'=_H\ep_H\}$
of van Kampen diagrams and disk homotopies over the
presentation $\pp'''=\langle B \mid S''' \rangle$
implies an intrinsic relaxed \tfi\ for the
function $f_3^i$, which is Lipschitz equivalent to $f^i$.

The analysis of the extrinsic tameness in 
this step is simplified by the fact that 
for all $q \in \Omega$, we have
$\td_{Y''}(\ep_H,\pi_\Omega(q))=
   \td_{Y'''}(\ep_H,\pi_{\Lambda_{u'}}(q))$, 
since the 1-skeleta of $Y''$ and $Y'''$ are
determined by the generating sets of the 
presentations $\pp''$ and $\pp'''$, which are
the same.  A similar argument to those
above shows that $\lambda_{u'}$ is extrinsically 
$f_3^e$-tame for the function
$f_3^e(n):=f_2^e(n+\maxr'''+1)+\maxr+1$.

{\em Step IV.  For $u'$ over $\pp'$:}  
Finally, we turn to building a van Kampen
diagram $\dd_{u'}'$ for $u'$ over the original
presentation $\pp'$.
For each nonempty word $w$ over $B$ of length at
most $\maxr'''$
satisfying $w=_H \ep_H$, let $\dd_w'$ be a
fixed choice of van Kampen diagram for $w$
with respect to the presentation $\pp'$ of $H$, and 
let ${\mathcal F}$ be the (finite) collection of these
diagrams.
A diagram $\dd_{u'}'$ over the presentation $\pp'$ is built by
replacing 2-cells of $\Lambda_{u'}$,
proceeding through the 2-cells of $\Lambda_{u'}$
one at a time.  Let $\tau$ be a 
2-cell of $\Lambda_{u'}$, and 
let $*_\tau$ be a choice of basepoint vertex
in $\bo \tau$.
Let $x$ be the word labeling
the path $\bo \tau$ starting at $*_\tau$ and 
reading counterclockwise.  Since
$l(x) \le L$, there is an associated van Kampen diagram
$\dd_\tau'=\dd_{x}'$ in the collection ${\mathcal F}$.
Note that although $\Lambda_{u'}$
is a combinatorial 2-complex, and so the cell $\tau$
is a polygon, the boundary label $x$
may not be freely or cyclically reduced.  The 
van Kampen diagram $\dd_x'$ may not be a polygon, but
instead a collection of polygons connected by edge
paths, and possibly with edge path ``tendrils''.
We replace the 2-cell $\tau$ with a copy $\dd_\tau'$
of the van Kampen
diagram $\dd_x'$, identifying the boundary
edge labels as needed, obtaining another
planar diagram.  Repeating this
for each 2-cell of of the resulting complex at
each step, results in the
van Kampen diagram $\dd_{u'}'$ for $u'$ with
respect to $\pp'$.  

From the process of constructing $\dd_{u'}'$
from $\Lambda$, for each 2-cell $\tau$ there
is a continuous onto map $\tau \ra \dd_\tau'$
preserving the boundary edge path labeling,
and so there is an induced continuous surjection
$\gamma: \Lambda_{u'} \ra \dd_{u'}'$.  Note that the
boundary edge paths of $\Lambda_{u'}$ and $\dd_{u'}'$
are the same.  
Then the composition
$\ph_{u'}':=\gamma \circ \lambda_{u'}:C_{l(u')} \times [0,1] 
\ra \dd_{u'}'$ is a disk homotopy.

To analyze the extrinsic tameness, we first note
that for all points $\hat q \in \Lambda_{u'}^1$, 
the image $\pi_{\Lambda_{u'}}(\hat q)$ in $Y'''$ and 
the image $\pi_{\dd_{u'}'}(\gamma(\hat q))$ in $Y$ are
the same point in the 1-skeleta $Y^1=(Y''')^1$,
and so $\td_{Y'''}(\ep_H,\pi_{\Lambda_{u'}}(\hat q))=
\td_Y(\pi_{\dd_{u'}'}(\gamma(\hat q)))$.
Let $M:=2\max\{\td_\dd(*,r) \mid \dd \in {\mathcal F}, r \in \dd\}$.

Suppose that $p$ is any point in $C_{l(u')}$ and
$0 \le s<t \le 1$.  If $\lambda_{u'}(p,s) \in \Lambda_{u'}^1$,
then define $s':=s$; otherwise, let $0 \le s'<s$ satisfy
$\lambda_{u'}(p,s') \in \Lambda_{u'}^1$ and 
$\lambda_{u'}(p,(s',s])$ is a subset of a single open
2-cell of $\Lambda_{u'}$.  Similarly, if 
$\lambda_{u'}(p,t) \in \Lambda_{u'}^1$,
then define $t':=t$, and otherwise, let $t < t' \le 1$ satisfy
$\lambda_{u'}(p,t') \in \Lambda_{u'}^1$ and 
$\lambda_{u'}(p,[t,t'))$ is a subset of a single open
2-cell of $\Lambda_{u'}$. 
Then 
\begin{eqnarray*}
\td_Y(\ep_H,\pi_{\dd_{u'}'}(\ph_{u'}(p,s)))
& = &  \td_{Y}(\ep_H,\pi_{\dd_{u'}'}(\gamma(\lambda_{u'}(p,s)))) \\
&\le&  \td_{Y}(\ep_H,\pi_{\dd_{u'}'}(\gamma(\lambda_{u'}(p,s')))) +M \\
&= &   \td_{Y'''}(\ep_H,\pi_{\Lambda_{u'}}(\lambda_{u'}(p,s'))) +M \\
&\le&  f_3^e(\td_{Y'''}(\ep_H,\pi_{\Lambda_{u'}}(\lambda_{u'}(p,t')))) +M \\
&=&    f_3^e(\td_{Y}(\ep_H,\pi_{\dd_{u'}'}(\gamma(\lambda_{u'}(p,t'))))) +M \\
&\le&  f_3^e(\td_{Y}(\ep_H,\pi_{\dd_{u'}'}(\gamma(\lambda_{u'}(p,t))))+M) +M.
\end{eqnarray*}
Therefore
$\ph_{u'}'$ is extrinsically $f_4^e$-tame, for the
function $f_4(n):=f_4(n+M)+M$.
Since the functions $f_j^e$ and $f_{j+1}^e$ are 
Lipschitz equivalent for all $j$, then $f_4^e$ is
Lipschitz equivalent to $f^e$.  

Now the collection
$\{(\dd_{u'}',\ph_{u'}') \mid u' \in B^*, u'=_H\ep_H\}$
of van Kampen diagrams and disk homotopies yields
an extrinsic relaxed \tfi\ for the pair $(H,\pp')$
with respect to a function that is Lipschitz equivalent to $f^e$.
\end{proof}

The obstruction to applying Step IV of the above
proof in the intrinsic case stems from the fact that
the map 
$\gamma:\Lambda_{u'} \ra \dd_{u'}'$ behaves well with
respect to extrinsic coarse distance, but 
may not behave
well with respect to intrinsic coarse distance.
The latter results because
the replacement of a 2-cell $\tau$ of $\Lambda_{u'}$
with a van Kampen diagram $\dd_\tau'$ can result in
the identification of vertices of $\Lambda_{u'}$.  




\end{document}